\documentclass{amsart}     
\usepackage{graphicx}
\usepackage{amssymb,amsmath,stmaryrd,mathrsfs,amsthm}
\usepackage[all,2cell]{xy}
\usepackage[neveradjust]{paralist}
\usepackage{enumitem}
\usepackage[noadjust]{cite}
\usepackage{hyperref}

\usepackage{subfigure}
\usepackage{array}
\usepackage{mathtools}
\usepackage{tikz}

\newcommand{\sB}{\ensuremath{\mathscr{B}}}
\newcommand{\sC}{\ensuremath{\mathscr{C}}}

\newcommand{\sL}{\ensuremath{\mathscr{L}}}

\newcommand{\sR}{\ensuremath{\mathscr{R}}}

\newcommand{\bbQ}{\ensuremath{\mathbb{Q}}}

\newcommand{\bbZ}{\ensuremath{\mathbb{Z}}}

\newcommand{\bC}{\ensuremath{\mathbf{C}}}
\newcommand{\bD}{\ensuremath{\mathbf{D}}}

\newcommand{\bS}{\ensuremath{\mathbf{S}}}
\newcommand{\bT}{\ensuremath{\mathbf{T}}}

\newcommand{\bZ}{\ensuremath{\mathbf{Z}}}

\newcommand{\fa}{\ensuremath{\mathfrak{a}}}

\newcommand{\fc}{\ensuremath{\mathfrak{c}}}

\newcommand{\fl}{\ensuremath{\mathfrak{l}}}

\newcommand{\fr}{\ensuremath{\mathfrak{r}}}
\newcommand{\fs}{\ensuremath{\mathfrak{s}}}

\newcommand{\ten}{\ensuremath{\otimes}}

\newcommand{\id}{\ensuremath{\operatorname{id}}}

\newcommand{\Ho}{\ensuremath{\operatorname{Ho}}}

\newcommand{\op}{\ensuremath{^{\mathit{op}}}}

\newdir{ >}{{}*!/-10pt/@{>}}
\newcommand{\iso}{\cong}

\newcommand{\too}[1][]{\ensuremath{\overset{#1}{\longrightarrow}}}
\newcommand{\oot}[1][]{\ensuremath{\overset{#1}{\longleftarrow}}}
\newcommand{\ot}{\ensuremath{\leftarrow}}

\newcommand{\toto}{\ensuremath{\rightrightarrows}}

\newcommand{\maps}{\colon}

\let\xto\xrightarrow

\def\xiso#1{\mathrel{\mathrlap{\smash{\xto[\smash{\raisebox{1.3mm}{$\scriptstyle\sim$}}]{#1}}}\hphantom{\xto{#1}}}}
\makeatletter
\def\slashedarrowfill@#1#2#3#4#5{%
  $\m@th\thickmuskip0mu\medmuskip\thickmuskip\thinmuskip\thickmuskip
   \relax#5#1\mkern-7mu%
   \cleaders\hbox{$#5\mkern-2mu#2\mkern-2mu$}\hfill
   \mathclap{#3}\mathclap{#2}%
   \cleaders\hbox{$#5\mkern-2mu#2\mkern-2mu$}\hfill
   \mkern-7mu#4$%
}
\def\rightslashedarrowfill@{%
  \slashedarrowfill@\relbar\relbar\mapstochar\rightarrow}
\newcommand\xslashedrightarrow[2][]{%
  \ext@arrow 0055{\rightslashedarrowfill@}{#1}{#2}}
\makeatother
\def\hto{\xslashedrightarrow{}}

\newcommand{\bc}[2]{\ensuremath{\prescript{}{#1}{#2}}}
\newcommand{\bcd}[2]{\ensuremath{{#2}_{#1}}}

\newcommand{\sh}[1]{{\ensuremath{\hspace{1mm}\makebox[-1mm]{$\langle$}\makebox[0mm]{$\langle$}\hspace{1mm}{#1}\makebox[1mm]{$\rangle$}\makebox[0mm]{$\rangle$}}}}
\newcommand{\scriptsh}[1]{{\ensuremath{\hspace{1mm}\makebox[-1mm]{$\scriptstyle\langle$}\makebox[0mm]{$\scriptstyle\langle$}\hspace{1mm}{#1}\makebox[1mm]{$\scriptstyle\rangle$}\makebox[0mm]{$\scriptstyle\rangle$}}}}
\newcommand{\bigsh}[1]{{\ensuremath{\hspace{1mm}\makebox[-1mm]{$\big\langle$}\makebox[0mm]{$\big\langle$}\hspace{1mm}{#1}\makebox[1mm]{$\big\rangle$}\makebox[0mm]{$\big\rangle$}}}}
\newcommand{\Bigsh}[1]{{\ensuremath{\hspace{1mm}\makebox[-1mm]{$\Big\langle$}\makebox[0mm]{$\Big\langle$}\hspace{1mm}{#1}\makebox[1mm]{$\Big\rangle$}\makebox[0mm]{$\Big\rangle$}}}}

\newtheorem{thm}{Theorem}[section]
  
\newtheorem{cor}{Corollary}
  \makeatletter\let\c@cor\c@thm\makeatother
  \numberwithin{cor}{section}
  
\newtheorem{prop}{Proposition}
  \makeatletter\let\c@prop\c@thm\makeatother
  \numberwithin{prop}{section}
  
\newtheorem{lem}{Lemma}
  \makeatletter\let\c@lem\c@thm\makeatother
  \numberwithin{lem}{section}
  
\theoremstyle{definition}
\newtheorem{defn}{Definition}
  \makeatletter\let\c@defn\c@thm\makeatother
  \numberwithin{defn}{section}

  \makeatletter\let\c@notn\c@thm\makeatother
  \numberwithin{notn}{section}
  
\theoremstyle{remark}
\newtheorem{rmk}{Remark}
  \makeatletter\let\c@rmk\c@thm\makeatother
  \numberwithin{rmk}{section}
  
\newtheorem{eg}{Example}
  \makeatletter\let\c@eg\c@thm\makeatother
  \numberwithin{eg}{section}
  
\newtheorem{egs}{Examples}
  \makeatletter\let\c@egs\c@thm\makeatother
  \numberwithin{egs}{section}

\makeatletter
\let\c@equation\c@thm
\makeatother
\numberwithin{equation}{section}

\makeatletter
\def\alwaysmath#1{\expandafter\expandafter\expandafter\global\expandafter\expandafter\expandafter\let\expandafter\expandafter\csname your@#1\endcsname\csname #1\endcsname
  \expandafter\def\csname #1\endcsname{\ensuremath{\csname your@#1\endcsname}}}
\alwaysmath{alpha}
\alwaysmath{beta}
\alwaysmath{gamma}
\alwaysmath{Gamma}
\alwaysmath{delta}
\alwaysmath{Delta}
\alwaysmath{epsilon}
\newcommand{\ep}{\ensuremath{\varepsilon}}
\alwaysmath{zeta}
\alwaysmath{eta}
\alwaysmath{theta}
\alwaysmath{Theta}
\alwaysmath{iota}
\alwaysmath{kappa}
\alwaysmath{lambda}
\alwaysmath{Lambda}
\alwaysmath{mu}
\alwaysmath{nu}
\alwaysmath{xi}
\alwaysmath{pi}
\alwaysmath{rho}
\alwaysmath{sigma}
\alwaysmath{Sigma}
\alwaysmath{tau}
\alwaysmath{upsilon}
\alwaysmath{Upsilon}
\alwaysmath{phi}
\alwaysmath{Phi}

\alwaysmath{chi}
\alwaysmath{psi}
\alwaysmath{Psi}
\alwaysmath{omega}
\alwaysmath{Omega}
\alwaysmath{ell}
\alwaysmath{infty}
\alwaysmath{sslash}
\alwaysmath{odot}
\makeatother

\let\al\alpha

\newcommand{\rdual}[1]{{{#1}^{\bigstar}}}
\newcommand{\ral}[1]{\al^{\bigstar}_{#1}}

\newcommand{\Ft}{\ensuremath{F_{\mathrm{tr}}}}
\newcommand{\Gt}{\ensuremath{G_{\mathrm{tr}}}}
\newcommand{\alt}{\ensuremath{\alpha_{\mathrm{tr}}}}
\newcommand{\tr}{\ensuremath{\operatorname{tr}}}

\newcommand{\fii}{\ensuremath{\mathfrak{i}}}

\newcommand{\Sip}{\ensuremath{\Sigma^\infty_+}}

\newcommand{\proj}[1]{\ensuremath{\mathrm{pr}}_{#1}}
\newcommand{\res}[1]{{({#1})}}

\newcommand{\Set}{\ensuremath{\mathbf{Set}}}

\newcommand{\Top}{\ensuremath{\mathbf{Top}}}
\newcommand{\Sp}{\ensuremath{\mathbf{Sp}}}
\newcommand{\Vect}[1]{\ensuremath{\mathbf{Vect}_{#1}}}

\newcommand{\bCh}[1]{\ensuremath{\mathbf{Ch}_{#1}}}

\newcommand{\gSp}[1]{\ensuremath{{#1}\text{-}\mathbf{Sp}}}
\newcommand{\bEx}[1]{\ensuremath{\mathbf{Sp}_{#1}}}

\newcommand{\Cat}{\ensuremath{\mathcal{C}\mathit{at}}}

\newcommand{\calMod}{\ensuremath{\mathcal{M}\mathit{od}/\!_{\mathcal{R}\mathit{ing}}}}
\newcommand{\calGrMod}{\ensuremath{\mathcal{G}\mathit{r}\mathcal{M}\mathit{od}/\!_{\mathcal{R}\mathit{ing}}}}
\def\calMat(#1){\ensuremath{\mathcal{M}\mathit{at}(\mathbf{#1})/\!_{\mathcal{S}\mathit{et}}}}

\newcommand{\calCh}{\ensuremath{\mathcal{C}\mathit{h}/\!_{\mathcal{R}\mathit{ing}}}}
\newcommand{\calChDGA}{\ensuremath{\mathcal{C}\mathit{h}/\!_{\mathit{DGA}}}}
\newcommand{\calEx}{\ensuremath{\mathcal{S}\mathit{p}/\!_{\mathcal{T}\mathit{op}}}}

\makeatletter
\def\calSpan{\futurelet\next@char\@calSpan}
\def\@calSpan{\if\next@char(\def\span@next{\@@calSpan}\else\def\span@next{\@@calSpan(S)}\fi\span@next}
\def\@@calSpan(#1){\ensuremath{\mathbf{#1}/\!_{\mathbf{#1}}}}
\newcommand{\calnCob}{\ensuremath{n\mathcal{C}\mathit{ob}/\!_{\mathcal{M}\mathit{fd}}}}

\newcommand{\bgptop}{\ensuremath{\mathcal{T}\!\mathit{op}_*/\!_{\mathcal{G}\mathit{rp}}}}
\newcommand{\bgpsp}{\ensuremath{\mathcal{S}\mathit{p}/\!_{\mathcal{G}\mathit{rp}}}}
\newcommand{\hobgpsp}{\ensuremath{\Ho(\mathcal{S}\mathit{p}/\!_{\mathcal{G}\mathit{rp}})}}

\newcommand{\bccat}[1]{\ensuremath{{#1}/\!_\star}}

\makeatletter
\def\Span{\futurelet\next@char\@bbSpan}
\def\@bbSpan{\if\next@char(\def\span@next{\@@bbSpan}\else\def\span@next{\@@bbSpan(S)}\fi\span@next}
\def\@@bbSpan(#1){\ensuremath{\mathbf{#1}\mathord{\sslash}\!_{\mathbf{#1}}}}

\def\Mat(#1){\ensuremath{\mathbb{M}\mathbf{at}(#1)\mathord{\sslash}\!_{\mathbb{S}\mathbf{et}}}}

\UseAllTwocells
\usetikzlibrary{scopes,calc,arrows,patterns}
\tikzset{ed/.style={auto,inner sep=0pt,font=\scriptsize}} 
\tikzset{>=stealth'}
\tikzset{vert/.style={draw,circle,inner sep=1pt,fill=white}}
\tikzset{vert2/.style={draw,circle,inner sep=2pt,fill=white}}

\colorlet{myblue}{blue!40!white}
\colorlet{myred}{red!35!white}
\colorlet{mygreen}{green!30!white}
\colorlet{myyellow}{yellow!10!white}
\tikzset{bluefill/.style={fill=myblue}}
\tikzset{redfill/.style={fill=myred}}
\tikzset{greenfill/.style={fill=mygreen}}
\tikzset{yellowfill/.style={fill=myyellow}}
\tikzset{dotsF/.style={pattern=north east lines,pattern color=black!60!white}}
\tikzset{dotsG/.style={pattern=north west lines,pattern color=black!60!white}}

\usetikzlibrary{decorations.pathmorphing}
\usetikzlibrary{decorations}
\tikzset{transf/.style={decorate,decoration={zigzag,amplitude=1pt,segment length=3pt}}}

\pgfdeclarelayer{background}
\pgfdeclarelayer{foreground}
\pgfsetlayers{background,main,foreground}

\def\bgcylinder#1#2#3#4#5#6{
  \def\cylempty{}\def\cylfrontcolor{#5}\def\cylbackcolor{#6}
  \begin{pgfonlayer}{background}
    \ifx\cylfrontcolor\cylempty\draw\else\fill[fill=my#5]\fi
    (#1) coordinate (dl)
    -- ++(0,#2) node[coordinate] (ul) {} 
    arc (-180:0:#3 and #4) coordinate (ur)
    -- ++(0,-#2) node[coordinate] (dr) {}
    arc (0:-180:#3 and #4);
    \ifx\cylbackcolor\cylempty\draw\else\fill[fill=my#6!80!black]\fi
    ($(ul)!.5!(ur)$) node[coordinate] (top) {} ellipse (#3 and #4);
    \path (dl) arc (-180:-90:#3 and #4) node[coordinate] (bot) {};
    \clip (dl) -- (ul) arc (-180:0:#3 and #4) -- (dr) arc (0:-180:#3 and #4);
  \end{pgfonlayer}
  \begin{pgfonlayer}{foreground}
    \clip (dl) -- (ul) arc (-180:0:#3 and #4) -- (dr) arc (0:-180:#3 and #4);
  \end{pgfonlayer}
  \clip (dl) -- (ul) arc (-180:0:#3 and #4) -- (dr) arc (0:-180:#3 and #4);
  \path (ul) ++(-0.1,0.1) coordinate (ul');
  \path (ur) ++(0.1,0.1) coordinate (ur');
  \path (dl) ++(-0.1,-0.1) coordinate (dl');
  \path (dr) ++(0.1,-0.1) coordinate (dr');
  \path (top) ++(0,0.1) coordinate (top');
  \path (bot) ++(-0,-0.1) coordinate (bot');
}

\def\drawtheta#1#2#3{
  \path ($(ul)!(#1)!(dl)$) ++(0,#2) coordinate (#3L);
  \draw[dashed] (#3L) to [out=-90,in=90,looseness=0.5] ($(ur)!(#3L)!(dr) + (0,-1)$) coordinate (#3R);
}

\def\elltheta#1#2#3{
  \path ($(ul)!(#1)!(dl)$) ++(0,#2) coordinate (#3L);
  \draw[dashed] (#3L) to [out=90,in=90,looseness=0.5] ($(ur)!(#3L)!(dr)$) coordinate (#3R);
}

\newenvironment{tikzcenter}{\begin{center}\begin{tikzpicture}}{\end{tikzpicture}\end{center}}

\makeatletter

\newif\iftikz@to@relp
\tikz@to@relpfalse
\newif\iftikz@to@relpp
\tikz@to@relppfalse
\tikzoption{abs}[]{\tikz@to@relpfalse\tikz@to@relppfalse}
\tikzoption{rel}[]{\tikz@to@relptrue\tikz@to@relppfalse}
\tikzoption{rrel}[]{\tikz@to@relpfalse\tikz@to@relpptrue}

\tikzstyle{every curve to}=          []
\tikzstyle{curve to}=                [to path=\tikz@to@curve@path]

\tikzoption{bend angle}{\def\tikz@to@bend{#1}}

\tikzoption{bend left}[]{%
  \def\pgf@temp{#1}%
  \ifx\pgf@temp\pgfutil@empty%
  \else%
    \def\tikz@to@bend{#1}%
  \fi%
  \let\tikz@to@out=\tikz@to@bend%
  \c@pgf@counta=180\relax%
  \advance\c@pgf@counta by-\tikz@to@out\relax%
  \edef\tikz@to@in{\the\c@pgf@counta}%
  \tikz@to@switch@on%
  \tikz@to@relativetrue%
}

\tikzoption{bend right}[]{%
  \def\pgf@temp{#1}%
  \ifx\pgf@temp\pgfutil@empty%
  \else%
    \def\tikz@to@bend{#1}%
  \fi%
  \c@pgf@counta=\tikz@to@bend\relax%
  \c@pgf@counta=-\c@pgf@counta\relax%
  \edef\tikz@to@out{\the\c@pgf@counta}%
  \c@pgf@counta=180\relax%
  \advance\c@pgf@counta by-\tikz@to@out\relax%
  \edef\tikz@to@in{\the\c@pgf@counta}%
  \tikz@to@switch@on%
  \tikz@to@relativetrue%
}

\tikzoption{relative}[true]{\csname tikz@to@relative#1\endcsname}
\newif\iftikz@to@relative
\tikz@to@relativefalse

\tikzoption{in}{\def\tikz@to@in{#1}\tikz@to@switch@on}
\tikzoption{out}{\def\tikz@to@out{#1}\tikz@to@switch@on}

\tikzoption{in looseness}{\tikz@to@set@in@looseness{#1}}
\tikzoption{out looseness}{\tikz@to@set@out@looseness{#1}}
\tikzoption{looseness}{\tikz@to@set@in@looseness{#1}\tikz@to@set@out@looseness{#1}}

\tikzoption{in control}{\tikz@to@set@in@control{#1}}
\tikzoption{out control}{\tikz@to@set@out@control{#1}}
\tikzoption{controls}{\tikz@to@parse@controls#1\pgf@stop}

\tikzoption{in min distance}{\tikz@to@set@distances{#1}{}{}{}}
\tikzoption{in max distance}{\tikz@to@set@distances{}{#1}{}{}}
\tikzoption{in distance}{\tikz@to@set@distances{#1}{#1}{}{}}
\tikzoption{out min distance}{\tikz@to@set@distances{}{}{#1}{}}
\tikzoption{out max distance}{\tikz@to@set@distances{}{}{}{#1}}
\tikzoption{out distance}{\tikz@to@set@distances{}{}{#1}{#1}}
\tikzoption{min distance}{\tikz@to@set@distances{#1}{}{#1}{}}
\tikzoption{max distance}{\tikz@to@set@distances{}{#1}{}{#1}}
\tikzoption{distance}{\tikz@to@set@distances{#1}{#1}{#1}{#1}}

\def\tikz@to@set@distances#1#2#3#4{%
  \tikz@to@setifnotempy{#1}{\tikz@to@in@min}{\let\tikz@to@end@compute=\tikz@to@end@compute@looseness}%
  \tikz@to@setifnotempy{#2}{\tikz@to@in@max}{\let\tikz@to@end@compute=\tikz@to@end@compute@looseness}%
  \tikz@to@setifnotempy{#3}{\tikz@to@out@min}{\let\tikz@to@start@compute=\tikz@to@start@compute@looseness}%
  \tikz@to@setifnotempy{#4}{\tikz@to@out@max}{\let\tikz@to@start@compute=\tikz@to@start@compute@looseness}%
  \tikz@to@switch@on%
}

\def\tikz@to@setifnotempy#1#2#3{%
  \def\pgf@temp{#1}%
  \ifx\pgf@temp\pgfutil@empty\else\def#2{#1}#3\fi%
}

\def\tikz@to@set@in@looseness#1{%
  \def\tikz@to@in@looseness{#1}%
  \let\tikz@to@end@compute=\tikz@to@end@compute@looseness%
  \tikz@to@switch@on%
}
\def\tikz@to@set@out@looseness#1{%
  \def\tikz@to@out@looseness{#1}%
  \let\tikz@to@start@compute=\tikz@to@start@compute@looseness%
  \tikz@to@switch@on%
}

\def\tikz@to@parse@controls#1and#2\pgf@stop{\tikz@to@set@in@control{#2}\tikz@to@set@out@control{#1}}

\def\tikz@to@set@in@control#1{%
  \def\tikz@to@in@control{#1}%
  \let\tikz@to@end@compute=\tikz@to@end@compute@control%
  \tikz@to@switch@on%
}
\def\tikz@to@set@out@control#1{%
  \def\tikz@to@out@control{#1}%
  \let\tikz@to@start@compute=\tikz@to@start@compute@control%
  \tikz@to@switch@on%
}

\def\tikz@to@bend{30}

\def\tikz@to@out{45}
\def\tikz@to@in{135}

\def\tikz@to@out@looseness{1}
\def\tikz@to@in@looseness{1}

\def\tikz@to@in@min{0pt}
\def\tikz@to@in@max{10000pt}
\def\tikz@to@out@min{0pt}
\def\tikz@to@out@max{10000pt}

\def\tikz@to@switch@on{\let\tikz@to@path=\tikz@to@curve@path}

\def\tikz@to@curve@path{%
  [every curve to]
  \pgfextra{\iftikz@to@relative\tikz@to@compute@relative\else\tikz@to@compute\fi}
  \tikz@computed@path
  \tikztonodes%
  \pgfextra{\tikz@to@relpfalse\tikz@to@relppfalse}%
}

\def\tikz@to@modify#1#2{%
  \pgfutil@ifundefined{pgf@sh@ns@#1}
  {}%
  {\edef#1{#1.#2}}
}%

\def\tikz@to@compute{%
  \let\tikz@tofrom=\tikztostart%
  \let\tikz@toto=\tikztotarget%
  \tikz@to@modify\tikz@tofrom\tikz@to@out%
  \tikz@to@modify\tikz@toto\tikz@to@in%
  \ifx\tikz@to@start@compute\tikz@to@start@compute@looseness%
    \tikz@to@compute@distance%
  \else%
    \ifx\tikz@from@start@compute\tikz@to@start@compute@looseness%
      \tikz@to@compute@distance%
    \fi%
  \fi%
  \tikz@to@start@compute%
  \tikz@to@end@compute%
  \iftikz@to@relp
    \edef\tikz@computed@path{.. controls \tikz@computed@start and \tikz@computed@end .. +(\tikz@toto)}
  \else
    \iftikz@to@relpp  
      \edef\tikz@computed@path{.. controls \tikz@computed@start and \tikz@computed@end .. ++(\tikz@toto)}
    \else
      \edef\tikz@computed@path{.. controls \tikz@computed@start and \tikz@computed@end .. (\tikz@toto)}
    \fi
  \fi
}

\def\tikz@to@compute@distance{\tikz@scan@one@point\tikz@@to@compute@distance(\tikz@tofrom)}
\def\tikz@@to@compute@distance#1{%
  \def\tikz@first@point{#1}%
  \iftikz@to@relp%
    \tikz@scan@one@point\tikz@@@to@compute@distance([shift={(\tikz@toto)}]\tikz@tofrom)%
  \else%
    \iftikz@to@relpp%
      \tikz@scan@one@point\tikz@@@to@compute@distance([shift={(\tikz@toto)}]\tikz@tofrom)%
    \else%
      \tikz@scan@one@point\tikz@@@to@compute@distance(\tikz@toto)%
    \fi%
  \fi}
\def\tikz@@@to@compute@distance#1{%
  \def\tikz@second@point{#1}%
  \tikz@to@compute@distance@main%
}
\def\tikz@to@compute@distance@main{%
  \pgf@process{\pgfpointdiff{\tikz@first@point}{\tikz@second@point}}%
  \ifdim\pgf@x<0pt\pgf@xa=-\pgf@x\else\pgf@xa=\pgf@x\fi%
  \ifdim\pgf@y<0pt\pgf@ya=-\pgf@y\else\pgf@ya=\pgf@y\fi%
  \pgf@process{\pgfpointnormalised{\pgfqpoint{\pgf@xa}{\pgf@ya}}}%
  \ifdim\pgf@x>\pgf@y%
    \c@pgf@counta=\pgf@x%
    \ifnum\c@pgf@counta=0\relax%
    \else%
      \divide\c@pgf@counta by 255\relax%
      \pgf@xa=16\pgf@xa\relax%
      \divide\pgf@xa by\c@pgf@counta%
      \pgf@xa=16\pgf@xa\relax%
    \fi%
  \else%
    \c@pgf@counta=\pgf@y%
    \ifnum\c@pgf@counta=0\relax%
    \else%
      \divide\c@pgf@counta by 255\relax%
      \pgf@ya=16\pgf@ya\relax%
      \divide\pgf@ya by\c@pgf@counta%
      \pgf@xa=16\pgf@ya\relax%
    \fi%
  \fi%
  \pgf@x=0.3915\pgf@xa%
  \pgf@xa=\tikz@to@out@looseness\pgf@x%
  \pgf@xb=\tikz@to@in@looseness\pgf@x%
  \pgfmathsetlength{\pgf@ya}{\tikz@to@out@min}
  \ifdim\pgf@xa<\pgf@ya%
    \pgf@xa=\pgf@ya%
  \fi%
  \pgfmathsetlength{\pgf@ya}{\tikz@to@out@max}
  \ifdim\pgf@xa>\pgf@ya%
    \pgf@xa=\pgf@ya%
  \fi%
  \pgfmathsetlength{\pgf@ya}{\tikz@to@in@min}
  \ifdim\pgf@xb<\pgf@ya%
    \pgf@xb=\pgf@ya%
  \fi%
  \pgfmathsetlength{\pgf@ya}{\tikz@to@in@max}
  \ifdim\pgf@xb>\pgf@ya%
    \pgf@xb=\pgf@ya%
  \fi%
}

\def\tikz@to@start@compute@looseness{%
  \edef\tikz@computed@start{([shift=(\tikz@to@out:\the\pgf@xa)]\tikz@tofrom)}%
}
\def\tikz@to@end@compute@looseness{%
  \edef\tikz@computed@end{+(\tikz@to@in:\the\pgf@xb)}%
}
\def\tikz@to@start@compute@control{%
  \let\tikz@computed@start=\tikz@to@out@control%
}
\def\tikz@to@end@compute@control{%
  \let\tikz@computed@end=\tikz@to@in@control%
}

\let\tikz@to@start@compute=\tikz@to@start@compute@looseness%
\let\tikz@to@end@compute=\tikz@to@end@compute@looseness%

\def\tikz@to@compute@relative{%
  \tikz@scan@one@point\tikz@@to@compute@relative(\tikztostart)%
}
\def\tikz@@to@compute@relative#1{%
  \def\tikz@tofrom{#1}%
  \tikz@scan@one@point\tikz@@@to@compute@relative(\tikztotarget)%
}
\def\tikz@@@to@compute@relative#1{%
  \def\tikz@toto{#1}%
  \begingroup
    \pgfutil@ifundefined{pgf@sh@ns@\tikztostart}
    {%
      \let\tikz@first@point=\tikz@tofrom%
      \let\tikz@tostart@tikz=\pgfutil@empty
    }%
    {%
      {%
        \tikz@tofrom%
        \pgf@xc=\pgf@x%
        \pgf@yc=\pgf@y%
        {%
          \pgftransformreset%
          \pgftransformshift{\pgfqpoint{\pgf@xc}{\pgf@yc}}%
          \pgftransformrotate{\tikz@to@out}%
          \pgftransformshift{\pgfqpoint{-\pgf@xc}{-\pgf@yc}}%
          \pgf@process{\pgfpointtransformed{\tikz@toto}}%
        }%
        \pgf@xc=\pgf@x%
        \pgf@yc=\pgf@y%
        \pgfpointshapeborder{\tikztostart}{\pgfqpoint{\pgf@xc}{\pgf@yc}}%
        \xdef\tikz@tofrom@smuggle{\noexpand\pgfqpoint{\the\pgf@x}{\the\pgf@y}}
      }%
      \let\tikz@first@point=\tikz@tofrom@smuggle%
      \tikz@first@point%
      \edef\tikz@tostart@tikz{(\the\pgf@x,\the\pgf@y)}%
    }%
    \pgfutil@ifundefined{pgf@sh@ns@\tikztotarget}
    {%
      \let\tikz@second@point=\tikz@toto%
    }%
    {%
      {%
        \tikz@toto%
        \pgf@xc=\pgf@x%
        \pgf@yc=\pgf@y%
        {%
          \pgftransformreset%
          \pgftransformshift{\pgfqpoint{\pgf@xc}{\pgf@yc}}%
          \pgftransformrotate{180}%
          \pgftransformrotate{\tikz@to@in}%
          \pgftransformshift{\pgfqpoint{-\pgf@xc}{-\pgf@yc}}%
          \pgf@process{\pgfpointtransformed{\tikz@tofrom}}%
        }%
        \pgf@xc=\pgf@x%
        \pgf@yc=\pgf@y%
        \pgfpointshapeborder{\tikztotarget}{\pgfqpoint{\pgf@xc}{\pgf@yc}}%
        \xdef\tikz@toto@smuggle{\noexpand\pgfqpoint{\the\pgf@x}{\the\pgf@y}}
      }%
      \let\tikz@second@point=\tikz@toto@smuggle%
    }%
    \tikz@second@point%
    \edef\tikz@totarget@tikz{(\the\pgf@x,\the\pgf@y)}%
    \tikz@to@compute@distance@main%
    \edef\tikz@to@first@distance{\the\pgf@xa}%
    \edef\tikz@to@second@distance{\the\pgf@xb}%
    \pgftransformreset%
    \pgf@process{\tikz@first@point}%
    \pgf@xa=\pgf@x%
    \pgf@ya=\pgf@y%
    \pgf@process{\tikz@second@point}%
    \advance\pgf@x by-\pgf@xa%
    \advance\pgf@y by-\pgf@ya%
    \pgfpointnormalised{}%
    \pgf@xc=\pgf@x%
    \pgf@yc=\pgf@y%
    \pgf@xb=-\pgf@x%
    \pgf@yb=-\pgf@y%
    {%
      \pgftransformshift{\tikz@first@point}%
      \pgftransformcm{\pgf@sys@tonumber\pgf@xc}{\pgf@sys@tonumber\pgf@yc}{\pgf@sys@tonumber\pgf@yb}{\pgf@sys@tonumber\pgf@xc}%
                      {\pgfpointorigin}%
      \pgf@process{\pgfpointtransformed{\pgfpointpolar{\tikz@to@out}{\tikz@to@first@distance}}}%
      \xdef\tikz@computed@start{(\the\pgf@x,\the\pgf@y)}%
    }
    {%
      \pgftransformshift{\tikz@second@point}%
      \pgftransformcm{\pgf@sys@tonumber\pgf@xc}{\pgf@sys@tonumber\pgf@yc}{\pgf@sys@tonumber\pgf@yb}{\pgf@sys@tonumber\pgf@xc}%
                      {\pgfpointorigin}%
      \pgf@process{\pgfpointtransformed{\pgfpointpolar{\tikz@to@in}{\tikz@to@second@distance}}}%
      \xdef\tikz@computed@end{(\the\pgf@x,\the\pgf@y)}%
    }
    \xdef\tikz@computed@path{
      \tikz@tostart@tikz
      .. controls \tikz@computed@start and \tikz@computed@end ..
      \tikz@totarget@tikz}%
  \endgroup
}

\title{Shadows and traces in bicategories}
\thanks{Both authors were supported by National Science Foundation
  postdoctoral fellowships during the writing of this paper.
  The final publication is available at \texttt{www.springerlink.com}.
  It differs immaterially from this accepted manuscript.}

\author{Kate Ponto \and Michael Shulman}

\begin{document}

\maketitle

\begin{abstract}
  Traces in symmetric monoidal categories are well-known and have many
  applications; for instance, their functoriality directly implies the
  Lefschetz fixed point theorem.  However, for some applications, such as generalizations of the Lefschetz theorem, one
  needs ``noncommutative'' traces, such as the Hattori-Stallings trace
  for modules over noncommutative rings.  In this paper we study a
  generalization of the symmetric monoidal trace which applies to
  noncommutative situations; its context is a bicategory equipped with
  an extra structure called a ``shadow.''
  In particular, we prove its functoriality and 2-functoriality, which are essential to its applications in fixed-point theory.
  Throughout we make use of an appropriate ``cylindrical'' type of string diagram, which we justify formally in an appendix.
\end{abstract}

\section{Introduction}
\label{sec:intro}

The purpose of this paper is to study and exposit a categorical notion
of \emph{trace} for endo-2-cells in a bicategory.  Since there are
also other sorts of categorical ``traces,'' we begin by briefly
describing where ours fits into the general picture.  Probably the
most basic sort of trace is the trace of a square matrix over a field.
This generalizes to square matrices over a commutative ring, and even
to endomorphisms of finitely generated projective modules over a
commutative ring.  It is well-known that there is an appropriate
categorical definition of such traces, which applies to any
endomorphism of a \emph{dualizable object} in a \emph{symmetric
  monoidal category}; see, for example,~\cite{dp:duality,kl:cpt}.

This general definition includes many important notions in topology, geometry, and algebra.
For instance, traces in the stable homotopy category can be identified with \emph{fixed point indices}, which include in particular \emph{Euler characteristics} (as the traces of identity maps)---while traces in the derived category of a ring are called \emph{Lefschetz numbers}.
Moreover, the category-theoretic definition of trace trivially implies that it is preserved by symmetric monoidal functors, such as homology.
We thereby immediately obtain the classical \emph{Lefschetz fixed point theorem}: if an endomorphism of a dualizable space (such as a finite-dimensional manifold) has no fixed points, then its fixed point index is zero, and hence so is the Lefschetz number of the map it induces on homology.
The search for similar category-theoretic expressions of other fixed-point invariants was the primary motivation for the generalization of trace we will present in this paper; see~\cite{kate:traces, kate:rel,equiv}.
  
One way to generalize traces is to remove the requirement that the objects be dualizable.
In this case, traces become extra \emph{structure} imposed on a symmetric monoidal category; see~\cite{jsv:traced-moncat}.
Such traces can even exist in cartesian monoidal categories, though there are no nontrivial dual pairs in such a category.
The connection with fixed points is then even more striking: a trace on a cartesian monoidal category is equivalent to an operator which chooses a fixed point for every map (see~\cite{hasegawa}).
These sort of traces are especially important in computer science, where the fixed-point operator is identified with a recursion combinator.
However, they are not particularly useful in topological fixed point theory, since a given map may have zero, one, two, or more fixed points of equal importance, and so there can be no uniform fixed-point--assigning operator.

Thus, we will generalize in a different direction: we still consider only dualizable objects, but relax the requirement of symmetry.
It is easy to define traces for dualizable objects in a monoidal category which is merely braided, or more precisely \emph{balanced} (see, e.g.,~\cite{jsv:traced-moncat}).
However, the applications in~\cite{kate:traces, kate:rel,equiv} require notions of trace in even less commutative situations.
For instance, Stallings~\cite{stallings} proved that for modules over a noncommutative ring, there is a unique notion of trace which is additive and cyclic.
This ``Hattori-Stallings trace'' (see also~\cite{hattori}) lives in a suitable quotient of the ring, rather than the ring itself.
We will present a generalization of the symmetric monoidal trace which includes the Hattori-Stallings trace.
It was invented by the first author for applications in fixed point theory; see~\cite{kate:traces}.

Modules over noncommutative rings do not form a monoidal category at
all, but rather a \emph{bicategory} (or, if necessary, a double
category; see e.g.~\cite{shulman:frbi}).  It makes no sense to ask whether a bicategory is
symmetric, braided, or balanced, but it turns out that there is a type
of structure we can impose on a bicategory which enables us to define
traces therein.  We refer to this structure as a \emph{shadow},
although it could also quite reasonably be called a ``2-trace,'' since
it is itself a categorified kind of trace.  (This is an example of the
\emph{microcosm principle} of~\cite{bd:hda3}.)  Many bicategories are
naturally equipped with shadows, such as the following:
\begin{enumerate}
\item Rings and bimodules (the Hattori-Stallings context), as well as
 generalizations such as DGAs and chain complexes.
\item Parametrized spaces and parametrized spectra, as studied
 in~\cite{maysig:pht}.
\item Spaces with group actions, as studied in~\cite{ranicki}.
\item $n$-dimensional manifolds and cobordisms.
\item Categories and profunctors (also called bimodules or
 distributors), with enriched and internal variations.
\end{enumerate}
The goal of this paper is to define and explain the notions of shadow and trace in bicategories and prove some of their basic properties.
The main property we are interested in is functoriality, since this is what makes Lefschetz-style theorems fall out easily.
Thus, the main part of the paper can be regarded as a build-up to our functoriality results.

Our intent is to make this paper a ``bridge'', accessible to two audiences.
On the one hand, there are topologists 
familiar with the topological applications and interested in the formal foundations who may appreciate an introduction to the basic category-theoretic ideas.
On the other hand, there may be category-theorists interested in applications of categorified traces, who will understand the categorical ideas and notation already, but may have little background in topology.
This makes for a difficult balancing act.

Since the central definitions and results in this paper are category-theoretic, we have chosen to spend more time on the necessary categorical background.
The topology appears only in examples, and for these we give some intuitive description, along with references for further reading.
We have also tried to include enough non-topological examples that even a reader without any topological background can grasp the categorical ideas---although we stress that the most important applications are topological (see~\cite{kate:traces, kate:rel,equiv,PS:mult}).

We start with a brief review of the classical theory of symmetric monoidal traces, so that the analogies with bicategorical trace will be clear.
\autoref{sec:traces} contains the basic definitions, examples, and properties, including the all-important functoriality.
Less often cited, but also important, is what we call ``2-functoriality'': symmetric monoidal traces commute not only with functors but with natural transformations.
Everything in this section can be found in classical references such as~\cite{dp:duality,jsv:traced-moncat}.
We also summarize the classical ``string diagram'' notation for symmetric monoidal categories, which provides a convenient notation and calculus for manipulating composites of many morphisms; see~\cite{penrose:negdimten,pr:spinors,js:geom-tenscalc-i,jsv:traced-moncat,selinger:graphical,street:ldtop-hocat}.
\autoref{sec:bicategories} is also a review of classical material, this time the definition of bicategories and their string diagrams.

The next sections \ref{sec:shadows} and \ref{sec:bicat-traces} contain the basic definitions of shadows and traces, respectively, with \autoref{sec:examples} devoted to a number of examples.
These definitions originally appeared in~\cite{kate:traces}, but here we study them carefully from a categorical perspective.
In particular, in \autoref{sec:prop-bicat} we prove a number of formal properties of the bicategorical trace, analogous to the familiar formal properties of symmetric monoidal trace.
These properties are most conveniently expressed and proven using an appropriate string diagram calculus, which in the case of shadows involves diagrams drawn on a \emph{cylinder}; we introduce these string diagrams in \S\ref{sec:shadows}.

Finally, in Sections \ref{sec:funct-bicat} and \ref{sec:transf} we prove the crucial results about functoriality and 2-functoriality for the bicategorical trace, starting with the necessary definitions.
Just as a bicategory must be equipped with the extra structure of a shadow in order to define traces, a functor of bicategories must be given the structure of a ``shadow functor'' in order for it to respect traces.
The appropriate notion of ``shadow transformation'' is somewhat more subtle, involving a bicategorical transformation whose components are dual pairs, rather than single 1-cells.
The definitions of shadow functor and shadow transformation are included in \S\S\ref{sec:funct-bicat} and \ref{sec:transf} respectively.

Finally, in Appendix \ref{sec:string-diagrams} we give a formal basis to our cylindrical string diagrams by proving that any such labeled string diagram determines a unique deformation-invariant composite.
This is a technical, but fairly straightforward extension of the classical proof for monoidal categories in~\cite{js:geom-tenscalc-i}.

\subsection*{Acknowledgments}
\label{sec:acknowledgments}

The authors would like to thank Niles Johnson, for careful reading and
helpful comments.


\section{Traces in symmetric monoidal categories}
\label{sec:traces}

We begin by reviewing traces in symmetric monoidal categories, using string diagram calculus, a summary of which can be found in \autoref{fig:string-moncat}.
These diagrams were first used by Penrose~\cite{penrose:negdimten,pr:spinors}, given a rigorous foundation by Joyal and Street~\cite{js:geom-tenscalc-i,jsv:traced-moncat}, and since then have been adapted to many different contexts; a comprehensive overview can be found in~\cite{selinger:graphical}.
They may be called ``Poincar\'e dual'' to the usual sorts of diagrams: instead of drawing objects as \emph{vertices} and morphisms as \emph{arrows}, we draw objects as \emph{strings} and morphisms as \emph{vertices}, often with boxes around them.
We may think of a morphism as a ``machine'' with its domain drawn as ``input'' strings coming into it and its codomain as ``output'' strings going out of it.

All of our string diagrams are read from top to bottom.
After a while we will omit the arrowheads on the strings, but in the beginning they can help to clarify the intent.
Note that although we draw the symmetry $M\ten N \iso N\ten M$ with one string crossing ``over'' the other, there is no meaning assigned to which one is in front, since our monoidal category is symmetric and not merely braided.

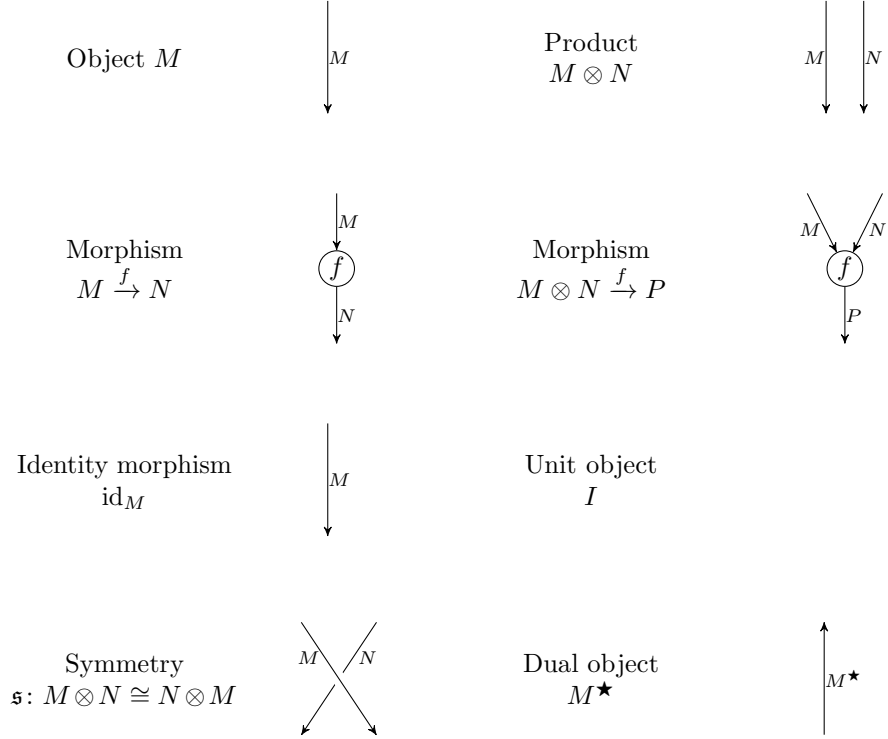
\begin{figure}[htb]
  \centering
  \begin{tabular}{m{30mm}m{20mm}m{40mm}m{20mm}}
    \begin{center}Object $M$\end{center} &
    \begin{center}\begin{tikzpicture}
      \draw[->] (0,1.5) -- node [ed] {$M$} (0,0);
    \end{tikzpicture}\end{center}
    &
    \begin{center}Product\\$M\ten N$\end{center} &
    \begin{center}\begin{tikzpicture}
      \draw[->] (0,1.5) -- node [ed,swap] {$M$} (0,0);
      \draw[->] (0.5,1.5) -- node [ed] {$N$} (0.5,0);
    \end{tikzpicture}\end{center}
    \\\\
    \begin{center}Morphism\\$M\xrightarrow{f} N$\end{center} &
    \begin{center}\begin{tikzpicture}
      \node[vert] (f) at (0,0) {$f$};
      \draw[->] (f) -- node [ed] {$N$} +(0,-1);
      \draw[<-] (f) -- node [ed,swap] {$M$} +(0,1);
    \end{tikzpicture}\end{center}
    &
    \begin{center}Morphism\\$M\ten N \xrightarrow{f} P$\end{center} &
    \begin{center}\begin{tikzpicture}
      \node[vert] (f) at (0,0) {$f$};
      \draw[->] (f) -- node [ed] {$P$} +(0,-1);
      \draw[<-] (f) -- node [ed] {$M$} +(-0.5,1);
      \draw[<-] (f) -- node [ed,swap] {$N$} +(0.5,1);
    \end{tikzpicture}\end{center}
    \\\\
    \begin{center}Identity morphism\\$\id_M$\end{center} &
    \begin{center}\begin{tikzpicture}
      \draw[->] (0,1.5) -- node [ed] {$M$} (0,0);
    \end{tikzpicture}\end{center}
    &
    \begin{center}Unit object\\$I$\end{center} &
    \\\\
    \begin{center}Symmetry\\$\fs\colon M\ten N \iso N\ten M$\end{center} &
    \begin{center}\begin{tikzpicture}
      \draw[->] (0.5,1.5) -- node [ed,near start] {$N$} (-0.5,0);
      \draw[white,line width=5pt] (-0.5,1.5) -- (0.5,0);
      \draw[->] (-0.5,1.5) -- node [ed,near start,swap] {$M$} (0.5,0);
    \end{tikzpicture}\end{center}
    &
    \begin{center}Dual object\\$\rdual{M}$\end{center} &
    \begin{center}\begin{tikzpicture}
      \draw[<-] (0,1.5) -- node [ed] {$\rdual{M}$} (0,0);
    \end{tikzpicture}\end{center}
  \end{tabular}
    \caption{String diagrams for monoidal categories}
  \label{fig:string-moncat}
\end{figure}

In~\cite{js:geom-tenscalc-i}, Joyal and Street gave a formal definition of the ``value'' of a ``string diagram'' whose arrows and vertices are labeled by objects and morphisms in a monoidal category, and showed that this value is invariant under deformations of diagrams.
Thus, manipulation of string diagrams is actually a fully rigorous way to prove theorems about symmetric monoidal categories.

\begin{defn}
  Let \bC\ be a symmetric monoidal category with product $\ten$ and
  unit object $I$.  An object $M$ of \bC\ is \textbf{dualizable} if
  there exists an object $\rdual{M}$, called its \textbf{dual}, and
  maps
  \begin{align*}
    I &\too[\eta] M\ten \rdual{M} & \rdual{M}\ten M &\too[\ep] I
  \end{align*}
  satisfying the triangle identities $(\id_M\ten \ep)(\eta\ten \id_M)
  = \id_M$ and $(\ep\ten \id_\rdual{M})(\id_\rdual{M}\ten \eta) =
  \id_{\rdual{M}}$.  We call $\ep$ the \textbf{evaluation} and $\eta$
  the \textbf{coevaluation}.
\end{defn}

Note that any two duals of an object $M$ are canonically isomorphic,
and that if $\rdual{M}$ is a dual of $M$, then $M$ is also a dual of
$\rdual{M}$.  If $M$ and $N$ are dualizable and $f\colon Q\otimes M
\rightarrow N\otimes P$ is a morphism in $\bC$, the \textbf{mate}
of $f$ is the composite
\[\xymatrix{\rdual{N}\otimes Q\ar[r]^-{\id\otimes \eta}&
\rdual{N}\otimes Q\otimes M\otimes \rdual{M}
\ar[r]^-{\id\otimes f\otimes \id}&\rdual{N}\otimes N\otimes P\otimes \rdual{M}
\ar[r]^-{\epsilon\otimes\id}&P\otimes \rdual{M}
}\]

Dual objects are represented graphically by turning around the direction of arrows, while the triangle identities translate into ``bent strings can be straightened;'' see Figures~\ref{fig:moncat-dual} and~\ref{fig:triangle}.
String diagrams for monoidal categories with duals are formalized in~\cite{js:geom-tenscalc-ii,js:pdta} allowing arbitrary behavior of strings (not restricting them to travel only vertically).
However, for most purposes it suffices to consider only vertical strings, regarding each ``turning around'' of a string as a vertex implicitly labeled by $\eta$ or $\varepsilon$ and the triangle identities simply as axioms (rather than deformations).

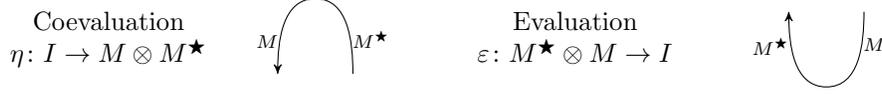
\begin{figure}[tb]
  \centering
    \begin{tabular}{m{30mm}m{20mm}m{40mm}m{20mm}}
    \begin{center}Coevaluation\\$\eta\maps I \to M\ten \rdual{M}$\end{center} &
    \begin{center}
      \begin{tikzpicture}
        \draw[<-] (0,0) to[out=90,in=180] node [ed,near start] {$M$} (0.5,1)
        to[out=0,in=90] node [ed,near end] {$\rdual{M}$} (1,0);
      \end{tikzpicture}
    \end{center}
    &
    \begin{center}Evaluation\\$\ep\maps \rdual{M}\ten M\to I$\end{center} &
      \begin{tikzpicture}
        \draw[<-] (0,0) to[out=-90,in=180] node [ed,near start,swap] {$\rdual{M}$} (0.5,-1)
        to[out=0,in=-90] node [ed,near end,swap] {$M$} (1,0);
      \end{tikzpicture}
  \end{tabular}
  \caption{Coevaluation and evaluation}
  \label{fig:moncat-dual}
\end{figure} 

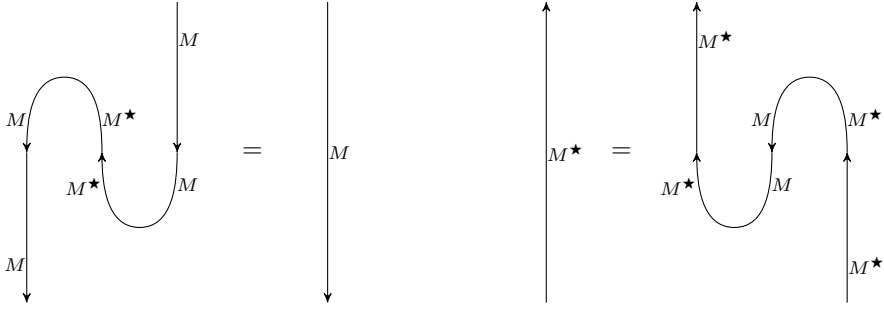
\begin{figure}[tb]
  \centering
  \begin{tabular}{m{50mm}m{15mm}m{50mm}}
    \begin{tikzpicture}
      \draw[<-] (0,0) -- node[ed,near start] {$M$} ++(0,2) coordinate (a);
      \draw[<-] (a) to[rrel,out=90,in=180] node [ed,near start] {$M$} (0.5,1)
      to[rrel,out=0,in=90] node [ed,near end] {$\rdual{M}$} (0.5,-1) coordinate (b);
      \draw[<-] (b) to[rrel,out=-90,in=180] node [ed,near start,swap] {$\rdual{M}$} (0.5,-1)
      to[rrel,out=0,in=-90] node [ed,near end,swap] {$M$} (0.5,1) coordinate (c);
      \draw[<-] (c) -- node[ed,near end,swap] {$M$} ++(0,2);
      \node at (3,2) {$=$};
      \draw[<-] (4,0) -- node[ed,swap] {$M$} ++(0,4);
    \end{tikzpicture} &
    &
    \begin{tikzpicture}
      \draw[->] (-2,0) -- node[ed,swap] {$\rdual{M}$} ++(0,4);
      \node at (-1,2) {$=$};
      \draw[<-] (0,4) -- node[ed,near start] {$\rdual{M}$} ++(0,-2) coordinate (a);
      \draw[<-] (a) to[rrel,out=-90,in=180] node [ed,near start,swap] {$\rdual{M}$} (0.5,-1)
      to[rrel,out=0,in=-90] node [ed,near end,swap] {$M$} (0.5,1) coordinate (b);
      \draw[<-] (b) to[rrel,out=90,in=180] node [ed,near start] {$M$} (0.5,1)
      to[rrel,out=0,in=90] node [ed,near end] {$\rdual{M}$} (0.5,-1) coordinate (c);
      \draw[<-] (c) -- node[ed,near end] {$\rdual{M}$} ++(0,-2);
    \end{tikzpicture}
  \end{tabular}
  \caption{The triangle identities}
  \label{fig:triangle}
\end{figure}


We now recall the definition of the trace of an endomorphism of a dualizable object in a symmetric monoidal category.
In fact, we can take the trace of more than just endomorphisms; we only require the dualizable object to appear as a factor of the source and target.

\begin{defn}\label{def:twisted-comm-trace}
  Let \bC\ be a symmetric monoidal category, $M$ a dualizable object
  of $\bC$, and $f\maps Q\ten M\to M\ten P$ a morphism in $\bC$.  The
  \textbf{trace} $\tr(f)$ of $f$ is the following composite:
  \begin{equation}
    Q \too[\eta] Q\ten M\ten \rdual{M} \too[f] M \ten P \ten \rdual{M}
    \xiso{\fs} \rdual{M}\ten M\ten P \too[\ep] P
    \label{eq:twisted-comm-trace}
  \end{equation}
\end{defn}

References for this notion of trace include~\cite{dp:duality,fy:brd-cpt,js:brd-tensor,jsv:traced-moncat,kl:cpt}.
The definition does not depend on the choice of dual $\rdual{M}$ or the evaluation and coevaluation maps.
When $Q=P=I$, it reduces to the more familiar definition of the trace of an endomorphism.
There are also two other degenerate cases of this definition that classically go by other names: the \textbf{Euler characteristic} (or \textbf{dimension}) of $M$ is the trace of its identity map, and the \textbf{transfer} of $M$ is the trace of a ``diagonal'' morphism $\Delta\maps M\to M\ten M$ when such exists (so in this case $P=M$ and $Q=I$).

In string diagram notation, the trace of a morphism looks like ``feeding its output into its input;'' see \autoref{fig:smctrace}.
The second picture looks cleaner and is more commonly drawn, but we note that it makes essential use of the symmetry, in regarding $f$ as a morphism $Q\ten M \to P\ten M$ and in switching the order in the source of the evaluation map.
Thus, since we intend to generalize away from symmetry, it is important to keep the first picture in mind as well.

\begin{figure}[tb]
  \centering
  \begin{tabular}{m{30mm}m{5mm}m{25mm}m{15mm}m{20mm}}
    \begin{tikzpicture}
      \node[vert] (f) at (0,0) {$f$};
      \draw[<-] (f) -- node [ed] {$Q$} +(-1,2.5);
      \draw[->] (f) -- node [ed,near end] {$P$} +(1,-3.5) coordinate (p);
      \draw[<-] (f) to[rrel,out=75,in=-90] node [ed] {$M$} (0.4,1.5)
      arc (180:0:0.3)
      -- node [ed] {$\rdual{M}$} ++(0,-2) node[coordinate] (c) {};
      \begin{pgfonlayer}{background}
        \draw (c) to[rrel,out=-90,in=90] (-2,-2.3) node[coordinate] (d) {};
      \end{pgfonlayer}
      \draw (d) arc (-180:0:0.3) node[coordinate] (e) {}
      to[out=90,in=-100,looseness=0.5] (f);
      \begin{pgfonlayer}{background}
        \draw[white,line width=5pt] (f) -- (p);
        \draw[white,line width=5pt] (f) -- (e);
      \end{pgfonlayer}
    \end{tikzpicture}
    & $=$ &
    $\xymatrix{Q \ar[d]^{\id\ten \eta} \\ Q\ten M\ten \rdual{M} \ar[d]^{f\ten
        \id} \\ M\ten P \ten \rdual{M} \ar[d]^{\iso} \\ \rdual{M}\ten
      M\ten P \ar[d]^{\ep\ten \id} \\ P}$
    & \begin{center}or\\just\end{center} &
    \begin{tikzpicture}
      \node[vert] (f) at (0,0) {$f$};
      \draw[<-] (f) -- node [ed] {$Q$} +(-1,2.5);
      \draw[->] (f) -- node [ed] {$P$} +(-1,-2.5);
      \draw[->] (f) to[out=-70,in=70,looseness=15]
      node [ed,swap] {$\rdual{M}$} (f);
    \end{tikzpicture}
  \end{tabular}
  \caption{The symmetric monoidal trace}
  \label{fig:smctrace}
\end{figure}
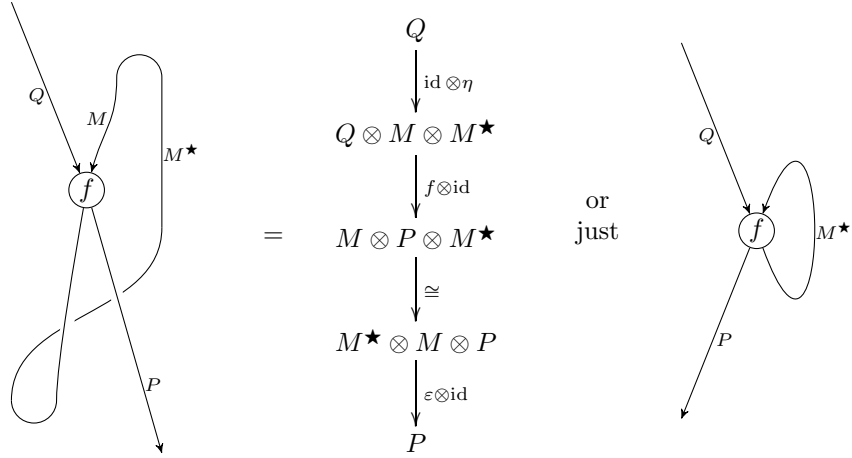

\begin{egs}\label{eg:vect}
  Let $\bC=\Vect{k}$ be the category of vector spaces over a
  field $k$.  A vector space is dualizable if and only if it is
  finite-dimensional, and its dual is the usual dual vector space.  We
  have $I=k$ and $\bC(I,I)\iso k$ by multiplication.  Using this identification,
  \autoref{def:twisted-comm-trace} recovers the usual trace of a matrix.  The
  Euler characteristic of a vector space (i.e.\ the trace of its identity map) is its dimension.

  This example generalizes to modules, chain complexes of modules, and
  the derived category of modules over a commutative ring.  In these
  cases the dualizable objects are the finitely generated projectives,
  finitely generated chain complexes of projectives, and chain
  complexes quasi-isomorphic to a finitely generated chain complex of
  projectives respectively.

  If $V= k[S]$ is the free module on a finite set $S$, then the diagonal $S\to S\times S$ induces a diagonal $V\to V\otimes V$.
  Thus, any endomorphism of $V$ has a ``transfer'' $k\to V$, which just picks out an element of $V$ (the image of $1$).
  If $f\colon V\rightarrow V$ is induced by an endomorphism $\hat{f}\colon S\rightarrow S$, then its transfer is the sum $\sum_{\hat{f}(s)=s} s$ of all the fixed points of $\hat{f}$.
\end{egs}

\begin{egs}\label{eg:stab-duality}
  Another important class of examples is topological: whenever $M$ is
  a closed smooth manifold (or a compact ENR), its suspension spectrum
  $\Sigma^\infty(M_+)$ is dualizable in the stable homotopy category.
  This is called \emph{Spanier-Whitehead duality}.

  For the reader unfamiliar with stable homotopy theory, what this
  means in concrete terms is that there is a pointed space $N$,
  together with evaluation and coevaluation maps
  \[ S^n \too[\eta] M_+ \wedge N \hspace{2cm} N \wedge M_+ \too[\ep] S^n
  \]
  for some large-dimensional sphere $S^n$, such that the triangle
  identities commute up to ``stable homotopy''.  Working ``stably''
  means roughly that we can do and undo smash products with spheres
  freely, and the stable homotopy category is a way of making that
  precise.  When formulated explicitly as above, this sort of duality
  is called \emph{$n$-duality}.

  Now, given an endomorphism $f$ of $M$, we can either map it into the
  stable homotopy category and take its trace as an endomorphism of
  $\Sigma^\infty(M_+)$, or define a trace directly using $n$-duality.
  In the latter case, we obtain an endomorphism of $S^n$, which is
  classified by an integer (its degree).  This is the \emph{fixed
    point index} of $f$, which counts the number of fixed points of
  $f$ with multiplicity;
  see~\cite{d:index,dp:duality,lms:equivariant,maysig:pht,td:groups}.
  In particular, the Euler characteristic of $\Sigma^\infty(M_+)$
  (that is, the trace of its identity map) can be identified with the
  usual Euler characteristic of $M$.

  Here we also
  have a diagonal $\Sigma^\infty(M_+) \to \Sigma^\infty(M_+) \wedge
  \Sigma^\infty(M_+)$ induced by the diagonal $\Delta\colon M\to
  M\times M$ of $M$, and hence there is a transfer $S\to
  \Sigma^\infty(M_+)$.
  This transfer can again be regarded as the ``formal sum'' of all the fixed points of the original endomorphism.
\end{egs}

\begin{rmk}\label{rmk:cartesian}
Note that the functors $k[-]\colon \Set \to \Vect{k}$ and $\Sigma^{\infty}(-)_{+}\colon\Top\to\Sp$ play a similar role in examples~\ref{eg:vect} and~\ref{eg:stab-duality}.
In both cases we start with an object in a cartesian monoidal category, where there are no nontrivial dualizable objects or traces, and apply a functor landing in a noncartesian monoidal category, after which our object becomes dualizable and we can calculate traces.
Moreover, in both cases the noncartesian monoidal category is ``additive'' and the trace gives us the ``sum'' of all the fixed points of a map in our original cartesian monoidal category.
\end{rmk}


The symmetric monoidal trace has the following fundamental property.

\begin{prop}[Cyclicity]\label{thm:twisted-comm-cyclic}
  If $M$ and $N$ are dualizable and $f\maps Q\otimes M\to N\otimes P$
  and $g\maps K\otimes N\to M\otimes L$ are morphisms, then
  \begin{equation}
    \tr\Big((g\otimes \id_P)(\id_K \otimes f)\Big)
    = \tr\Big(\fs(f\otimes \id_L) (\id_Q\otimes g)\fs\Big).\label{eq:smc-cyclicity}
  \end{equation}
  In particular, for $f\colon M\to N$ and $g\colon N\to M$, we have $\tr(g f) = \tr(f g)$.
\end{prop}

This property can be proven directly with a long sequence of
equations, but a much more conceptual proof is possible using string
diagrams.  In \autoref{fig:smc-cyclicity} we have drawn both sides
of~\eqref{eq:smc-cyclicity} as string diagrams.  To prove
\autoref{thm:twisted-comm-cyclic} it then suffices to observe that one
of these diagrams can be deformed into the other (this is easiest to
see at first when $Q=P=K=L=I$ are the unit object).  The fundamental
theorem of Joyal and Street~\cite{js:geom-tenscalc-i} then implies
that~\eqref{eq:smc-cyclicity} holds in any symmetric monoidal
category.

\begin{figure}[tb]
  \centering
  \begin{tikzpicture}
    \node[vert] (f) at (0,0) {$f$};
    \draw[<-] (f) -- node [ed] {$Q$} +(-1,2);
    \draw[->] (f) -- node [ed,near end] {$P$} +(1.5,-4.5);
    \node[vert2] (g) at (-0.5,-1.5) {$g$};
    \draw[<-] (g) -- node [ed] {$K$} +(-1.3,3.5);
    \draw[->] (g) -- node [ed,near end] {$L$} +(1,-3);
    \draw[->] (f) -- node [ed] {$N$} (g);
    \draw[<-] (f) to[rrel,out=75,in=-90] node [ed] {$M$} (0.4,1.5)
    arc (180:0:0.3)
    -- node [ed] {$\rdual{M}$} ++(0,-3)
    to[rrel,out=-90,in=90] (-3,-2.3)
    coordinate[label={[ed]left:$\rdual{M}$}] arc (-180:0:0.3)
    coordinate[label={[ed]right:$M$}] to[out=90,in=-100] (g);
  \end{tikzpicture}
  \qquad\raisebox{3cm}{=}\qquad
  \begin{tikzpicture}
    \node[vert2] (g) at (0,-0.3) {$g$};
    \draw[<-] (g) to[rrel,out=110,in=-80] node [ed,swap] {$K$} (-1.5,2.2);
    \draw[->] (g) to[rrel,out=-70,in=90] node [ed,near end] {$L$} (1,-2.5)
    to[rrel,out=-90,in=80] (-1,-1.5);
    \node[vert] (f) at (-0.5,-1.5) {$f$};
    \draw[<-] (f) to[rrel,out=110,in=-90] node [ed] {$Q$} (-1,2)
    to[rrel,out=90,in=-110] (0.75,1.4);
    \draw[->] (f) to[rrel,out=-70,in=110] node [ed,swap] {$P$} (0.5,-1.5)
    to[rrel,out=-70,in=100] (1,-1.3);
    \draw[->] (g) -- node [ed] {$M$} (f);
    \draw[<-] (g) to[rrel,out=75,in=-90] node [ed] {$N$} (0.4,1)
    arc (180:0:0.3)
    -- node [ed] {$\rdual{N}$} ++(0,-2) 
    to[rrel,out=-90,in=90] (-3,-1.8)
    coordinate[label={[ed]left:$\rdual{N}$}] arc (-180:0:0.3)
    coordinate[label={[ed]right:$N$}] to[out=90,in=-100] (f);
  \end{tikzpicture}
  \caption{Cyclicity of the symmetric monoidal trace}
  \label{fig:smc-cyclicity}
\end{figure}
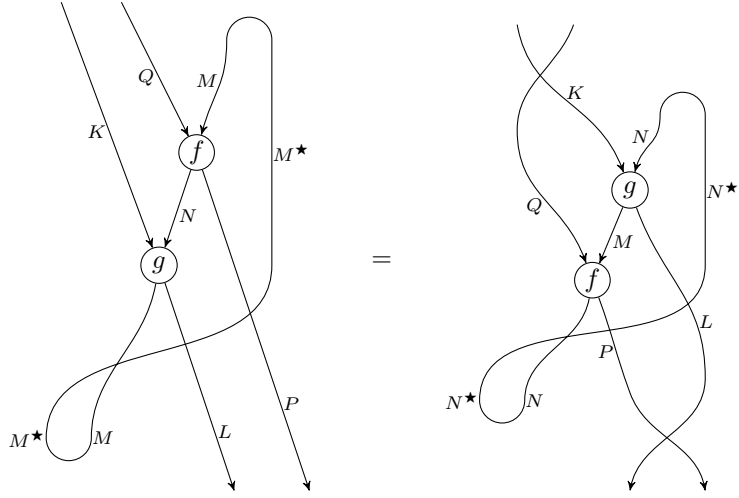

%
%

The symmetric monoidal trace satisfies many other useful naturality properties, most of which are likewise easiest to prove using string diagrams; see for instance~\cite{dp:duality,lms:equivariant,may:traces,jsv:traced-moncat}.
We will consider generalizations of many of these properties in \S\ref{sec:prop-bicat}.

Finally, as mentioned in the introduction, one of the main advantages of having an abstract formulation
of trace is that disparate notions of trace which all fall into the
general framework can be compared functorially.  Recall that a
\textbf{lax symmetric monoidal functor} $F\maps \bC\rightarrow \bD$
between symmetric monoidal categories consists of a functor $F$ and
natural transformations
\begin{align*}
 \fc\maps F(M)\ten F(N)&\too F(M\otimes N)\\
 \fii\maps I_{\bD}&\too F(I_{\bC})
\end{align*}
satisfying appropriate coherence axioms.  We say $F$ is
\textbf{normal} if \fii\ is an isomorphism, and \textbf{strong} if
\fc\ and \fii\ are both isomorphisms.

When drawing string diagrams, we follow~\cite{mccurdy:stripes,mccurdy:gmtd} by imagining a monoidal functor as a kind of `fluid' or `environment' in which our strings and vertices can be immersed, and we notate such immersion graphically by a pattern of dots or lines characteristic of the functor.
(We avoid the use of color, since that will be used to denote 0-cells when we come to string diagrams for bicategories in \S\ref{sec:bicategories}.)
For simplicity, we continue to label strings and vertices by the objects and morphisms in the domain category \bC, since the presence of a functor pattern indicates application of the functor to yield corresponding objects and morphisms in \bD.
Thus, for instance,
\raisebox{-2.5mm}{\begin{tikzpicture}[scale=.7]
  \draw (0,0) -- node[ed] {$M$} (0,-1);
  \path (-.5,-1) rectangle (.5,0);
\end{tikzpicture}}
denotes the object $M\in\bC$, while
\raisebox{-2.5mm}{\begin{tikzpicture}[scale=.7]
  \draw (0,0) -- node[ed] {$M$} (0,-1);
  \path[dotsF] (-.5,-1) rectangle (.5,0);
\end{tikzpicture}}
denotes the object $F(M)\in\bD$.
With this notation, the structure maps \fc\ and \fii\ are shown in \autoref{fig:monfr}, and their coherence axioms in \autoref{fig:monfr-coh}.
(The last axiom is just the naturality of the transformation \fc.)
When either of these structure maps is an isomorphism, we draw its inverse in the same way, but upside down.

\begin{figure}[tbp]
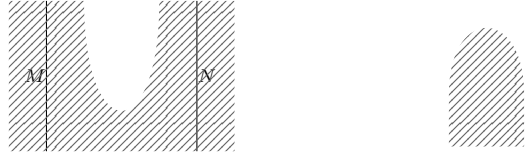

  \centering
  \begin{tabular}{m{45mm}m{45mm}}
    \begin{tikzcenter}
      \draw (1,0) -- node[ed] {$M$} (1,2);
      \draw (3,0) -- node[ed,swap] {$N$} (3,2);
      \path[dotsF] (.5,0) -- (3.5,0) -- (3.5,2) -- (2.5,2)
      to[out=-90,in=-90,looseness=5] (1.5,2) -- (0.5,2) -- cycle;
    \end{tikzcenter}
    &
    \begin{tikzcenter}
      \path[dotsF] (0,0) -- (1,0) -- ++(0,.7) to[rrel,out=90,in=90,looseness=3] (-1,0) -- cycle;
    \end{tikzcenter}
  \end{tabular}
  \caption{The constraints of a lax monoidal functor}
  \label{fig:monfr}
\end{figure}

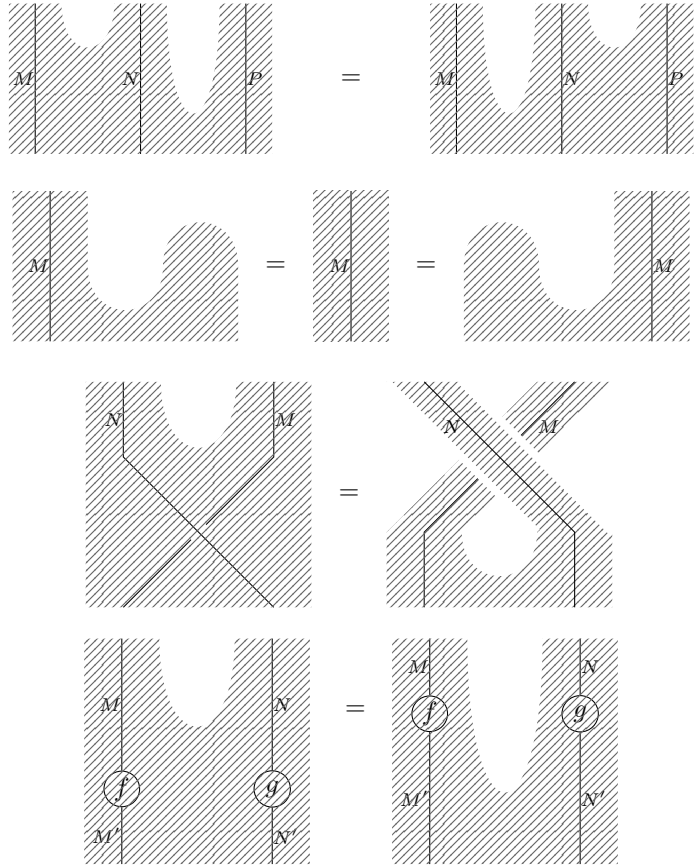
\begin{figure}[tbp]
  \centering
  \begin{tikzcenter}[xscale=.7]
    \draw (1,0) -- node[ed] {$M$} (1,2);
    \draw (3,0) -- node[ed] {$N$} (3,2);
    \draw (5,0) -- node[ed,swap] {$P$} (5,2);
    \path[dotsF] (.5,0) -- (5.5,0) -- (5.5,2) -- (4.5,2)
    to[out=-90,in=-90,looseness=5] (3.5,2) -- (2.5,2)
    to[out=-90,in=-90,looseness=2] (1.5,2) -- (0.5,2)
    -- cycle;
    \node at (7,1) {$=$};
    \begin{scope}[xshift=8cm]
      \draw (1,0) -- node[ed] {$M$} (1,2);
      \draw (3,0) -- node[ed,swap] {$N$} (3,2);
      \draw (5,0) -- node[ed,swap] {$P$} (5,2);
      \path[dotsF] (.5,0) -- (5.5,0) -- (5.5,2) -- (4.5,2)
      to[out=-90,in=-90,looseness=2] (3.5,2) -- (2.5,2)
      to[out=-90,in=-90,looseness=5] (1.5,2) -- (0.5,2)
      -- cycle;
    \end{scope}
  \end{tikzcenter}
  \bigskip
  \begin{tikzcenter}
    \draw (1,0) -- node[ed] {$M$} +(0,2);
    \path[dotsF] (.5,0) -- (3.5,0) -- (3.5,1) to[out=90,in=90,looseness=2]
    (2.5,1) to[out=-90,in=-90,looseness=2] (1.5,1) -- (1.5,2) -- (0.5,2) -- cycle;
    \node at (4,1) {$=$};
    \draw (5,0) -- node[ed] {$M$} +(0,2);
    \path[dotsF] (4.5,0) rectangle +(1,2);
    \node at (6,1) {$=$};
    \begin{scope}[xshift=10cm,x={(-1cm,0)}]
      \draw (1,0) -- node[ed,swap] {$M$} +(0,2);
      \path[dotsF] (.5,0) -- (3.5,0) -- (3.5,1) to[out=90,in=90,looseness=2]
      (2.5,1) to[out=-90,in=-90,looseness=2] (1.5,1) -- (1.5,2) -- (0.5,2) -- cycle;
    \end{scope}
  \end{tikzcenter}
  \bigskip
  \begin{tikzcenter}
    \draw (0,0) --  (2,2) -- node[ed,swap] {$M$} (2,3);
    \fill[white] (.2,2) -- (2.2,0) -- (1.8,0) -- (-.2,2) -- cycle;
    \draw (0,3) -- node[ed,swap] {$N$} (0,2) -- (2,0);
    \path[dotsF] (-.5,0) -- (2.5,0) -- (2.5,3) -- (1.5,3)
    to[out=-90,in=-90,looseness=3] (.5,3) -- (-.5,3) -- cycle;
    \node at (3,1.5) {$=$};
    \draw (4,0) -- (4,1) -- node[ed,near end,swap] {$M$} (6,3);
    \fill[white] (3.5,3) -- (5.5,1) -- (6.5,1) -- (4.5,3) -- cycle;
    \path[dotsF] (3.5,0) -- (3.5,1) -- (5.5,3) -- (6.5,3) -- (4.5,1)
    to[out=-90,in=-90,looseness=2] (5.5,1) -- (3.5,3) -- (4.5,3) -- (6.5,1)
    -- (6.5,0) -- cycle;
    \draw[line width=3pt,white] (3.5,3) -- (5.5,1);
    \draw[line width=3pt,white] (4.5,3) -- (6.5,1);
    \draw (4,3) -- node[ed,near start,swap] {$N$} (6,1) -- (6,0);
  \end{tikzcenter}
  \bigskip
  \begin{tikzpicture}
    \node[vert] (f) at (1,1) {$f$};
    \node[vert2] (g) at (3,1) {$g$};
    \draw (1,0) -- node[ed] {$M'$} (f) -- node[ed] {$M$} (1,3);
    \draw (3,0) -- node[ed,swap] {$N'$} (g) -- node[ed,swap] {$N$} (3,3);
    \path[dotsF] (.5,0) -- (3.5,0) -- (3.5,3) -- (2.5,3)
    to[out=-90,in=-90,looseness=4] (1.5,3) -- (0.5,3) -- cycle;
  \end{tikzpicture}
  \quad\raisebox{2cm}{$=$}\quad
  \begin{tikzpicture}
    \node[vert] (f) at (1,2) {$f$};
    \node[vert2] (g) at (3,2) {$g$};
    \draw (1,0) -- node[ed] {$M'$} (f) -- node[ed] {$M$} (1,3);
    \draw (3,0) -- node[ed,swap] {$N'$} (g) -- node[ed,swap] {$N$} (3,3);
    \path[dotsF] (.5,0) -- (3.5,0) -- (3.5,3) -- (2.5,3)
    to[out=-90,in=-90,looseness=7] (1.5,3) -- (0.5,3) -- cycle;
  \end{tikzpicture}
  \caption{The axioms of a lax symmetric monoidal functor}
  \label{fig:monfr-coh}
\end{figure}

Unlike the situation for monoidal categories, it seems that string diagrams for monoidal functors have not yet been formalized.
Thus, we must view their use as merely a convenient shorthand for writing out more precise proofs.

\begin{prop}\label{thm:funct-pres-dual}\label{thm:funct-pres-tr}
  Let $F\maps \bC\rightarrow \bD$ be a normal lax symmetric monoidal
  functor, let $M\in \bC$ be dualizable with dual $\rdual{M}$, and
  assume that $\fc\maps F(M)\otimes F(\rdual{M})\rightarrow
  F(M\otimes \rdual{M})$ is an isomorphism.
  \begin{enumerate}
  \item Then $F(M)$ is dualizable with dual $F(\rdual{M})$.
  \end{enumerate}

  \noindent Also assume $\fc\maps F(P) \ten F(M) \to F(P\ten M)$ is an isomorphism (as it is whenever $P=I$, since $F$ is normal).
  \begin{enumerate}[resume]

  \item Then for any map $f\maps Q\ten M
  \to M\ten P$, we have
  \begin{equation}
    F(\tr(f)) = \tr\left(\fc^{-1} \circ F(f)\circ \fc\right).\label{eq:fprestr}
  \end{equation}
  \end{enumerate}
\end{prop}

The evaluation and coevaluation of $F(M)$ are shown graphically in \autoref{fig:presdual}.
Note the need for invertibility of \fii\ and of the one relevant component of \fc.

\begin{figure}[tbp]
  \centering
  \begin{tikzpicture}
    \path[use as bounding box] (0,0) rectangle (3,2.5);
    \draw (0.5,0) to[out=90,in=90,looseness=3]
    node[ed,very near start] {$M$} node[ed,very near end] {$\rdual{M}$} (2.5,0);
    \path[dotsF] (0,0) -- ++(1,0) to[rrel,out=90,in=90,looseness=3.5] (1,0) -- ++(1,0)
    to[out=90,in=90,looseness=2.5] (0,0);
  \end{tikzpicture}
  \hspace{1cm}
  \begin{tikzpicture}
    \path[use as bounding box] (0,0) rectangle (3,-2.5);
    \draw (0.5,0) to[out=-90,in=-90,looseness=3]
    node[ed,very near start,swap] {$\rdual{M}$} node[ed,very near end,swap] {$M$} (2.5,0);
    \path[dotsF,use as bounding box] (0,0) -- ++(1,0) to[rrel,out=-90,in=-90,looseness=3.5] (1,0) -- ++(1,0)
    to[out=-90,in=-90,looseness=2.5] (0,0);
  \end{tikzpicture}
  \caption{Preservation of duals}
  \label{fig:presdual}
\end{figure}
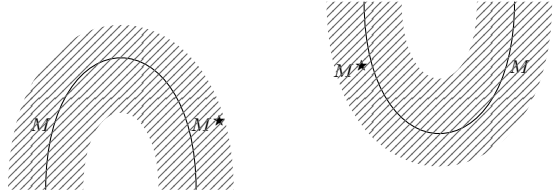

The equation~\eqref{eq:fprestr} is shown graphically in \autoref{fig:prestr}.
We leave it to the reader to prove this equation graphically, using the coherence laws drawn in \autoref{fig:monfr-coh} along with naturality (moving unrelated morphisms past each other side-by-side) and the invertibility of \fii\ and \fc.

\begin{figure}[hbt]
  \centering
  \def\drawtrace{
    \begin{pgfonlayer}{foreground}
      \node[vert] (f) at (0,0) {$f$};
      \draw (f) -- node [ed] {$Q$} +(-1,2.5);
      \draw (f) -- node [ed,near end] {$P$} +(1,-3.5) coordinate (p);
      \draw (f) to[rrel,out=75,in=-90] node [ed] {$M$} (0.4,1.5)
      node[coordinate] (a) {}
      arc (180:0:0.4) coordinate (b)
      -- ++(0,-2) coordinate (c);
    \end{pgfonlayer}
    \begin{pgfonlayer}{background}
      \draw (c) to[rrel,out=-90,in=90] (-2.5,-2.3) node[coordinate] (d) {};
    \end{pgfonlayer}
    \begin{pgfonlayer}{foreground}
      \draw (d) arc (-180:0:0.4) coordinate (e)
      to[out=90,in=-100,looseness=0.5] (f);
    \end{pgfonlayer}
  }
  \begin{tikzpicture}
    \drawtrace
    \fill[white] ($(f)+(-.2,-.5)$) -- ($(e)+(-.2,0)$) -- ++(.4,0)
    -- ($(f)+(-.1,-.5)$) -- cycle;
    \fill[white] ($(f)+(.2,-.5)$) -- ($(p)+(.2,0)$) -- ++(-.4,0)
    --  ($(f)+(.1,-.5)$) -- cycle;
    \begin{pgfonlayer}{foreground}
    \path[dotsF] (-1.5,-3.5) rectangle (1.5,2.5);
    \end{pgfonlayer}
  \end{tikzpicture}
  \quad\raisebox{3cm}{$=$}\quad
  \begin{tikzpicture}
    \drawtrace
    \fill[white] ($(f)+(-.35,-.5)$) -- ($(e)+(-.35,0)$) -- ++(.7,0)
    -- ($(f)+(.05,-.5)$) -- cycle;
    \fill[white] ($(f)+(.4,-.5)$) -- ($(p)+(.4,0)$) -- ++(-.8,0)
    --  ($(f)+(-.1,-.5)$) -- cycle;
    \begin{pgfonlayer}{foreground}
    \path[dotsF] (-1.3,2.5) -- (-.4,.5) -- (-.3,-.5) -- ($(e)+(-.2,0)$)
    to[out=-90,in=-90,looseness=1] ($(d)+(.2,0)$)
    -- ($(d)+(-.2,0)$) to[out=-90,in=-90,looseness=1.5]
    ($(e)+(.2,0)$) -- (0,-.7) -- (.8,-3.5) -- (1.2,-3.5) --
    (.3,-.5) -- (.4,.5) -- ($(a)+(.2,0)$)
    to[out=90,in=90,looseness=1] ($(b)+(-.2,0)$)
    -- ($(c)+(-.2,0)$) -- ++(.4,0)
    -- ($(b)+(.2,0)$) to[out=90,in=90,looseness=1.5] ($(a)+(-.2,0)$)
    -- (0,.8) -- (-.8,2.5) -- (-1.3,2.5);
    \end{pgfonlayer}
    \begin{pgfonlayer}{background}
      \path[dotsF] ($(c)+(-.2,0)$) .. controls +(0,-1.1) and +(0,1.4) .. ($(d)+(-.2,0)$)
      -- ($(d)+(.2,0)$) .. controls +(0,1.1) and +(0,-1.4) .. ($(c)+(.2,0)$);
    \end{pgfonlayer}
  \end{tikzpicture}
  \caption{Preservation of traces}
  \label{fig:prestr}
\end{figure}
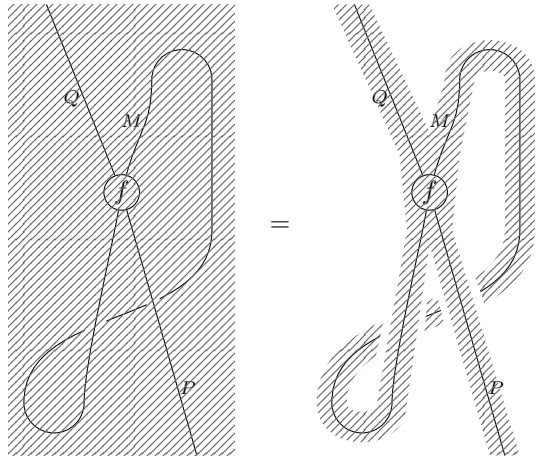

\begin{eg}\label{eg:lefschetz}
As mentioned in \S\ref{sec:intro}, \autoref{thm:funct-pres-tr} implies the Lefschetz fixed point theorem, by the following argument.
The rational chain complex functor is a strong symmetric
monoidal functor from the stable homotopy category to the derived category of $\mathbb{Q}$.
Composing this functor with the homology functor, which is strong symmetric monoidal
by the K\"{u}nneth theorem, gives a strong symmetric monoidal functor
to graded vector spaces.  Applying the proposition to this composite of 
functors identifies the Lefschetz number with the fixed point index.
\end{eg}

\begin{eg}
A simpler example is given by extension of scalars.
If $R$ and $S$ are commutative rings and $\psi\colon R\rightarrow S$ is a 
ring homomorphism, extension of scalars is a strong symmetric monodal 
functor from $R$-modules to $S$-modules.  If $f\colon Q\otimes M\rightarrow
M\otimes P$ is a map of $R$-modules, this proposition implies
$\tr(f\otimes_RS)=\tr(f)\otimes_RS$.
\end{eg}

In addition to this ``functoriality,'' the symmetric monoidal trace also satisfies a sort of ``2-functoriality.''
Recall that if $F,G\maps \bC\to\bD$ are lax symmetric monoidal functors, a
\textbf{monoidal natural transformation} is a natural transformation
$\al\maps F\to G$ which is compatible with the monoidal constraints of
$F$ and $G$ in an evident way.
Graphically, we draw a monoidal natural transformation as a `membrane' or `interface' between regions denoting the two functors, as shown in Figure~\ref{fig:montransdata}.
Note that the morphism $\alpha_{M}\colon F(M) \to G(M)$ is not explicitly pictured as a node, but rather implied as the $M$-string passes through the $\alpha$ membrane.

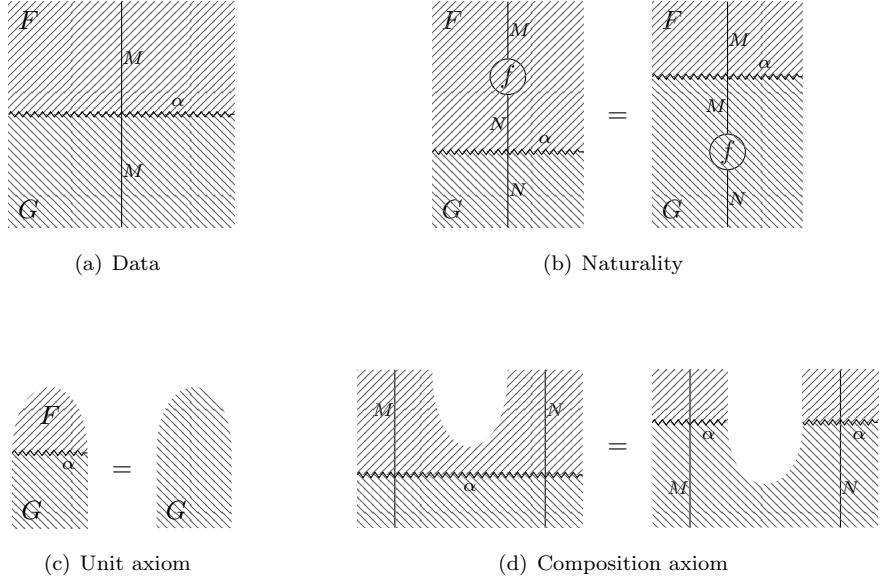
\begin{figure}[hbt]
  \begin{center}
  \renewcommand{\subfigcapskip}{10pt}
  \begin{tabular}{c@{\hspace{1.5cm}}c}
    \subfigure[Data]{\label{fig:montransdata}
  \begin{tikzpicture}[scale=1.5]
    \path[dotsF] (0,1) rectangle (2,2);
    \path[dotsG] (0,0) rectangle (2,1);
    \coordinate (a) at (1,1);
    \node[anchor=north west] at (0,2) {$F$};
    \node[anchor=south west] at (0,0) {$G$};
    \draw (1,2) -- node[ed] {$M$} (a) -- node[ed] {$M$} (1,0);
    \draw[transf] (0,1) -- (a) -- node[ed,above=3pt] {$\alpha$} (2,1);
  \end{tikzpicture}}
  &
  \subfigure[Naturality]{
  \begin{tikzpicture}
    \coordinate (a) at (1,1);
    \node[anchor=north west] at (0,3) {$F$};
    \node[anchor=south west] at (0,0) {$G$};
    \node[vert] (f) at (1,2) {$f$};
    \path[dotsF] (0,1) rectangle (2,3);
    \path[dotsG] (0,0) rectangle (2,1);
    \draw (1,3) -- node[ed] {$M$} (f) -- node[ed,swap] {$N$} (a) -- node[ed] {$N$} (1,0);
    \draw[transf] (0,1) -- (a) -- node[ed,above=3pt] {$\alpha$} (2,1);
  \end{tikzpicture}
  \hspace{.2cm}\raisebox{1.4cm}{$=$}\hspace{.2cm}
  \begin{tikzpicture}
    \coordinate (a) at (1,2);
    \node[anchor=north west] at (0,3) {$F$};
    \node[anchor=south west] at (0,0) {$G$};
    \node[vert] (f) at (1,1) {$f$};
    \path[dotsF] (0,2) rectangle (2,3);
    \path[dotsG] (0,0) rectangle (2,2);
    \draw (1,3) -- node[ed] {$M$} (a) -- node[ed,swap] {$M$} (f) -- node[ed] {$N$} (1,0);
    \draw[transf] (0,2) -- (a) -- node[ed,above=3pt] {$\alpha$} (2,2);
  \end{tikzpicture}}
  \\\\\\
  \subfigure[Unit axiom]{
  \begin{tikzpicture}
    \path[dotsG] (0,0) rectangle (1,1);
    \path[dotsF] (0,1) to[out=90,in=90,looseness=3] (1,1) -- (0,1);
    \draw[transf] (0,1) -- node[ed,near end,below=3pt] {$\alpha$} (1,1);
    \node at (.5,1.5) {$F$};
    \node[anchor=south west] at (0,0) {$G$};
  \end{tikzpicture}
  \hspace{.2cm}\raisebox{.7cm}{$=$}\hspace{.2cm}
  \begin{tikzpicture}
    \path[dotsG] (0,0) -- (0,1) to[rrel,out=90,in=90,looseness=3] (1,0) -- (1,0) -- cycle;
    \node[anchor=south west] at (0,0) {$G$};
  \end{tikzpicture}}
  &
  \subfigure[Composition axiom]{
  \begin{tikzpicture}[yscale=.7]
    \draw (1,-1) -- node[ed,near end] {$M$} (1,2);
    \draw (3,-1) -- node[ed,near end,swap] {$N$} (3,2);
    \path[dotsF] (.5,0) -- (3.5,0) -- (3.5,2) -- (2.5,2)
    to[out=-90,in=-90,looseness=5] (1.5,2) -- (0.5,2) -- cycle;
    \path[dotsG] (.5,-1) rectangle (3.5,0);
    \draw[transf] (.5,0) -- node[ed,below=3pt] {$\alpha$} (3.5,0);
  \end{tikzpicture}
  \hspace{.2cm}\raisebox{1cm}{$=$}\hspace{.2cm}
  \begin{tikzpicture}[yscale=.7]
    \draw (1,-1) -- node[ed,near start] {$M$} (1,2);
    \draw (3,-1) -- node[ed,near start,swap] {$N$} (3,2);
    \path[dotsG] (.5,-1) -- (3.5,-1) -- (3.5,1) -- (2.5,1)
    to[out=-90,in=-90,looseness=4] (1.5,1) -- (0.5,1) -- cycle;
    \path[dotsF] (.5,1) rectangle (1.5,2);
    \path[dotsF] (2.5,1) rectangle (3.5,2);
    \draw[transf] (.5,1) -- node[ed,below=3pt,near end] {$\alpha$} (1.5,1)
    (2.5,1) -- node[ed,below=3pt,near end] {$\alpha$} (3.5,1);
  \end{tikzpicture}}
  \end{tabular}\end{center}
  \caption{A monoidal natural transformation and its axioms}
  \label{fig:montrans}\label{fig:montransax}
\end{figure}

\begin{prop}\label{thm:trans-pres-tr}
  Let $F,G\maps \bC\to\bD$ be normal lax symmetric monoidal functors,
  let $\al\maps F\to G$ be a monoidal natural transformation, let $M$ be
  dualizable in \bC, and assume that $F$ and $G$ satisfy the
  hypotheses of \autoref{thm:funct-pres-tr}.
  Then
  \begin{enumerate}
  \item $\al_M\maps F(M)\to G(M)$ is an isomorphism, whose inverse is the mate
  of $\al_{\rdual{M}}$, and\label{item:tpt1}
  \item For any $f\maps Q\ten M \to M\ten P$, the following square commutes.\label{item:tpt2}
  \[\xymatrix@C=3pc{F(Q) \ar[rr]^{\tr\left(\fc^{-1} \circ F(f)\circ \fc\right)}
    \ar[d]_{\al_Q} && F(P) \ar[d]^{\al_P}\\
    G(Q) \ar[rr]_{\tr\left(\fc^{-1} \circ G(f)\circ \fc\right)} && G(P)}
  \]
  \item In particular, for $f\maps M\to M$, we have $\tr(F(f)) = \tr(G(f))$.\label{item:tpt3}
  \end{enumerate}
\end{prop}

\autoref{thm:trans-pres-tr}\ref{item:tpt1} is known colloquially as ``duals invert.''
Half of it is pictured in Figure~\ref{fig:duals-invert}; it follows from the axioms in \autoref{fig:montransax} and a triangle identity.
The other half is of course similar.
(This is one of the examples of string diagram calculus for monoidal functors given in~\cite{mccurdy:stripes}.)

A graphical picture of \autoref{thm:trans-pres-tr}\ref{item:tpt2} is shown in Figure~\ref{fig:transftr}.
It is easy to give a direct proof of this equality by successive applications of the axioms in \autoref{fig:montransax}.
In fact, there is an even easier proof: application of \autoref{thm:funct-pres-tr} to $F$ and $G$ reduces the desired statement to naturality of $\alpha$.
However, this easy proof will no longer be available in the bicategorical case.

Finally, \autoref{thm:trans-pres-tr}\ref{item:tpt3} follows because when $F$ and $G$ are normal, the component $\alpha_{I}$ at the unit object is always an isomorphism.

\begin{figure}[hbt]
  \centering
  \subfigure[Duals Invert]{\label{fig:duals-invert}
  \begin{tikzpicture}[scale=.6]
    \draw (0,0) coordinate (bot)
    to[out=90,in=100,looseness=2]
    node [ed,near start] {$M$}
    node [ed,very near end] {$\rdual{M}$}
    (2,3) coordinate (mid)
    to[out=-80,in=-90] node[ed,swap,pos=.9] {$M$}
    (4,6) coordinate (midtop) -- ++(0,2) coordinate (top);
    \path[dotsF] (bot) -- ++(-.4,0) to[out=90,in=100,looseness=2.2] ($(mid)+(.4,0)$)
    -- ++(-.8,0) to[out=100,in=90,looseness=1.3] ($(bot)+(.4,0)$) -- cycle;
    \path[dotsG] (mid) -- ++(.4,0) to[out=-80,in=-90,looseness=1.3] ($(midtop)+(-.4,0)$)
    -- ++(.8,0) to[out=-90,in=-80,looseness=2.2] ($(mid)+(-.4,0)$) -- cycle;
    \path[dotsF] (midtop) ++(-.4,0) rectangle ++(.8,2);
    \draw[transf] (mid) ++(-.4,0) -- ++(.8,0);
    \draw[transf] (midtop) ++(-.4,0) -- ++(.8,0);

    \node at (5.2,4) {$=$};

    \draw (6.6,0) -- node[ed,swap,near end] {$M$} ++(0,8);
    \path[dotsF] (6.2,0) rectangle ++(.8,8);
  \end{tikzpicture}
}\hspace{1.5cm}
\subfigure[Preservation of trace]{\label{fig:transftr}
  \begin{tikzpicture}[scale=.7]
    \node[vert] (f') at (0,0) {$f$};
    \draw (f') -- node [ed] {$Q$} +(-1,2.5);
    \draw (f') -- node [ed,near end] {$P$} ++(1,-3.5) -- ++(0,-1);
    \draw (f') to[rrel,out=75,in=-90] node [ed] {$M$} (0.4,1.5) node (a) {}
    arc (180:0:0.4) coordinate (b)
    -- ++(0,-2) coordinate (c)
    to[rrel,out=-90,in=90] (-2.5,-2.3) node (d) {}
    arc (-180:0:0.4) coordinate (e)
    to[out=90,in=-100,looseness=0.5] (f');
      \path[dotsF] (-1.3,2.5) -- (-.4,.5) -- (-.3,-.5) -- ($(e)+(-.2,0)$)
      to[out=-90,in=-90,looseness=1] ($(d)+(.2,0)$)
      -- ($(d)+(-.2,0)$) to[out=-90,in=-90,looseness=1.5]
      ($(e)+(.2,0)$) -- (0,-.7) -- (.8,-3.5) -- (1.2,-3.5) --
      (.3,-.5) -- (.4,.5) -- ($(a)+(.2,0)$)
      to[out=90,in=90,looseness=1] ($(b)+(-.2,0)$)
      -- ($(c)+(-.2,0)$) .. controls +(0,-1.1) and +(0,1.4) .. ($(d)+(-.2,0)$)
      -- ($(d)+(.2,0)$) .. controls +(0,1.1) and +(0,-1.4) .. ($(c)+(.2,0)$)
      -- ($(b)+(.2,0)$) to[out=90,in=90,looseness=1.5] ($(a)+(-.2,0)$)
      -- (0,.8) -- (-.8,2.5) -- (-1.3,2.5);
      \path[dotsG] (.8,-4.5) rectangle (1.2,-3.5);
    \draw[transf] (.8,-3.5) -- +(0.4,0);
  \end{tikzpicture}
  \quad\raisebox{2cm}{$=$}\quad
  \begin{tikzpicture}[scale=.7]
    \node[vert] (f') at (0,0) {$f$};
    \draw (f') -- node [ed] {$Q$} ++(-1,2.5) -- ++(0,1);
    \draw (f') -- node [ed,near end] {$P$} ++(1,-3.5);
    \draw (f') to[rrel,out=75,in=-90] node [ed] {$M$} (0.4,1.5) node (a) {}
    arc (180:0:0.4) coordinate (b)
    -- ++(0,-2) coordinate (c)
    to[rrel,out=-90,in=90] (-2.5,-2.3) node (d) {}
    arc (-180:0:0.4) coordinate (e)
    to[out=90,in=-100,looseness=0.5] (f');
      \path[dotsG] (-1.3,2.5) -- (-.4,.5) -- (-.3,-.5) -- ($(e)+(-.2,0)$)
      to[out=-90,in=-90,looseness=1] ($(d)+(.2,0)$)
      -- ($(d)+(-.2,0)$) to[out=-90,in=-90,looseness=1.5]
      ($(e)+(.2,0)$) -- (0,-.7) -- (.8,-3.5) -- (1.2,-3.5) --
      (.3,-.5) -- (.4,.5) -- ($(a)+(.2,0)$)
      to[out=90,in=90,looseness=1] ($(b)+(-.2,0)$)
      -- ($(c)+(-.2,0)$) .. controls +(0,-1.1) and +(0,1.4) .. ($(d)+(-.2,0)$)
      -- ($(d)+(.2,0)$) .. controls +(0,1.1) and +(0,-1.4) .. ($(c)+(.2,0)$)
      -- ($(b)+(.2,0)$) to[out=90,in=90,looseness=1.5] ($(a)+(-.2,0)$)
      -- (0,.8) -- (-.8,2.5) -- (-1.3,2.5);
      \path[dotsF] (-.8,2.5) rectangle (-1.3,3.5);
    \draw[transf] (-.8,2.5) -- +(-0.5,0);
  \end{tikzpicture}}
  \caption{2-functoriality of traces}
\end{figure}
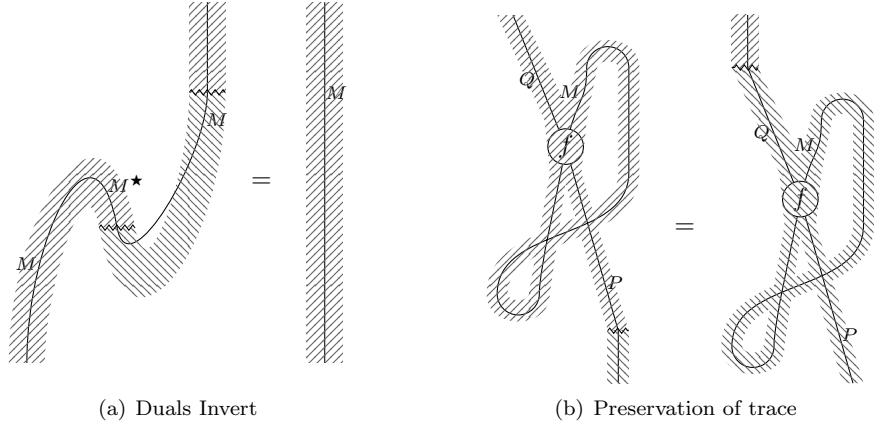

\begin{eg}\label{eg:intrat-lefschetz}
Extension of scalars along the inclusion $\iota\colon \bbZ\rightarrow \bbQ$
defines a natural transformation from the functor
$(H_*(-;\bbZ)/\mathrm{Torsion})\otimes \bbQ$ to the functor $H_*(-;\bbQ)$.
(We quotient $H_*(-;\bbZ)$ by torsion in order to make it strong symmetric monoidal.)
\autoref{thm:trans-pres-tr} then implies the familiar fact that the Lefschetz number computed using 
$H_*(-;\bbZ)/\mathrm{Torsion}$ is the same as that
computed using $H_*(-;\bbQ)$.
An analogous argument works for Lefschetz numbers computed at the level of chain complexes, in which case quotienting by torsion is no longer necessary since the chain complexes consist of free abelian groups.
\end{eg}

\section{Bicategories}
\label{sec:bicategories}

As suggested in the introduction, sometimes we also need a notion of
trace in non-commutative
situations, where we don't even have a monoidal category, let alone a
symmetric one.  For example, if $R$ is a non-commutative ring, then
there is no monoidal structure on the category of $R$-modules.  What
there is, however, is a tensor product of a right $R$-module with a
left $R$-module, which is a special case of the tensor product of
bimodules.
An appropriate categorical context for this is a \emph{bicategory}, originally due to~\cite{benabou:bicats}.

\begin{defn}
  A \textbf{bicategory} \sB\ consists of the following.
  \begin{itemize}
  \item A collection of \emph{objects} or \emph{0-cells} $R,S,T,\dots$.
  \item For each pair of objects, a category $\sB(R,S)$.
  \item For each object, a \emph{unit} $U_R\in \sB(R,R)$.
  \item For each triple of objects, a \emph{composition} functor
    \[\odot\maps \sB(R,S)\times\sB(S,T)\too \sB(R,T).\]
  \item Natural isomorphisms
    \begin{align*}
      \fa\maps M\odot (N\odot P) &\too[\iso] (M\odot N)\odot P\\
      \fl\maps U_R \odot M &\too[\iso] M\\
      \fr\maps M\odot U_S &\too[\iso] M
    \end{align*}
    satisfying the same coherence axioms as for a monoidal category.
  \end{itemize}
\end{defn}

We call the objects of $\sB(R,S)$ \emph{1-cells} and its morphisms \emph{2-cells}.
We regard 1-cells and 2-cells in a bicategory as analogous to the objects and morphisms in a monoidal category, respectively, with the 0-cells playing a ``bookkeeping'' role informing us which 1-cells can be tensored with which others.
We often write a 1-cell $M\in\sB(R,S)$ as $M\maps R\hto S$ to distinguish it from a 2-cell $f\maps M\to N$.

\begin{eg}
  There is a bicategory \calMod\ whose objects are
  not-necessarily-com\-mutative rings, and such that $\calMod(R,S)$ is
  the category of $R$-$S$-bimodules.  The unit $U_R$ is $R$ regarded
  as an $R$-$R$-bimodule, and the composition $\odot$ is the tensor
  product of bimodules, $M\odot N = M\ten_S N$.
\end{eg}

\begin{eg}
  If \bC\ is a monoidal category, we have a bicategory $\bccat{\bC}$ with
  one object $\star$ and $(\bccat{\bC})(\star,\star)=\bC$, $U_\star = I$, and
  $\odot=\ten$.
  In this sense, one can think of a bicategory as ``a monoidal category with many objects''.
\end{eg}

\begin{rmk}
  Our naming convention for bicategories is a little unusual: to
  clarify the discussion of later examples, we have chosen to use
  names which indicate both the objects \emph{and} the 1-cells.  Thus,
  for example, the name \calMod, which we read as ``modules over
  rings,'' indicates that the objects are rings and the 1-cells are
  modules, while the name $\bccat{\bC}$ indicates that there is a unique
  object $\star$ and the 1-cells are the objects of \bC.
\end{rmk}

\begin{rmk}\label{rmk:comp-order}
  We have chosen to write the composite of 1-cells $M\maps R\hto S$
  and $N\maps S\hto T$ in a bicategory as $M\odot N\maps R\hto T$, rather than $N\odot
  M$.  This is called \emph{diagrammatic order} and makes sense for
  bicategories such as \calMod; if we had chosen the other order we
  would have to define $\calMod(R,S)$ to be the category of
  $S$-$R$-bimodules in order to have $M\odot N = M\ten_S N$.

  There are other sorts of bicategories, however, for which the
  opposite choice makes more sense.  For instance, there is a
  bicategory \Cat\ whose objects are categories, whose 1-cells are
  functors, and whose 2-cells are natural transformations; in this
  case it makes more sense to write the composite of functors $F\maps
  A\to B$ and $G\maps B\to C$ as $G \circ F\maps A\to C$.
\end{rmk}

Finally, string diagrams in a bicategory are also obtained by Poincar\'e duality, but one dimension up.
Now 0-cells are represented by 2-dimensional regions (in color or shading, depending on whether the reader is fortunate enough to be reading this paper in color), 1-cells are represented by strings, and 2-cells are represented by vertices; see \autoref{fig:string-bicat}.
Note that since 1-cells and 2-cells in a bicategory are analogous to objects and morphisms in a monoidal category, respectively, we can regard bicategorical string diagrams as obtained from those for monoidal categories by adding 0-cell labels to the 2-dimensional regions.
However, the strings can no longer be crossed over each other, since there is no ``symmetry'' isomorphism in a bicategory.

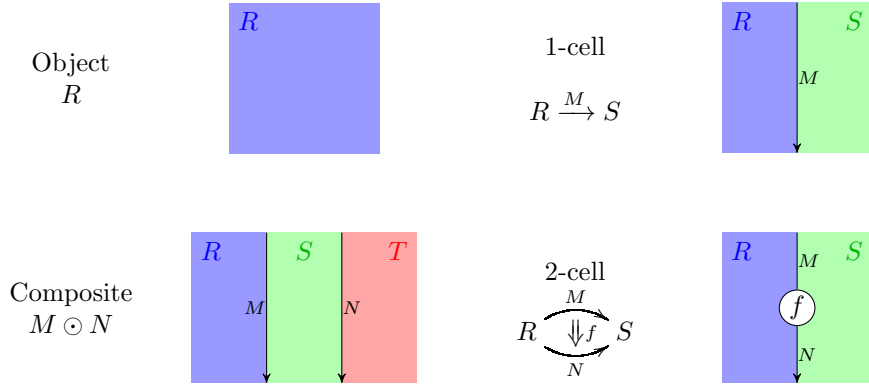
\begin{figure}[tb]
  \centering
  \begin{tabular}{m{17mm}m{38mm}m{27mm}m{25mm}}
    \begin{center}Object\\$R$\end{center} &
    \begin{tikzcenter}
      \fill[bluefill] (0,0) rectangle (2,2);
      \node[anchor=north west,blue] at (0,2) {$R$};
    \end{tikzcenter}
    &
    \begin{center}1-cell\\\ \\$R\xrightarrow{M} S$\end{center} &
    \begin{tikzcenter}
      \fill[bluefill] (0,0) rectangle (1,2);
      \fill[greenfill] (1,0) rectangle (2,2);
      \node[anchor=north west,blue] at (0,2) {$R$};
      \node[anchor=north east,green!70!black] at (2,2) {$S$};
      \draw[->] (1,2) -- node[ed] {$M$} +(0,-2);
    \end{tikzcenter}
    \\\\
    \begin{center}Composite\\$M\odot N$\end{center} &
    \begin{tikzcenter}
      \fill[bluefill] (0,0) rectangle (1,2);
      \fill[greenfill] (1,0) rectangle (2,2);
      \fill[redfill] (2,0) rectangle (3,2);
      \node[anchor=north west,blue] at (0,2) {$R$};
      \node[anchor=north,green!70!black] at (1.5,2) {$S$};
      \node[anchor=north east,red] at (3,2) {$T$};
      \draw[->] (1,2) -- node[ed,swap] {$M$} +(0,-2);
      \draw[->] (2,2) -- node[ed] {$N$} +(0,-2);
    \end{tikzcenter}
    &
    \begin{center}2-cell\\$\xymatrix{R\rtwocell^{M}_N{f} & S}$\end{center} &
    \begin{tikzcenter}
      \fill[bluefill] (0,0) rectangle (1,2);
      \fill[greenfill] (1,0) rectangle (2,2);
      \node[anchor=north west,blue] at (0,2) {$R$};
      \node[anchor=north east,green!70!black] at (2,2) {$S$};
      \node[fill=white,draw,circle,inner sep=1pt] (f) at (1,1) {$f$};
      \draw[->] (1,2) -- node[ed] {$M$} (f) -- node[ed] {$N$} +(0,-1);
    \end{tikzcenter}
  \end{tabular}
  \caption{String diagrams for bicategories}
  \label{fig:string-bicat}
\end{figure}

\section{Shadows}
\label{sec:shadows}

We will give further examples of bicategories shortly, but first we introduce the additional structure that we will need in order to define traces.
Extra structure is necessary because in defining the symmetric monoidal trace, we used the symmetry isomorphism $M\ten \rdual{M} \iso \rdual{M}\ten M$, whereas for 1-cells $M\maps R\hto S$ and $N\maps S\hto R$ in a bicategory, we cannot even ask whether $M\odot N$ and $N\odot M$ are isomorphic, since they are objects of different categories: one is a 1-cell $R\hto R$ and the other is a 1-cell $S\hto S$.
For example, if $R$ is a noncommutative ring, $M$ is a left $R$-module, and $N$ is a right $R$-module, then $M\odot N$ is an $R$-$R$-bimodule, while $N\odot M$ is just an abelian group.

In this latter example, there is a naive way to compare the two: if we quotient out the $R$-$R$-bimodule structure on $M\odot N$, we obtain an abelian group, which is in fact isomorphic to $N\odot M$.
This approach turns out to be surprisingly effective, and moreover many bicategories admit a similar sort of ``quotienting'' operation, as we will see.
Thus, we encapsulate the important properties of such an operation in the following abstract definition.

\begin{defn}[\cite{kate:traces}]\label{def:shadow}
  Let \sB\ be a bicategory.  A \textbf{shadow functor} for \sB\
  consists of functors
  \[\bigsh{-}\maps \sB(R,R) \to \bT\]
  for each object $R$ of \sB\ and some fixed category \bT, equipped
  with a natural isomorphism
  \[\theta\maps \sh{M\odot N}\too[\iso] \sh{N\odot M}\]
  for $M\maps R\hto S$ and $N\maps S\hto R$ such that the following
  diagrams commute whenever they make sense:
  \[\xymatrix{\bigsh{(M\odot N)\odot P} \ar[r]^\theta \ar[d]_{\sh{\fa}} &
    \bigsh{P \odot (M\odot N)} \ar[r]^{\sh{\fa}} &
    \bigsh{(P\odot M) \odot N}\\
    \bigsh{M\odot (N\odot P)} \ar[r]^\theta & \bigsh{(N\odot P)
      \odot M} \ar[r]^{\sh{\fa}} & \bigsh{N\odot (P\odot
      M)}\ar[u]_\theta }\]
  \[\xymatrix{\bigsh{M\odot U_R} \ar[r]^\theta \ar[dr]_{\sh{\fr}} &
    \bigsh{U_R\odot M} \ar[d]^{\sh{\fl}} \ar[r]^\theta &
    \bigsh{M\odot U_R} \ar[dl]^{\sh{\fr}}\\
    &\bigsh{M}}\]
\end{defn}

\begin{rmk}\label{thm:not-braiding}
  The above hexagon is \emph{not} one of the hexagon axioms for a
  braided monoidal category.  In fact, if $\theta$ were viewed as a
  ``braiding'', the two sides of this hexagon would describe unequal
  braids.  The shadow axioms describe a ``cyclic'' operation, rather
  than a ``linear'' one that happens to have transpositions.
\end{rmk}

It may seem that we should require $\theta^2=1$, but in fact this is
automatic:

\begin{prop}\label{shadow-symm}
  If $\sh{-}$ is a shadow functor on \sB, then the composite
  \[\xymatrix{\bigsh{M\odot N} \ar[r]^\theta &
    \bigsh{N\odot M} \ar[r]^\theta &
    \bigsh{M\odot N}}\]
  is the identity.
\end{prop}
\begin{proof}
  Let $M\maps R\hto S$ and $N\maps S\hto R$, and set $P=U_R$ in the
  hexagon axiom for a shadow functor.  Naturality of $\theta$, the
  unit axioms for a shadow functor, and the axioms relating $\fl,\fr$
  to $\fa$ in a bicategory reduce the hexagon to the desired
  statement.
\end{proof}

\begin{rmk}
  In the proof of \autoref{shadow-symm} we didn't use the assumption 
  that $\theta$ is an isomorphism.  Thus, we could just as
  well have dispensed with that assumption in the definition.
  It is also worth noting that if we assume \autoref{shadow-symm}, then we can derive either of the unit axioms from the other.
\end{rmk}

\begin{eg}
  Let \bC\ be a symmetric monoidal category; then its identity functor
  gives $\bccat{\bC}$ a canonical shadow functor with $\bT=\bC$.
  The isomorphism $\theta$ is induced by the symmetry of \bC.
  By \autoref{thm:not-braiding}, it does not suffice for \bC\ to be braided.
\end{eg}

\begin{eg}\label{eg:bimod-shadow}
  We define the shadow of a bimodule $M\maps R\hto R$ in \calMod\ to
  be the coequalizer
  \[R\ten M \toto M \to \sh{M},
  \]
  where the two parallel maps are the left and right actions of $R$.
  Here \bT\ is abelian groups and $\sh{M}$ is the abelian group obtained from $M$ by forcing the left and right actions of $R$ to be equal; it might be called an ``underived version of Hochschild homology''.
  The isomorphism $\theta$ is obvious.
\end{eg}

We will consider more examples in \S\ref{sec:examples} after defining traces.



Finally, given that a shadow is a fundamentally cyclic operation,
it is natural to represent it by closing up planar bicategorical string diagrams into a \emph{cylinder}, thereby allowing strings to migrate cyclically around the back of the cylinder to the other side.
This is illustrated in \autoref{fig:string-shadow}.
In Appendix~\ref{sec:string-diagrams} we will extend the fundamental theorem of Joyal and Street to bicategories with shadows, by defining the ``value'' of such a labeled cylindrical diagram and proving that it is invariant under deformation.

\begin{figure}[tb]
  \centering
  \subfigure[Shadow $\protect\sh{M}$]{
    \begin{tikzpicture}
      \bgcylinder{0,0}{1.7}{1}{.3}{blue}{blue}
      \node[anchor=south west,blue] at (dl) {$R$};
      \draw (top) -- node[ed] {$M$} (bot);
    \end{tikzpicture}}
  \hspace{2cm}
  \subfigure[Cyclicity $\theta_{M,N}$]{
    \begin{tikzpicture}
      \bgcylinder{0,0}{2.3}{1.2}{.3}{blue}{green}
      \begin{pgfonlayer}{foreground}
        \drawtheta{0,1}{.5}{th}
      \end{pgfonlayer}
      \node[anchor=south west,blue] at (dl) {$R$};
      \filldraw[greenfill] (thL) -- ++(-0.1,0) -- (ul') --
      ($(ul)!.4!(ur)$) to[out=-90,in=90] node[ed,near end] {$M$} (thL);
      \filldraw[greenfill] (thR) to[out=-90,in=70]
      ($(dr' |- bot')!.7!(bot')$)  -- ($(dl' |- bot')!.7!(bot')$)
      -- node[ed,swap,near start] {$N$} ($(ul)!.7!(ur)$)
      -- (ur') -- (thR);
      \node[anchor=north east,green!50!black] at ($(ur)!.2!(dr)$) {$S$};
    \end{tikzpicture}}
  \caption{String diagrams for shadows}
  \label{fig:string-shadow}
\end{figure}
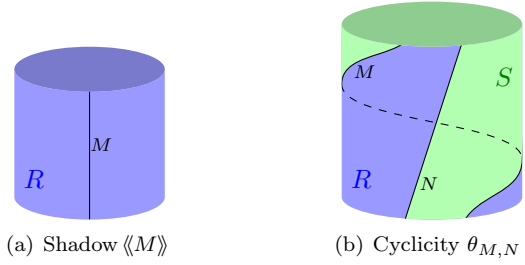

\begin{rmk}\label{rmk:noemb}
  The boundary circles of these cylinders should not be regarded as fixed, but are free to rotate as we deform the diagram.
  For instance, this means that the two pictures shown in \autoref{fig:noemb} should be regarded as the same.
  (Technically, we do have to distinguish a ``basepoint'' or ``cut point'' on the top and bottom circles in order to assign a well-specified value, since $\sh{M\odot N}$ is rarely \emph{equal} to $\sh{N\odot M}$.
  These basepoints then have to rotate along with the corresponding boundary circle.
  This will be made precise in Appendix~\ref{sec:string-diagrams}.
  When drawing pictures in the rest of the paper, we will always assume that these basepoints are in the back.)
\end{rmk}

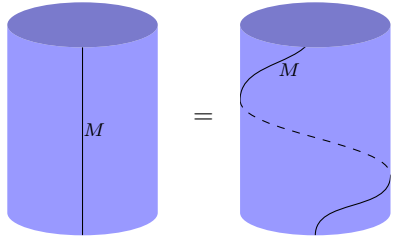
\begin{figure}
  \centering
  \begin{tikzpicture}
    \bgcylinder{0,0}{2.5}{1}{.3}{blue}{blue}
    \draw (top) -- node[ed] {$M$} (bot);
  \end{tikzpicture}
  \quad\raisebox{1.5cm}{$=$}\quad
  \begin{tikzpicture}
    \bgcylinder{0,0}{2.5}{1}{.3}{blue}{blue}
    \drawtheta{1,0}{1.5}{theta}
    \draw (top) to[out=-90,in=90] node[ed] {$M$} (thetaL);
    \draw (thetaR) to[out=-90,in=90] (bot);
  \end{tikzpicture}
  \caption{Cylinder boundaries are not fixed}
  \label{fig:noemb}
\end{figure}

\suppressfloats[t]
\section{Duality and trace}
\label{sec:bicat-traces}

We say that a 1-cell $M\maps R\hto S$ in a bicategory is \textbf{right
  dualizable} if there is a 1-cell $\rdual{M}\maps S\hto R$, called
its \textbf{right dual}, and evaluation and coevaluation 2-cells
$\eta\maps U_R \to M\odot \rdual{M}$ and $\ep\maps \rdual{M}\odot M\to
U_S$ satisfying the triangle identities.
We then also say that $(M,\rdual{M})$ is a
\textbf{dual pair}, that $\rdual{M}$ is \textbf{left dualizable}, and
that $M$ is its \textbf{left dual}.  See \autoref{fig:string-bicat-coeval}
for the string diagrams for coevaluation and evaluation.

\begin{figure}[tb]
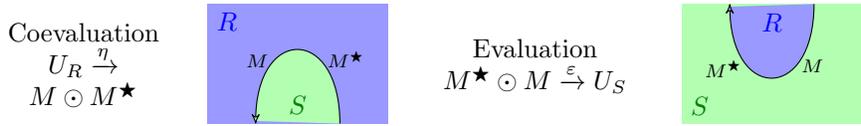

  \centering
  \begin{tabular}{m{22mm}m{28mm}m{28mm}m{28mm}}
    \begin{center}Coevaluation\\$U_R \xrightarrow{\eta} M \odot \rdual{M}$\end{center} &
    \begin{tikzcenter}[scale=.8]
      \fill[bluefill] (0,0) rectangle (3,2);
      \draw[<-,greenfill] (.8,0) to[out=90,in=90,looseness=3]
      node [ed,near start] {$M$} node[ed,near end] {$\rdual{M}$} (2.2,0);
      \node[anchor=north west,blue] at (0,2) {$R$};
      \node[anchor=south,green!50!black] at (1.5,0) {$S$};
    \end{tikzcenter}
    &
    \begin{center}Evaluation\\$\rdual{M}\odot M \xrightarrow{\ep} U_S$\end{center} &
    \begin{tikzcenter}[scale=.8]
      \fill[greenfill] (0,2) rectangle (3,0);
      \draw[<-,bluefill] (.8,2) to[out=-90,in=-90,looseness=3]
      node [ed,near end,swap] {$M$} node[ed,near start,swap] {$\rdual{M}$} (2.2,2);
      \node[anchor=south west,green!50!black] at (0,0) {$S$};
      \node[anchor=north,blue] at (1.5,2) {$R$};
    \end{tikzcenter}
  \end{tabular}
  \caption{Coevaluation and evaluation in a bicategory}
  \label{fig:string-bicat-coeval}
\end{figure}

In the bicategorical context, duals are also
frequently called \emph{adjoints}, since in the bicategory \Cat\ (see
\autoref{rmk:comp-order}) a dual pair is just an adjoint pair of
functors.
As in the monoidal case, there are equivalent characterizations of
duals, especially when the bicategory \sB\ is closed; see, for
instance,~\cite[16.4]{maysig:pht}.  Note that the notion is now asymmetric:
right duals and left duals are different.

When $M$ and $N$ are right dualizable, every 2-cell $f\maps Q\odot M\rightarrow N\odot P$ has a \textbf{mate} $\rdual{f}\colon\rdual{N}\odot Q\rightarrow P\odot\rdual{M}$, defined as the composite
\[
\rdual{N}\odot Q \xto{1\odot\eta} \rdual{N}\odot Q\odot M \odot \rdual{M}
\xto{1\odot f \odot 1} \rdual{N}\odot N \odot P \odot \rdual{M}
\xto{\ep\odot 1} P\odot \rdual{M}.
\]
This operation is pictured graphically in Figure~\ref{fig:mate}.

There is an evident dual construction for left dualizable objects, and the triangle identities show that the two are inverses.
Therefore, giving a 2-cell $Q\odot M\rightarrow N\odot P$ is equivalent to giving a 2-cell $\rdual{N}\odot Q\rightarrow P\odot\rdual{M}$.
For this reason, it is justified to draw \emph{horizontal} strings labeled by dual pairs, as in Figure~\ref{fig:dualhoriz}, since they can be ``tipped'' up or down to represent a 2-cell or its mate, as needed.
(We could also consider such pictures as taking place in a \emph{double category} whose vertical arrows are dual pairs, such as is used in~\cite{ks:r2cats} to describe naturality properties of mates.)

\begin{figure}[bt]
  \centering
  \subfigure[The mate of a 2-cell]{\label{fig:mate}
  \begin{tikzpicture}[scale=0.8]
    \clip (0,0) rectangle (6,4);
    \fill[bluefill] (0,0) rectangle (3,4);
    \fill[greenfill] (3,0) rectangle (6,4);
    \begin{pgfonlayer}{foreground}
      \node[vert] (f) at (3,2) {$f$};
    \end{pgfonlayer}
    \draw[redfill] (f) to[out=60,in=-90,looseness=0.7] node[ed,swap] {$M$} (4,3)
    arc (180:0:0.4) -- node[ed,near end] {$\rdual{M}$}  +(0,-3.1)
    -| (f);
    \draw[yellowfill] (f) to[out=-150,in=90,looseness=0.7] node[ed,swap] {$N$} (2,1)
    arc (0:-180:0.4) -- node[ed,near end] {$\rdual{N}$} +(0,3.1)
    -| (f);
    \draw (3,4) -- node[ed,swap] {$Q$} (f);
    \draw (f) -- node[ed] {$P$} (3,0);
  \end{tikzpicture}}
  \hspace{1.5cm}
  \subfigure[Drawing dual pairs]{\label{fig:dualhoriz}
  \begin{tikzpicture}[scale=1.5]
    \begin{pgfonlayer}{foreground}
      \node (phi) [vert] at (0,0)
      {$f$};
    \end{pgfonlayer}
    \clip (-1.4,-.9) rectangle (1.4,.9);
    \filldraw[redfill] (phi) -- +(0,-1) -- +(-1.5,-1) -- +(-1.5,0) -- (phi);
    \filldraw[greenfill] (phi) -- +(0,1) -- +(1.5,1) -- +(1.5,0) -- (phi);
    \filldraw[bluefill] (phi) -- node [ed,swap] {$Q$} +(0,1) -- +(-1.5,1)
    -- +(-1.5,0) -- node [ed] {$(N,\rdual{N})$} (phi);
    \filldraw[yellowfill] (phi) -- node [ed] {$P$} +(0,-1) --
    +(1.5,-1) -- +(1.5,0) -- node [ed,swap] {$(M,\rdual{M})$} (phi);
  \end{tikzpicture}}
\caption{Mates in a bicategory}
\end{figure}
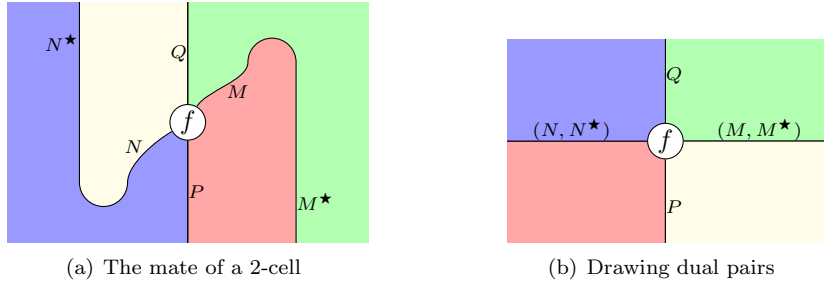


\begin{defn}
  Let \sB\ be a bicategory with a shadow functor and $M$ a dualizable 1-cell of \sB.
  The \textbf{trace} of a 2-cell $f\maps Q\odot M\rightarrow M\odot P$ is the composite:
  \[\sh{Q}\xto{\sh{\id\odot\eta}}
  \sh{Q\odot M\odot \rdual{M}} \xto{\sh{f\odot\id}}
  \sh{M\odot P\odot \rdual{M}} \too[\theta]
  \sh{\rdual{M}\odot M\odot P} \xto{\sh{\ep \odot \id}}
  \sh{P}.
  \]
\end{defn}

The trace is an arrow from $\sh{Q}$ to $\sh{P}$ in the target category \bT\ for the shadow, and is independent of the choice of $\rdual{M}$, \eta, and \ep.
A similar, but not identical, definition works when $M$ is left dualizable instead.
A string diagram picture of this trace is shown in \autoref{fig:bicat-trace}, along with an equivalent, more concise version using the above convention for drawing dual-pair strings horizontally.

Of course, if we take $Q=U_{R}$ and $P=U_{S}$, we obtain the trace of an endomorphism.
The frequency with which this case occurs implies that the shadows of unit 1-cells are particularly important.
We refer to the shadow of $U_R$ as ``the shadow of $R$'' and write $\sh{R}=\sh{U_R}$.

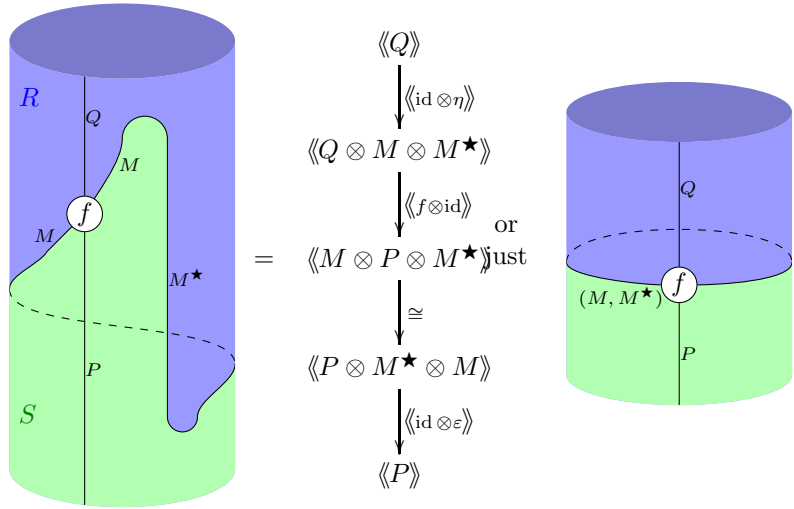
\begin{figure}[bp]
  \centering
  \begin{tabular}{m{29mm}m{3mm}m{19mm}m{9mm}m{25mm}}
    \begin{tikzpicture}
      \bgcylinder{0,-1.9}{5.7}{1.5}{0.5}{blue}{blue}
      \node[vert] (f) at (1,1.5) {$f$};
      \drawtheta{f}{-1}{theta}
      \begin{pgfonlayer}{background}
        \draw[greenfill,looseness=0.7]
        (thetaL) to[rrel,out=90,in=-120] (0.5,0.5)
        -- node [ed,near start] {$M$} (f.center)
        to[rrel,out=60,in=-90] node [ed,swap,near end] {$M$} (.5,1)
        arc (180:0:0.3)
        -- node [ed] {$\rdual{M}$} ($(f)+(1.1,-2.7)$)
        arc (-180:0:0.2)
        to [out=90,in=-90] (thetaR)
        -- ++(0.1,0)
        -- (dr' |- bot') -- (dl' |- bot') -- ($(thetaL)+(-0.1,0)$) -- (thetaL);
      \end{pgfonlayer}
      \draw (f |- top) -- node[ed] {$Q$} (f);
      \draw (f) -- node[ed] {$P$} (f |- bot);
      \path (ul) +(0,-0.5) node[blue,anchor=north west] {$R$};
      \path (dl) +(0,.5) node[green!50!black,anchor=south west] {$S$};
    \end{tikzpicture}
    & $=$ &
    $\xymatrix{\sh{Q} \ar[d]^{\sh{\id\ten\eta}} \\
      \sh{Q\ten M\ten \rdual{M}} \ar[d]^{\sh{f\ten \id}} \\
      \sh{M\ten P \ten \rdual{M}} \ar[d]^{\iso} \\
      \sh{P\ten \rdual{M}\ten M} \ar[d]^{\sh{\id\ten\ep}} \\
      \sh{P}}$
    & \begin{center}or\\just\end{center} &
    \begin{tikzcenter}
      \bgcylinder{0,0}{3.5}{1.5}{0.4}{blue}{blue}
      \node[vert] (f) at (1.5,1.2) {$f$};
      \elltheta{f}{.3}{theta}
      \begin{pgfonlayer}{background}
        \draw[greenfill] (f)
        to[out=180,in=-90,looseness=0.4] node[ed,pos=.4,below=2pt] {$(M,\rdual{M})$} (thetaL)
        -- ++(-0.1,0)
        -- (dl' |- bot') -- (dr' |- bot') -- ($(thetaR)+(0.1,0)$)
        -- (thetaR)
        to[out=-90,in=0,looseness=0.5] (f);
      \end{pgfonlayer}
      \draw (f |- top) -- node[ed] {$Q$} (f);
      \draw (f) -- node[ed] {$P$} (f |- bot);
    \end{tikzcenter}
  \end{tabular}
  \caption{The bicategorical trace}
  \label{fig:bicat-trace}
\end{figure}

\begin{eg}
  Of course, in $\bccat{\bC}$ the bicategorical trace reduces to the canonical symmetric monoidal trace.
\end{eg}

\begin{eg}\label{eg:equivar-trace}
  In \calMod, a bimodule $M\maps \bbZ\hto R$ is right dualizable when
  it is a finitely generated projective (right) $R$-module.  The trace
  of an endomorphism $f\maps M\to M$ is then a map from
  $\sh{\bbZ}=\bbZ$ to $\sh{R}$, which is determined by the image of
  $1$.  This element of $\sh{R}$ is known as the
  \emph{Hattori-Stallings trace} of $f$, see \cite{hattori,stallings}.

  More generally, if $\psi\maps R\rightarrow R$ is a ring homomorphism and
  $f\maps M\rightarrow M$ is a $\psi$-equivariant map (meaning that
  $f(mr) = f(m)\psi(r)$), we can view $f$ as a 2-cell $M\to M\odot
  R_\psi$ in \calMod.  Here $R_\psi$ denotes $R$ regarded as an
  $R$-$R$-bimodule with the right action twisted by \psi.  The trace
  of $f$ is then a map $\bbZ \iso \sh{\mathbb{Z}} \to \sh{R_\psi}$ in
  $\mathbf{Ab}$, or equivalently an element of $\sh{R_\psi}$.
  (Explicitly, $\sh{R_\psi}$ is the abelian group obtained from $R$ by
  quotienting by $r s\sim s \psi(r)$ for all $r$.)

  Note that even when $R$ is commutative, this latter type of trace cannot be
  expressed without bicategorical technology, since $R_\psi$ is
  unavoidably an $R$-$R$-bimodule, and the category of
  $R$-$R$-bimodules is not symmetric (or even braided) monoidal.
\end{eg}

\begin{rmk}\label{thm:sh-cat-trace}
  The categorically inclined reader will observe that a shadow is a ``categorified trace''.
  Just as trace is a cyclic function on endomorphisms in a 1-category, shadow is a cyclic \emph{functor} on endo-1-cells in a \emph{bi}category.
  Thus, to define traces in a bicategory, the bicategory must be equipped with a categorified trace, just as (for example) to define monoids in a category, the category must be equipped with a monoidal structure.
  This is a version of the \emph{microcosm principle} of~\cite{bd:hda3}.

  Since there is a canonical (in fact, unique) trace in any symmetric monoidal category in which all objects are dualizable, it is natural to think of constructing a canonical shadow in an analogous way for any symmetric monoidal bicategory with duals for objects (see~\cite{ds:monbi-hopfagbd}).
  In fact, all of our examples of shadows on bicategories can be constructed in this way, and our cylindrical string diagrams can thereby be identified with a fragment of the expected ``surface diagrams'' that apply to monoidal bicategories (see~\cite{street:ldtop-hocat}).
  However, in applications we prefer to avoid this abstract point of view, since a symmetric monoidal bicategory with duals for objects is quite a complicated object, while the shadows arising in practice have much simpler descriptions.
\end{rmk}

\section{Examples of shadows and traces}
\label{sec:examples}

In this section we will consider several different examples of bicategorical shadows and traces.
We begin with a generalization of \autoref{eg:equivar-trace}.

\begin{eg}
  In \autoref{eg:equivar-trace} we considered right $R$-modules as 1-cells $\bbZ\hto R$ in \calMod.
  More generally, an $S$-$R$-bimodule $M$ is right dualizable in
  \calMod\ if and only if it is finitely generated and projective as a
  right $R$-module.  The Euler characteristic of such an $M$ (that is,
  the trace of its identity map) is the map $\sh{S}\to\sh{R}$ which
  sends each $s\in S$ to the Hattori-Stallings trace of the $R$-module
  homomorphism $(s\cdot -)\maps M\to M$ (which descends to $\sh{S}$
  because the Hattori-Stallings trace is cyclic).  More generally, the
  trace of $f\maps M\to M$ is the map $\sh{S}\to\sh{R}$ determined by
  sending $s\in S$ to the Hattori-Stallings trace of $m\mapsto
  f(sm)=s\cdot f(m)$.

  In particular, let $G$ be a group, $k$ a field, and $V$ a
  finite-dimensional left $k G$-module (that is, a representation of
  $G$).  If we regard $V$ as a 1-cell $k G\hto k$ in \calMod, then it
  is right dualizable, so it has an Euler characteristic, which is a
  map $\sh{k G}\rightarrow k$.  It is easy to check that $\sh{k G}$ is
  the $k$-vector space generated by the conjugacy classes of $G$, and
  this Euler characteristic is essentially the character of the
  representation $V$.
\end{eg}

The following two examples are somewhat degenerate, but can be useful.

\begin{eg}\label{eg:span-traces}
  Let \bS\ be a category with pullbacks (a.k.a.\ fiber products) and define a bicategory
  $\calSpan(S)$ whose objects are those of \bS\ and whose 1-cells
  $R\hto S$ are diagrams $R\ot M\to S$ in \bS; these are often called
  \emph{spans} or \emph{correspondences}.  Composition is by
  pullback.  We define the shadow of a span $R\ot M\to R$ to be the
  object $P$ in the pullback square
  \[\xymatrix{P \ar[r]\ar[d] & M \ar[d]\\
    R\ar[r]^<>(.5){\Delta} & R\times R.}
  \]
  This defines a shadow on $\calSpan(S)$ with values in \bS.  If
  $f\maps R\to R$ is an arrow in \bS, the shadow of the corresponding
  span $R\oot[\id] R \too[f] R$ (the ``graph'' of $f$) is the equalizer of $f$ and $\id_R$;
  that is, the ``object of fixed points'' of $f$.  In particular, for an
  object $R$ of \bS, we have $\sh{R} = R$.

  In $\calSpan(S)$, a 1-cell $R\oot[h] M\too[g] S$ is right
  dualizable if and only if $h$ is an isomorphism.  If
  $R\oot[\id] R\too[g] S$ is such a dualizable 1-cell, an endomorphism
  of it is a map $f\maps R\to R$ such that $gf=g$.  Its trace is precisely
  $gf=g\maps \sh{R} = R \to S = \sh{S}$.
\end{eg}

\begin{eg}\label{eg:ncob-bicat}
  There is a bicategory \calnCob\ whose objects are closed
  $(n-1)$-dimensional manifolds, whose 1-cells are $n$-dimensional
  cobordisms, and whose 2-cells are boundary-preserving
  diffeomorphisms.  To make the associativity and unit isomorphisms
  coherent, we have to include ``thin'' cobordisms, and also give
  ``collars'' to the thick ones.  A cobordism from $R$ to $R$ is an
  $n$-manifold $M$ whose boundary is $R\sqcup R\op$; we define its
  shadow $\sh{M}$ to be the result of gluing together these two copies
  of $R$.

  There are no interesting bicategorical traces in
  \calnCob, since all the 2-cells are isomorphisms.  (In essence, here
  the shadow itself ``is'' the interesting trace.)
\end{eg}

Many bicategories of interest are obtained as ``homotopy bicategories'' of other bicategories, by which we mean that we invert some class of \emph{2-cells} considered as ``weak equivalences.''
The theory of homotopy bicategories is not well-developed, but since our goal is only to give examples, we will gloss over all the technicalities.

\begin{eg}
  Let \calCh\ be the bicategory whose objects are rings (not necessarily commutative) and whose 1-cells are chain complexes of bimodules.
  We define shadows analogously to \calMod, introducing appropriate signs in the definition of $\theta$.
  Here \bT\ is the category $\bCh{\mathbb{Z}}$ of chain complexes of abelian groups.

  We can now construct a bicategory $\Ho(\calCh)$, whose objects are the same as those as
  \calCh, but in which $\Ho(\calCh)(A,B)$ is the
  \emph{derived} category of $A\ten B\op$ (that is, we invert the
  quasi-isomorphisms of chain complexes).  The composition $\odot$ is now
  the derived functor of the tensor product of bimodules, and the
  shadow of a 1-cell $M\maps A\hto A$ now \emph{is} the Hochschild
  homology $\mathrm{HH}(A;M)$ of $A$ with coefficients in $M$.  The target
  category \bT\ of the shadow is now $\Ho(\bCh{\mathbb{Z}})$.

  Similarly, we have a bicategory \calChDGA\ whose objects are DGAs rather than rings, and $\Ho(\calChDGA)$ is defined similarly.
  For technical reasons, we may sometimes want to restrict the objects of
  $\Ho(\calChDGA)$ to be DGAs which are cofibrant, at least as chain
  complexes if not in some model structure for DGAs.  We can also
  consider modules over ring spectra, in which case the shadow is
  topological Hochschild homology.
\end{eg}

\begin{eg}
  There is a bicategory $\bgptop$ whose objects are discrete groups,
  whose 1-cells $G\hto H$ are based topological spaces with a left
  action of $G$ and a right action of $H$, and whose 2-cells are
  equivariant maps.  The unit $U_G$ is $G_+$ regarded as a $G$-$G$
  space, and the composition of $M\maps G\hto H$ and $N\maps H\hto K$
  is the smash product over $H$, i.e.\ the coequalizer
  \[M\wedge H_+\wedge N\toto M\wedge N\to M\odot N.
  \]
  Likewise, the shadow of a $G$-$G$ space $M$ is the coequalizer
  \[G_+\wedge M\toto M\to \sh{M}
  \]
  where the two parallel maps are the left and right actions of $G$ on
  $M$.  The target category $\bT$ of the shadow is the category
  $\Top_*$ of based spaces.

  For a useful homotopy version of this example, we need to
  ``stabilize'' (i.e.\ move from spaces to spectra, as in
  \autoref{eg:stab-duality}) as well as pass to homotopy categories.
  This would result in a bicategory we refer to as $\Ho(\bgpsp)$, whose shadow would take values in the ordinary stable homotopy category $\Ho(\Sp)$.
  A careful foundational analysis leading to such a bicategory has not yet been done, but we can still use a naive sort of stabilization to obtain a notion of duality analogous to Spanier-Whitehead duality.
  That is, we say a 1-cell $M$ in \bgptop\ is \emph{$n$-dualizable} if we have maps
  \[G_+\wedge S^n\rightarrow M\odot N\hspace{1cm}
  N\odot M\rightarrow H_+\wedge S^n\]
  satisfying the usual relations up to equivariant stable homotopy.
  Once $\Ho(\bgpsp)$ is shown to exist, we expect its intrinsic notion of duality to be describable as $n$-duality; thus from now on we will refer informally to ``duality'' and ``trace'' in $\Ho(\bgpsp)$.

  \label{eg:gtop-duality} It turns out that very little is known about this sort of duality, but there are a couple examples that have interesting traces.
  On the one hand, if $G$ is finite and $M$ has a free $G$-action, then duality for $M$, regarded as a 1-cell $1\hto G$ in \hobgpsp, is equivalent to duality for $M$ in $\Ho(\gSp{G})$, the equivariant stable homotopy category for $G$; see \cite[8.6]{adams}.
  The corresponding traces in $\Ho(\gSp{G})$ and in $\hobgpsp$ are not identical, but the first is a direct summand of the second.
  Both have interpretations as equivariant fixed point indices, but the trace in $\hobgpsp$ also detects fixed orbits; see \cite{equiv}.

  On the other hand, if $G$ is discrete (possibly infinite) and $M$ is a finite free right $G$-CW complex, then $M\colon 1\hto G$ is dualizable in $\hobgpsp$.  This  is known as Ranicki duality; see \cite{ranicki}.
\label{eg:reidemeister-1}
  For example, if $X$ is a closed smooth manifold (or a compact ENR), its universal
  cover $\tilde{X}$ has a free action of $\pi_1(X)$
  and this equivariant space is Ranicki dualizable.

  Now suppose that $X$ is a space of this sort, and let $f\colon X\rightarrow X$ be a continuous endomorphism.
  For each choice of base point $\ast$ in $X$ and base path from $\ast$ to 
  $f(\ast)$, there is an induced  group homomorphism
  \[\psi\colon \pi_1(X)\rightarrow \pi_1(X)\] and a $\psi$-equivariant map 
  \[\tilde{f}\colon \tilde{X}\rightarrow \tilde{X}.\]
  As in \autoref{eg:equivar-trace}, we can regard $\tilde{f}$ as a
  2-cell $\tilde{X} \to \tilde{X} \odot \pi_1(X)_\psi$, where
  $\pi_1(X)_\psi$ denotes $\pi_1(X)$ regarded as a
  $\pi_1(X)$-$\pi_1(X)$-space with the right action twisted by $\psi$.
  The trace of $\tilde{f}$ is, by definition, an element of the zeroth stable homotopy group of the discrete set $\sh{\pi_1(X)_\psi}$, which is just the free abelian group $\bbZ\sh{\pi_1(X)_\psi}$.

  This trace is known as the \emph{Reidemeister trace}.  It is a
  refinement of the fixed-point index, which separates out
  contributions coming from different ``fixed-point classes''.  A
  classical description of the Reidemester trace can be found
  in~\cite{b:lefschetz,j:nielsen}; the bicategorical description is
  due to~\cite{kate:traces}.  In \S\ref{sec:funct-bicat}, we will see
  that the abstract proof of the Lefschetz fixed-point theorem from
  \autoref{eg:lefschetz} carries over directly to the Reidemeister
  trace when expressed bicategorically.  In fact, the Reidemester
  trace is refined enough to support a converse to the Lefschetz fixed
  point theorem, but that requires a deeper, more concrete, argument.
\end{eg}

\begin{eg}\label{eg:bicat-ex}\label{thm:fiberwise+cw-duality}
  The authors of~\cite{maysig:pht} define a bicategory whose objects are topological spaces, and in which a 1-cell $R\hto S$ is a parametrized spectrum over $R\times S$.
  We call this bicategory $\Ho(\calEx)$, since it is also the homotopy bicategory of a point-set-level bicategory \calEx.
  If $M\maps R\hto R$ is a spectrum parametrized over $R\times R$, we define its shadow to be $\sh{M} = r_!\Delta^*M$, where $\Delta^*$ denotes pullback along the diagonal $R\to R\times R$ and $r_!$ denotes pushforward along the map $R\to \star$ to the point.
  This defines a shadow on $\Ho(\calEx)$ landing in the ordinary stable homotopy category $\Ho(\Sp)$.

  Using the identifications $ (B\times \star)\cong B\cong (\star\times B)$, a spectrum $M$ over $B$ can also be regarded as a 1-cell $\hat{M}\colon B \hto \star$ or $\check{M}\colon \star\hto B$ in $\Ho(\calEx)$.
  The 1-cell $\hat{M}$ is right dualizable if and only if $M$ is dualizable in the symmetric monoidal category $\Ho(\bEx{B})$ of parametrized spectra over $B$.
  In this case $M$ is said to be \emph{fiberwise dualizable}, since the fiberwise suspension spectrum of a fibration is dualizable in this sense if and only if each \emph{fiber} is dualizable in $\Ho(\Sp)$; see \cite[15.1.1]{maysig:pht}.

  Right dualizability of $\check{M}$, called \emph{Costenoble-Waner duality}, is very different.
  For example, if the fiberwise sphere spectrum $S_B= \Sigma^{\infty}_{B,+}(B)$ is Costenoble-Waner dualizable, then $\Sip(B)$ is necessarily dualizable in $\Ho(\Sp)$; see \cite[18.1.6]{maysig:pht}.
  In other words, fiberwise duality
  contains information only about the \emph{fibers} of a parametrized
  space, while Costenoble-Waner duality also knows about the \emph{base} space.

The traces of endomorphisms of $M$, $\hat{M}$, and $\check{M}$ are also closely related.
We will study these relationships in a more general setting in~\cite{PS3}; here we mention only the simplest case of Euler characteristics.
If $M$ is fiberwise dualizable, then the Euler characteristics of $M$ (in $\Ho(\Sp_B)$) and $\hat{M}$ (in $\Ho(\calEx)$) are defined, and the following triangle commutes:
\[\xymatrix@C=1.5pc{ r_!S_B\ar[rr]^{r_!(\chi(M))} \ar[dr]_{\chi(\hat{M})} && r_!S_B\ar[dl]\\ & S}.\] 
where $S_B$ is the parametrized sphere spectrum over $B$, $S$ is the ordinary sphere spectrum, and the map $r_! S_B \to S$ is adjunct to the defining isomorphism $S_B \cong r^* S$.
On the other hand, if $M$ is Costenoble-Waner dualizable, then $r_!M$ is dualizable, so that the Euler characteristics of $\check{M}$ (in $\Ho(\calEx)$) and $r_!M$ (in $\Ho(\Sp)$) are defined, and the following triangle commutes:
\[\xymatrix{S \ar[rr]^-{\chi(\check{M})} \ar[dr]_{\chi(r_!M)} && r_!S_B \ar[dl]\\
  & S}.
\]
Note that in the first case, the Euler characteristic of $\hat{M}$ contains \emph{less} information than that of $M$, while in the second case, the Euler characteristic of $\check{M}$ contains \emph{more} information than that of $r_!M$.

  Both kinds of dualizable parametrized spectra often arise as fiberwise suspension spectra of parametrized spaces, i.e.\ by applying a functor from $\calSpan(Top)$ to \calEx.
  Thus, this is a bicategorical instance of \autoref{rmk:cartesian}, since $\calSpan(Top)$ is ``cartesian'' and has no interesting dual pairs itself.

Like \autoref{eg:gtop-duality}, traces of some parameterized endomorphisms have familiar 
fixed-point interpretations.  If $M$ is a parameterized space over $B$ whose 
suspension spectrum is fiberwise dualizable and $f\colon 
M\rightarrow M$ is a fiberwise endomorphism, the trace of the suspension of $f$ is the fiberwise fixed 
point index of $f$, \cite{d:index}.  If $B$ is a closed smooth manifold (or compact ENR), 
$S_B^0$ is Costenoble-Waner dualizable.  An endomorphism $f$ of $B$ induces
a ``twisted endomorphism" of $S_B^0$ and the trace of the twisted
endomorphism is the Reidemeister trace of $f$; see~\cite{kate:higher}.
\end{eg}

\begin{eg}\label{eg:mat}
  A simpler bicategorical instance of \autoref{rmk:cartesian} is provided by bicategories of \emph{matrices}.
  Let $\calMat(Ab)$ denote the bicategory whose objects are sets, whose 1-cells $R\hto S$ are $(R\times S)$-matrices $(M_{r,s})_{r\in R,s\in S}$ of abelian groups, and whose composition \odot\ is given by ``matrix multiplication:''
  \[(M\odot N)_{r,t} = \bigoplus_{s\in S} \left(M_{r,s} \ten N_{s,t}\right).
  \]
  The shadow of a square matrix $(M_{r_1,r_2})_{r_1,r_2\in R}$ is given by its ``trace:'' $\sh{M} = \bigoplus_{r\in R} M_{r,r}$.
  The unit $U_R$ is given by
  \[(U_R)_{r_1,r_2} =
  \begin{cases}
    \mathbb{Z} &\text{if } r_1=r_2\\
    0 & \text{otherwise}
  \end{cases}
  \]
  and so the shadow of a set $R$ is $\sh{R} = \bigoplus_R \mathbb{Z} = \mathbb{Z}[R]$.

  It is easy to construct a functor of bicategories $\mathbb{Z}[-]\colon\calSpan(Set)\to\calMat(Ab)$ which is the identity on 0-cells, and which turns a span $R\leftarrow M \to S$ into a matrix whose $(r,s)$-entry is the free abelian group on the fiber of $M$ over $(r,s)$.
  Moreover, if $S$ is finite, and so is each such fiber, then $\bbZ[M]\maps R\hto S$ is dualizable in \calMat(Ab).
  (The case $S=\star$ is an analogue of fiberwise duality, while the case $R=\star$ is an analogue of Costenoble-Waner duality; in~\cite{PS3} we will unify these examples in a general context.)
  Finally, if $f\maps M\to M$ is an endo-2-cell in \calSpan(Set), then
  \[\tr(\bbZ[f])\maps \bbZ[R]\to\bbZ[S]\]
  is the function which maps each generator $r\in R$ to the sum $\sum_{s\in S} \mathrm{ind}(f_{r,s})\cdot s$, where $\mathrm{ind}(f_{r,s})$ is the number of fixed points of $f_{r,s}\maps M_{r,s}\to M_{r,s}$.
\end{eg}

There are a number of other examples, some of which can be found in~\cite{kate:traces, kate:rel,equiv}, but most can be considered extensions of those we have mentioned above.
One other large class of examples consists of ``monoids and bimodules'' in some other bicategory; see for instance~\cite[9.4]{kate:traces} and~\cite[\S11]{shulman:frbi}.
And as mentioned above, in~\cite{PS3} we will study a general class of examples including both \calEx\ and $\calMat(Ab)$.

\section{Properties of trace in bicategories with shadows}
\label{sec:prop-bicat}

We collect here some basic properties of the bicategorical trace, most of them analogous to the well-known properties of the symmetric monoidal trace.
Like the latter, they are easiest to prove using string diagram calculus.
In this section, we assume that \sB\ is a bicategory equipped with a shadow functor.

\begin{prop}[``Tightening'']\label{thm:tr-vfunc-glob}
  Let $M$ be right dualizable, let $f\maps Q\ten M \to M\ten P$ be a
  2-cell, and let $g\maps Q'\to Q$ and $h\maps P\to P'$ be 2-cells.
  Then
  \[\sh{h}\circ \tr(f)\circ \sh{g} = \tr\big((\id_M\ten h)\circ f \circ (g\ten \id_M)\big).\]
\end{prop}

The above equality is shown graphically in \autoref{fig:tightening}.  It
should be easy to visualize a deformation relating the two pictures:
we simply ``pull on the string'' through $g$, $f$, and $h$ (hence the name, which we have taken from~\cite{jsv:traced-moncat}).

\begin{figure}[tbp]
  \centering
  \def\tightening#1#2#3#4{
  \begin{tikzpicture}
    \bgcylinder{0,-2.5}{7}{1.5}{0.5}{blue}{blue}
    \begin{scope}[every node/.style={fill=white}]
      \path (1,1.5) node[vert] (f) {$f$}
      +(-0.5,#1) node[vert2] (g) {$g$}
      +(0.5,-#2) node[vert2] (h) {$h$};
    \end{scope}
    \draw (g |- top) -- node [ed,near end] {$Q'$} (g)
    -- node [ed,swap] {$Q$} (f) -- node [ed,pos=0.35] {$P$} (h)
    -- node [ed] {$P'$} (h |- bot);
    \drawtheta{f}{-#4}{theta}
    \begin{pgfonlayer}{background}
      \draw[greenfill,looseness=0.7]
      (thetaL) to[rrel,out=90,in=-120] (0.5,0.5)
      -- node [ed,near start] {$M$} (f.center)
      -- node [ed,swap,near end] {$M$} ++(.5,#3)
      arc (180:0:0.3)
      -- node [ed] {$\rdual{M}$} ($(f)+(1.1,-1.5)+(0,-#4)$)
      arc (-180:0:0.2)
      to [out=90,in=-90] (thetaR)
      -- ++(0.1,0)
      -- (dr') |- (bot') -| (dl') -- ($(thetaL)+(-0.1,0)$) -- (thetaL);
    \end{pgfonlayer}
  \end{tikzpicture}
  }
  \tightening{1.3}{1.3}{2}{2.1}
  \qquad \raisebox{4cm}{=} \qquad
  \tightening{1.75}{3.3}{1}{1.1}
  \caption{String diagram picture of \autoref{thm:tr-vfunc-glob} (Tightening)}\label{fig:tightening}
\end{figure}
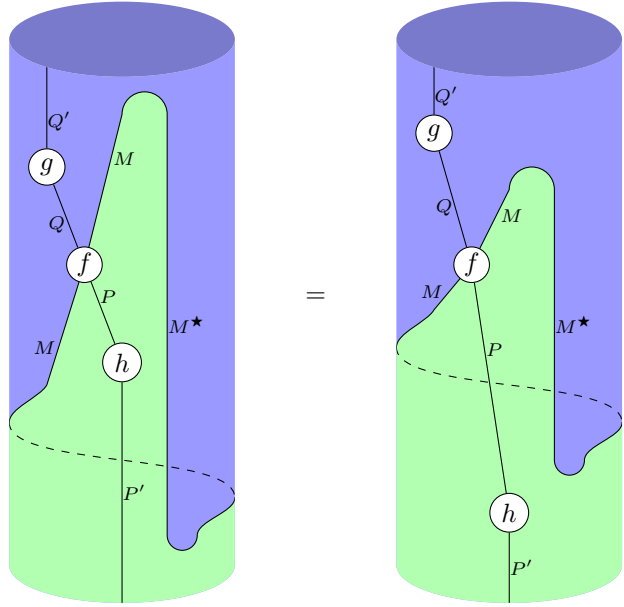

\begin{prop}[``Sliding'' or ``Cyclicity'']\label{thm:tr-cyclnat}
  Let $M$ and $N$ be right dualizable 1-cells in \sB\ and $g\maps
  K\odot N \rightarrow M\odot L$ and $f\maps Q\odot M\rightarrow
  N\odot P$ be 2-cells.  Then the following square commutes:
  \[\xymatrix@C=1.5in{\sh{Q\odot K}\ar[d]_\iso
    \ar[r]^{\tr((f\odot \id_L)(\id_Q\odot g))} &
    \sh{P\odot L}\ar[d]^\iso\\
    \sh{K\odot Q}\ar[r]_{\tr((g\odot \id_P)(\id_K\odot f))}& \sh{L\odot P}}\]
\end{prop}

This equality is shown graphically in \autoref{fig:sliding}.  Here
the deformation requires a little more imagination: the idea is to
slide $f$ to the right and around the back of the cylinder, keeping
the strings labeled $Q$, $K$, $L$, and $P$ fixed where they hit the
top and bottom boundaries of the cylinder.

\begin{figure}[tbp]
  \centering
  \def\slidingcenter#1#2#3#4#5#6#7#8{
    \bgcylinder{0,#8}{7.5}{1.7}{0.5}{#5}{#6}
    \begin{scope}[every node/.style={fill=white}]
      \path (1.6,1.5) node[vert2] (f) {$#1$}
      +(-0.7,-1.2) node[vert2] (g) {$#2$};
    \end{scope}
    \drawtheta{g}{-0.8}{theta}
    \begin{pgfonlayer}{background}
      \draw[#7fill]
      (thetaL) to [out=90,in=-120] node [ed,pos=0.6] {$#3$} (g.center)
      -- (f.center)
      to[rrel,out=75,in=-90,looseness=0.5] node [ed,swap,near end] {$#3$} (.5,1)
      arc (180:0:0.2)
      -- node [ed] {$\rdual{#3}$} +(0,-4.5)
      arc (-180:0:0.2)
      to[out=90,in=-90] (thetaR)
      -- ++(0.1,0)
      -- (dr') |- (bot') -| (dl') -- ($(thetaL)+(-0.1,0)$) -- (thetaL);
    \end{pgfonlayer}
    \begin{pgfonlayer}{foreground}
      \draw (g) -- node [ed] {$#4$} (f);
    \end{pgfonlayer}
  }
  \begin{tikzpicture}
    \slidingcenter{f}{g}{M}{N}{yellow}{blue}{green}{-2}
    \drawtheta{f}{3}{thetatop}
    \begin{pgfonlayer}{background}
      \draw[bluefill]
      (g.center) -- (f.center) -- ++(0,1.2)
      to[out=90,in=-90,looseness=0.5] node [ed] {$Q$} (thetatopR)
      -- ++(0.1,0)
      -- (ur') -- (f |- top)
      to[out=-90,in=90] node [ed,near start] {$K$} ($(g) + (0,2.7)$)
      -- node [ed,swap] {$K$} (g.center);
      \draw[bluefill]
      (thetatopL) to[out=90,in=-90] node [ed,swap] {$Q$} (g |- top) -- (ul')
      -- ($(thetatopL)+(-0.1,0)$) -- (thetatopL);
      \draw[redfill]
      (g.center) -- (f.center)
      -- node [ed,near end] {$P$} (f |- bot) -- (g |- bot)
      -- node [ed,near start] {$L$} (g.center);
    \end{pgfonlayer}
  \end{tikzpicture}
  \qquad \raisebox{4cm}{=} \qquad
  \begin{tikzpicture}
    \slidingcenter{g}{f}{N}{M}{blue}{blue}{red}{-4}
    \begin{pgfonlayer}{background}
      \draw[yellowfill]
      (g.center) -- (f.center)
      -- node [ed] {$K$} (f |- top) -- (g |- top)
      -- node [ed,near start,swap] {$Q$} (g.center);
    \end{pgfonlayer}
    \drawtheta{g}{-2.7}{thetabot}
    \begin{pgfonlayer}{background}
      \draw[greenfill]
      (g.center) -- (f.center) -- node [ed] {$L$} ++(0,-3.5)
      to[out=-90,in=90] node [ed,near end,swap] {$L$} (g |- bot)
      -| (dl') -- ($(thetabotL)+(-0.1,0)$) -- (thetabotL)
      to[out=90,in=-90] node [ed,swap] {$P$} ($(g) + (0,-1.7)$) -- (g.center);
      \draw[greenfill]
      (thetabotR) to [out=-90,in=90] node [ed,swap] {$P$} (f |- bot')
      -| (dr') -- ($(thetabotR)+(0.1,0)$) -- (thetabotR);
    \end{pgfonlayer}
  \end{tikzpicture}
  \caption{String diagram picture of \autoref{thm:tr-cyclnat} (Sliding)}\label{fig:sliding}
\end{figure}
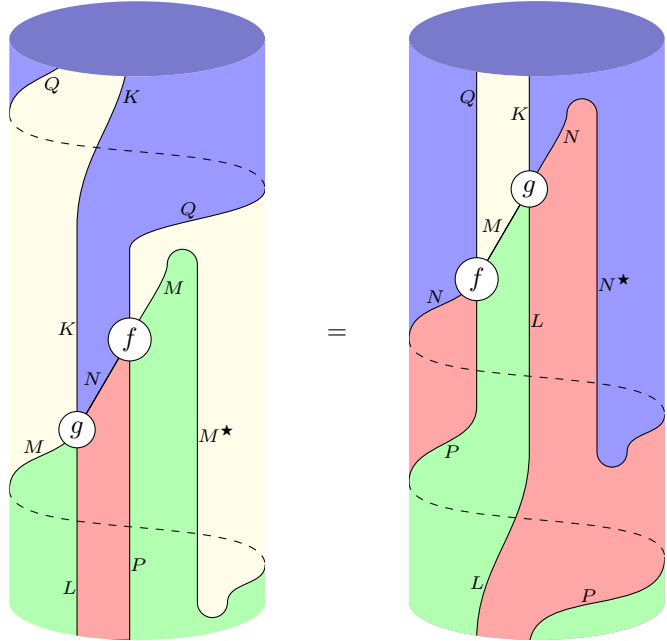


\begin{cor}
  If $M$ and $N$ are right dualizable and $f\colon M\to N$ and
  $g\colon N\to M$ are 2-cells, then
  $\tr(f g) = \tr(g f)$.
\end{cor}

Of course, the unit $U_R$ is always its own dual.

\begin{prop}\label{thm:tr-unit}
  If $f\maps Q\odot U_R \to U_R\odot P$ is any 2-cell, we have
  $\tr(f) = \sh{f}$.
\end{prop}

The string diagram picture of \autoref{thm:tr-unit} is so tautologous
as to not be worth drawing, since unit 1-cells are represented by
empty space.

If $M$ and $N$ are right dualizable with right duals $\rdual{M}$ and
$\rdual{N}$, then $M\odot N$ is right dualizable with right dual
$\rdual{N}\odot\rdual{M}$.  (This can be a source of many dual pairs
that would otherwise be nontrivial to construct, as observed
in~\cite{maysig:pht}.)  In this case, if $f\maps Q\odot M\to M\odot P$
and $g\maps P\odot N \to N\odot L$ are two 2-cells, we have the
composite
\[(\id_M\odot g)(f\odot \id_N)\maps Q\odot M\odot N \too M\odot N\odot L.\]

\begin{prop}\label{thm:tr-vfunc}
  In the above situation, we have
  \[\tr\big((\id_M\odot g)(f\odot \id_N)\big) = \tr(g)\circ \tr(f).\]
\end{prop}

This equality is shown graphically in \autoref{fig:vanishing};
again the deformation should be fairly easy to visualize.

In \cite{PS:mult}, this proposition was used to extend classical multiplicativity 
results for the Lefschetz number and Nielsen number to the Reidemeister trace. 

\begin{figure}[tbp]
  \centering
  \begin{tikzpicture}
    \bgcylinder{0,0}{7}{2}{0.5}{blue}{blue}
    \begin{scope}[every node/.style={fill=white}]
      \path (1,4.5) node[vert] (f) {$f$}
      +(0.5,-1) node[vert2] (g) {$g$};
    \end{scope}
    \drawtheta{f}{-1.7}{thetaM}
    \drawtheta{g}{-1.2}{thetaN}
    \draw (f |- top) -- node [ed] {$Q$} (f)
    -- node [ed] {$P$} (g)
    -- node [ed,near end] {$L$} (g |- bot);
    \begin{pgfonlayer}{background}
      \draw[redfill]
      (f) to[rrel,out=45,in=-90,looseness=0.5] node [ed] {$M$} (0.7,1)
      arc (180:0:0.5)
      -- node [ed] {$\rdual{M}$} +(0,-5)
      arc (-180:0:0.2) to [out=90,in=-90] (thetaMR)
      -- (dr') |- (bot') -| (dl') -- (thetaML)
      to [out=90,in=-110] node [ed] {$M$} (f);
      \draw[greenfill]
      (g) to[rrel,out=75,in=-90,looseness=0.5] node [ed,swap] {$N$} (0.5,2)
      arc (180:0:0.2)
      -- node [ed,swap] {$\rdual{N}$} +(0,-5)
      arc (-180:0:0.5) to [out=90,in=-90] (thetaNR)
      -- (dr') |- (bot') -| (dl') -- (thetaNL)
      to [out=90,in=-110] node [ed] {$N$} (g);
    \end{pgfonlayer}
  \end{tikzpicture}
  \qquad \raisebox{4cm}{=} \qquad
  \begin{tikzpicture}[yscale=.8]
    \bgcylinder{0,-2}{9}{1.5}{0.5}{blue}{blue}
    \begin{scope}[every node/.style={fill=white}]
      \path (1,5.3) node[vert] (f) {$f$}
      +(0,-4.2) node[vert2] (g) {$g$};
    \end{scope}
    \drawtheta{f}{-1}{thetaM}
    \drawtheta{g}{-1}{thetaN}
    \draw (f |- top) -- node [ed,swap] {$Q$} (f)
    -- node [ed,near start] {$P$} (g)
    -- node [ed,near end] {$L$} (g |- bot);
    \begin{pgfonlayer}{background}
      \draw[redfill]
      (f) to[rrel,out=45,in=-90,looseness=0.5] node [ed,swap] {$M$} (0.5,0.7)
      arc (180:0:0.2)
      -- node [ed] {$\rdual{M}$} +(0,-3.5)
      arc (-180:0:0.2) to [out=90,in=-90] (thetaMR)
      -- (dr') |- (bot') -| (dl') -- (thetaML)
      to [out=90,in=-110] node [ed,near end] {$M$} (f);
      \draw[greenfill]
      (g) to[rrel,out=45,in=-90,looseness=0.5] node [ed,swap] {$N$} (0.5,0.7)
      arc (180:0:0.2)
      -- node [ed] {$\rdual{N}$} +(0,-3.5)
      arc (-180:0:0.2) to [out=90,in=-90] (thetaNR)
      -- (dr') |- (bot') -| (dl') -- (thetaNL)
      to [out=90,in=-110] node [ed,near end] {$N$} (g);
    \end{pgfonlayer}
  \end{tikzpicture}
  \caption{String diagram picture of \autoref{thm:tr-vfunc}}
  \label{fig:vanishing}
\end{figure}
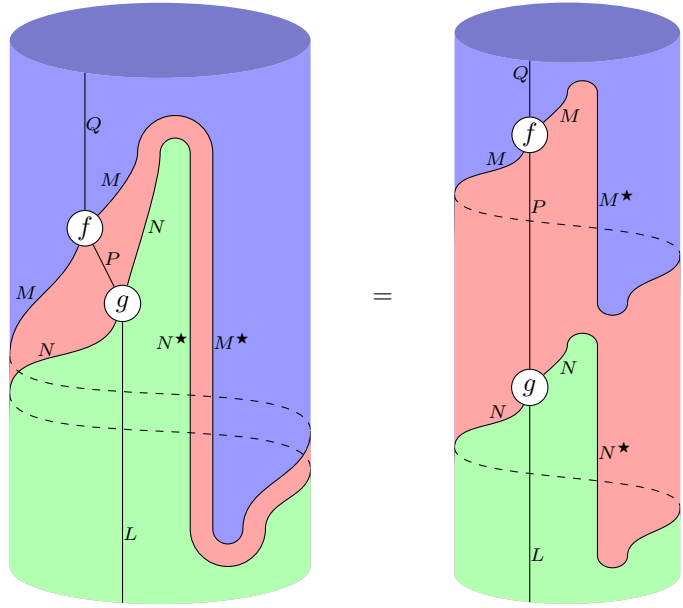

Finally, recall that any 2-cell $f\colon Q\otimes M \to M\otimes P$ has a mate $\rdual{f}\colon \rdual{M}\otimes Q \to P\otimes \rdual{M}$.
Thus, in addition to calculating the trace of $f$,  we can use the analogous notion of trace for the left dualizable object $\rdual{M}$ to calculate the trace of $\rdual{f}$.

\begin{prop}\label{thm:tr-dual}
  If $M$ is right dualizable and $f\maps Q\odot M\rightarrow M\odot P$ is any 2-cell, then
  $\tr(f)=\tr(\rdual{f})$.
\end{prop}

This equality is pictured graphically in \autoref{fig:trmate}.

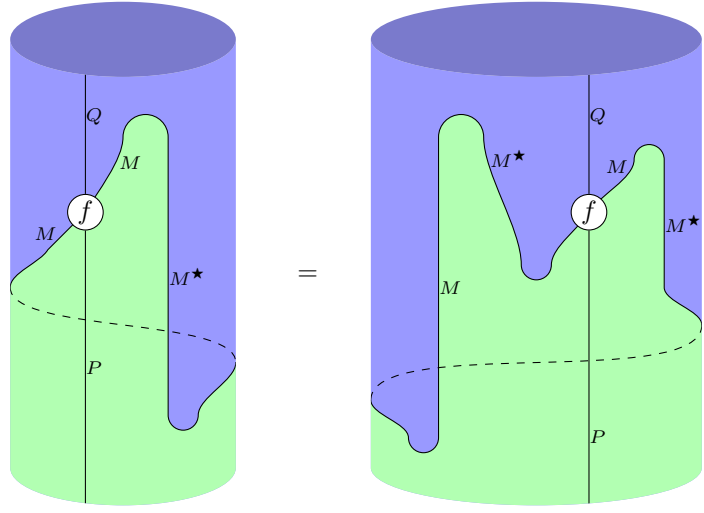
\begin{figure}[tbp]
  \centering
  \begin{tikzpicture}
    \bgcylinder{0,-1.9}{5.7}{1.5}{0.5}{blue}{blue}
    \node[vert] (f) at (1,1.5) {$f$};
    \drawtheta{f}{-1}{theta}
    \begin{pgfonlayer}{background}
      \draw[greenfill,looseness=0.7]
      (thetaL) to[rrel,out=90,in=-120] (0.5,0.5)
      -- node [ed,near start] {$M$} (f.center)
      to[rrel,out=60,in=-90] node [ed,swap,near end] {$M$} (.5,1)
      arc (180:0:0.3)
      -- node [ed] {$\rdual{M}$} ($(f)+(1.1,-2.7)$)
      arc (-180:0:0.2)
      to [out=90,in=-90] (thetaR)
      -- ++(0.1,0)
      -- (dr' |- bot') -- (dl' |- bot') -- ($(thetaL)+(-0.1,0)$) -- (thetaL);
    \end{pgfonlayer}
    \draw (f |- top) -- node[ed] {$Q$} (f);
    \draw (f) -- node[ed] {$P$} (f |- bot);
  \end{tikzpicture}
  \qquad\raisebox{3cm}{$=$}\qquad
  \begin{tikzpicture}[x={(-1,0)}]
    \bgcylinder{0,-1.9}{5.7}{2.2}{0.5}{blue}{blue}
    \node[vert] (f) at (1.5,1.5) {$f$};
    \drawtheta{f}{-1.5}{theta}
    \begin{pgfonlayer}{background}
      \draw[greenfill,looseness=0.7]
      (thetaL) to[rrel,out=90,in=-90] (0.5,0.5)
      -- node[ed,swap] {$\rdual{M}$} ++(0,1.7) arc (180:0:0.2)
      to [out=-90,in=60] node [ed,swap,near start] {$M$} (f.center)
      to[rrel,out=-120,in=90] (.5,-0.7)
      arc (-180:0:0.2)
      to[rrel,out=90,in=-90] node [ed,swap,near end] {$\rdual{M}$} (.5,1.7)
      arc (180:0:0.3)
      -- node [ed] {$M$} ++(0,-4)
      arc (-180:0:0.2)
      to [out=90,in=-90] (thetaR)
      -- ++(0.1,0)
      -- (dr' |- bot') -- (dl' |- bot') -- ($(thetaL)+(-0.1,0)$) -- (thetaL);
    \end{pgfonlayer}
    \draw (f |- top) -- node[ed] {$Q$} (f);
    \draw (f) -- node[ed,near end] {$P$} (f |- bot);
  \end{tikzpicture}
  \caption{The trace of a mate}
  \label{fig:trmate}
\end{figure}

In~\cite{jsv:traced-moncat}, a list of similar properties of the canonical symmetric monoidal trace was used to define the notion of \emph{traced symmetric monoidal category}.
We could use the above properties to define an abstract notion of ``traced bicategory'', but we have no use for such a definition at present.

\section{Functoriality of trace in bicategories with shadows}
\label{sec:funct-bicat}

We now move on to the crucial property of bicategorical trace for
topological applications: its functoriality.  As mentioned in the
introduction, this is what implies refinements of the Lefschetz fixed
point theorem; see \autoref{thm:reidemeister-2}.

First of all, we recall that
a {\bf lax functor} $F\maps \sB\rightarrow \sC$ between bicategories
consists of a function $F_0$ from the objects of \sB\ to the objects
of \sC, functors \[F_{R,S}\maps \sB(R,S)\rightarrow
\sC(F_0(R),F_0(S))\] and natural transformations
\begin{align*}
  \fc\maps F_{R,S}(M)\odot F_{S,T}(N)&\too F_{R,T}(M\odot N)\\
  \fii\maps U_{F(R)}&\too F(U_R)
\end{align*}
satisfying appropriate coherence axioms.  If all maps \fii\ are
isomorphisms, we call $F$ \textbf{normal}, and if all maps \fc\ and
\fii\ are isomorphisms, we call it a \textbf{strong functor}.
We often abuse notation by writing simply $F$ instead of $F_0$ and $F_{R,S}$.

We draw string diagram pictures for functors between bicategories similar to the way we draw them for monoidal functors, by superimposing a pattern on the strings and vertices of the diagram to which the functor is applied.
However, we do \emph{not} superimpose this pattern on a colored 0-cell region unless we wish to indicate application of the functor to a 1- or 2-cell bounded by the 0-cell in question.
That is, if the 0-cell $R$ in \sB\ is denoted by the color blue, then in the context of a functor $F\colon \sB\to\sC$ where we are drawing string diagrams in \sC, a blue region
\begin{tikzpicture}[scale=0.5]
  \fill[bluefill] (0,0) rectangle (1,1);
\end{tikzpicture}
will denote the 0-cell $F(R)$---there being nothing else it could mean, since $R$ itself is not a 0-cell in \sC.
This potentially confusing convention is absolutely necessary, because otherwise we could not distinguish between the 1-cells $F(M\odot N)$ and $F M \odot F N$.
With our convention, however, we can draw them as
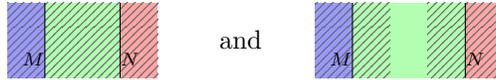
\begin{center}
\begin{tikzpicture}[scale=.5]
  \fill[bluefill] (0,-1) rectangle ++(1,2);
  \fill[greenfill] (1,-1) rectangle ++(2,2);
  \fill[redfill] (3,-1) rectangle ++(1,2);
  \path[dotsF] (0,-1) rectangle ++(4,2);
  \draw (1,1) -- node[ed,near end,swap] {$M$} ++(0,-2);
  \draw (3,1) -- node[ed,near end] {$N$} ++(0,-2);
\end{tikzpicture}
\qquad \raisebox{4mm}{and}\qquad
\begin{tikzpicture}[scale=.5]
  \fill[bluefill] (0,-1) rectangle ++(1,2);
  \fill[greenfill] (1,-1) rectangle ++(3,2);
  \fill[redfill] (4,-1) rectangle ++(1,2);
  \path[dotsF] (0,-1) rectangle ++(2,2);
  \path[dotsF] (3,-1) rectangle ++(2,2);
  \draw (1,1) -- node[ed,near end,swap] {$M$} ++(0,-2);
  \draw (4,1) -- node[ed,near end] {$N$} ++(0,-2);
\end{tikzpicture}
\end{center}
respectively.
The data and axioms of a lax functor can then be drawn as in Figures~\ref{fig:laxfrdata} and~\ref{fig:laxfrax}.
Note that as before, the final axiom is just naturality of \fc.

\begin{figure}[tbp]
  \centering
    \begin{tikzpicture}
      \fill[bluefill] (0,0) rectangle ++(1,2);
      \fill[greenfill] (1,0) rectangle ++(2,2);
      \fill[redfill] (3,0) rectangle ++(1,2);
      \draw (1,0) -- node[ed] {$M$} (1,2);
      \draw (3,0) -- node[ed,swap] {$N$} (3,2);
      \path[dotsF] (0,0) -- ++(4,0) -- ++(0,2) -- ++(-1.5,0)
      to[rrel,out=-90,in=-90,looseness=5] (-1,0) -- ++(-1.5,0) -- cycle;
    \end{tikzpicture}
    \hspace{2cm}
    \begin{tikzpicture}
      \fill[greenfill] (-.5,0) rectangle (1.5,2);
      \path[dotsF] (0,0) -- (1,0) -- ++(0,.7) to[rrel,out=90,in=90,looseness=3] (-1,0) -- cycle;
    \end{tikzpicture}
  \caption{The data for a lax functor}
  \label{fig:laxfrdata}
\end{figure}
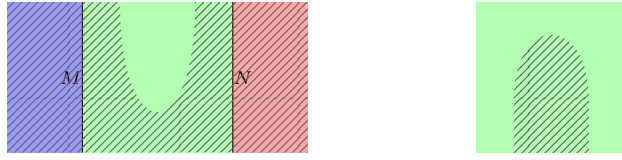

\begin{figure}[tbp]
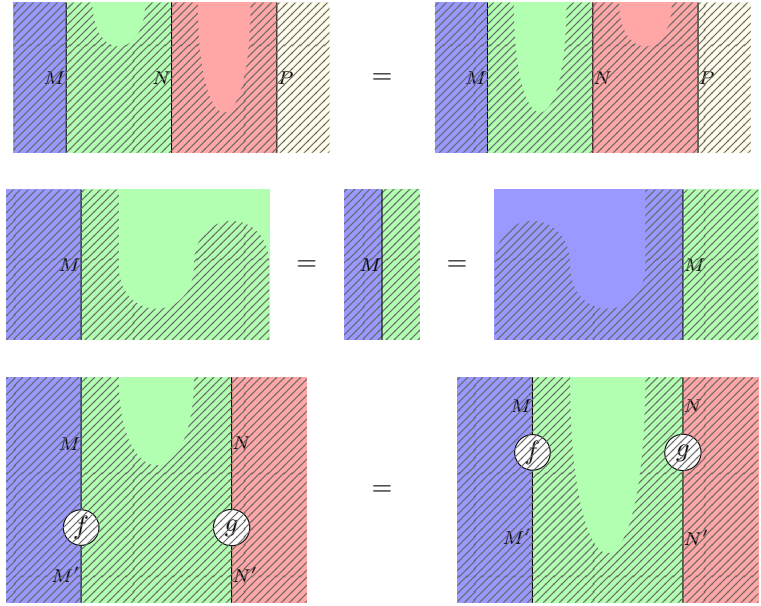

  \centering
    \begin{tikzcenter}[xscale=.7]
      \fill[bluefill] (0,0) rectangle (1,2);
      \fill[greenfill] (1,0) rectangle (3,2);
      \fill[redfill] (3,0) rectangle (5,2);
      \fill[yellowfill] (5,0) rectangle (6,2);
    \draw (1,0) -- node[ed] {$M$} (1,2);
    \draw (3,0) -- node[ed] {$N$} (3,2);
    \draw (5,0) -- node[ed,swap] {$P$} (5,2);
    \path[dotsF] (0,0) -- (6,0) -- (6,2) -- (4.5,2)
    to[out=-90,in=-90,looseness=5] (3.5,2) -- (2.5,2)
    to[out=-90,in=-90,looseness=2] (1.5,2) -- (0,2)
    -- cycle;
    \node at (7,1) {$=$};
    \begin{scope}[xshift=8cm]
      \fill[bluefill] (0,0) rectangle (1,2);
      \fill[greenfill] (1,0) rectangle (3,2);
      \fill[redfill] (3,0) rectangle (5,2);
      \fill[yellowfill] (5,0) rectangle (6,2);
      \draw (1,0) -- node[ed] {$M$} (1,2);
      \draw (3,0) -- node[ed,swap] {$N$} (3,2);
      \draw (5,0) -- node[ed,swap] {$P$} (5,2);
      \path[dotsF] (0,0) -- (6,0) -- (6,2) -- (4.5,2)
      to[out=-90,in=-90,looseness=2] (3.5,2) -- (2.5,2)
      to[out=-90,in=-90,looseness=5] (1.5,2) -- (0,2)
      -- cycle;
    \end{scope}
  \end{tikzcenter}
  \bigskip
  \begin{tikzcenter}
    \fill[bluefill] (0,0) rectangle (1,2);
    \fill[greenfill] (1,0) rectangle (3.5,2);
    \draw (1,0) -- node[ed] {$M$} +(0,2);
    \path[dotsF] (0,0) -- (3.5,0) -- (3.5,1) to[out=90,in=90,looseness=2]
    (2.5,1) to[out=-90,in=-90,looseness=2] (1.5,1) -- (1.5,2) -- (0,2) -- cycle;
    \node at (4,1) {$=$};
    \fill[bluefill] (4.5,0) rectangle +(.5,2);
    \fill[greenfill] (5,0) rectangle +(.5,2);
    \path[dotsF] (4.5,0) rectangle +(1,2);
    \draw (5,0) -- node[ed] {$M$} +(0,2);
    \node at (6,1) {$=$};
    \begin{scope}[xshift=10cm,x={(-1cm,0)}]
      \fill[greenfill] (0,0) rectangle (1,2);
      \fill[bluefill] (1,0) rectangle (3.5,2);
      \draw (1,0) -- node[ed,swap] {$M$} +(0,2);
      \path[dotsF] (0,0) -- (3.5,0) -- (3.5,1) to[out=90,in=90,looseness=2]
      (2.5,1) to[out=-90,in=-90,looseness=2] (1.5,1) -- (1.5,2) -- (0,2) -- cycle;
    \end{scope}
  \end{tikzcenter}
  \bigskip
  \begin{tikzcenter}
    \fill[bluefill] (0,0) rectangle ++(1,3);
    \fill[greenfill] (1,0) rectangle ++(2,3);
    \fill[redfill] (3,0) rectangle ++(1,3);
    \node[vert] (f) at (1,1) {$f$};
    \node[vert2] (g) at (3,1) {$g$};
    \draw (1,0) -- node[ed] {$M'$} (f) -- node[ed] {$M$} (1,3);
    \draw (3,0) -- node[ed,swap] {$N'$} (g) -- node[ed,swap] {$N$} (3,3);
    \path[dotsF] (0,0) -- ++(4,0) -- ++(0,3) -- ++(-1.5,0)
    to[rrel,out=-90,in=-90,looseness=4] (-1,0) -- ++(-1.5,0) -- cycle;
    \node at (5,1.5) {$=$};
    \begin{scope}[xshift=6cm]
    \fill[bluefill] (0,0) rectangle ++(1,3);
    \fill[greenfill] (1,0) rectangle ++(2,3);
    \fill[redfill] (3,0) rectangle ++(1,3);
    \node[vert] (f) at (1,2) {$f$};
    \node[vert2] (g) at (3,2) {$g$};
    \draw (1,0) -- node[ed] {$M'$} (f) -- node[ed] {$M$} (1,3);
    \draw (3,0) -- node[ed,swap] {$N'$} (g) -- node[ed,swap] {$N$} (3,3);
    \path[dotsF] (0,0) -- ++(4,0) -- ++(0,3) -- ++(-1.5,0)
    to[rrel,out=-90,in=-90,looseness=8] (-1,0) -- ++(-1.5,0) -- cycle;
    \end{scope}
  \end{tikzcenter}
  \caption{The axioms for a lax functor}
  \label{fig:laxfrax}
\end{figure}

\begin{defn}[\cite{kate:traces}]\label{def:shadow-functor}
  Let \sB\ and $\sC$ be bicategories with shadow functors, whose
  target categories are \bT\ and $\bZ$, respectively.  A
  \textbf{lax shadow functor} is a lax functor $F\maps \sB\to\sC$
  together with a functor $\Ft\maps \bT\to\bZ$ and a natural
  transformation
  \begin{equation}
    \phi\maps \sh{-}_{\sC} \circ F \too \Ft \circ \sh{-}_\sB
    \label{eq:shadow-functor}
  \end{equation}
  such that the following diagram commutes whenever it makes sense.
  \[\xymatrix{\sh{F(M)\odot F(N)} \ar[r]^\theta\ar[d]_\fc & \sh{F(N)\odot F(M)} \ar[d]^\fc\\
    \sh{F(M\odot N)}\ar[d]_\phi & \sh{F(N\odot M)} \ar[d]^\phi\\
    \Ft\,\sh{M\odot N} \ar[r]_{\Ft(\theta)} & \Ft\,\sh{N\odot M}.}
  \]
  If $F$ is a strong functor and \phi\ is an isomorphism, we
  call $F$ a \textbf{strong shadow functor}.
\end{defn}

If $F$ is a lax shadow functor, we depict the functor \Ft\ by covering an entire cylinder in the pattern that denotes $F$, and we depict the transformation $\phi$ and its axioms as shown in \autoref{fig:shfr}.

\begin{figure}[tbp]
  \centering
  \subfigure[Data]{\label{fig:shfrdata}
  \begin{tikzpicture}[scale=.8]
    \bgcylinder{0,0}{3}{1.2}{.3}{blue}{blue}
    \draw (top) -- node [ed,near start] {$M$} (bot);
    \path[dotsF] ($(ur)!.5!(top)$) coordinate (start) --
    ($(ul)!.5!(top)$) -- ++(0,-1)
    to[out=-90,in=90] ($(dl)!.2!(ul)$) -- (dl)
    |- (bot) -| (dr) -- ($(dr)!.2!(ur)$)
    to[out=90,in=-90] ($(start)+(0,-1)$)
    -- cycle;
  \end{tikzpicture}}
  \hspace{2cm}
  \subfigure[Naturality]{
  \begin{tikzpicture}[scale=.8]
    \bgcylinder{0,0}{4}{1.2}{.3}{blue}{blue}
    \node[vert] (f) at (1.2,1) {$f$};
    \draw (top) -- node [ed,near start] {$M$} (f) -- (bot);
    \path[dotsF] ($(ur)!.5!(top)$) coordinate (start) --
    ($(ul)!.5!(top)$) -- ++(0,-1)
    to[out=-90,in=90] ($(dl)!.4!(ul)$) -- (dl)
    |- (bot) -| (dr) -- ($(dr)!.4!(ur)$)
    to[out=90,in=-90] ($(start)+(0,-1)$)
    -- cycle;
  \end{tikzpicture}
  \quad\raisebox{2cm}{$=$}\quad
  \begin{tikzpicture}[scale=.8]
    \bgcylinder{0,0}{4}{1.2}{.3}{blue}{blue}
    \node[vert] (f) at (1.2,3) {$f$};
    \draw (top) -- (f) -- node [ed,near end] {$M$} (bot);
    \path[dotsF] ($(ur)!.5!(top)$) coordinate (start) --
    ($(ul)!.5!(top)$) -- ++(0,-2)
    to[out=-90,in=90] ($(dl)!.15!(ul)$) -- (dl)
    |- (bot) -| (dr) -- ($(dr)!.15!(ur)$)
    to[out=90,in=-90] ($(start)+(0,-2)$)
    -- cycle;
  \end{tikzpicture}}
  \\
  \subfigure[Coherence axiom]{\label{fig:shfrax}
  \begin{tikzpicture}[scale=.8]
    \bgcylinder{0,-.5}{5}{1.6}{.3}{blue}{blue}
    \begin{pgfonlayer}{foreground}
      \drawtheta{0,0}{1.5}{theta}
    \end{pgfonlayer}
    \draw[greenfill] ($(ur)!.6!(top)$) coordinate (a)
    -- ($(ul)!.6!(top)$) -- node [ed,swap] {$M$} ++(0,-2)
    to[out=-90,in=90] (thetaL) -- ++(-.1,0) -- (dl')
    |- ($(dl')!.7!(bot')$) to[out=90,in=-90] ($(a)+(0,-2)$)
    --  node [ed,swap] {$N$} (a);
    \draw[greenfill] (thetaR) .. controls +(0,-.5) and +(0,1) ..
    ($(dr')!.7!(bot')$) -| (dr') -- cycle;
    \path[dotsF] ($(ur)!.3!(top)$) coordinate (start) --
    ($(ur)!.8!(top)$) to[out=-90,in=-90,looseness=5]
    ($(ul)!.8!(top)$) -- ($(ul)!.3!(top)$) -- ++(0,-1.2)
    to[out=-90,in=90,looseness=1] ($(dl)+(0,2.7)$) -- (dl)
    |- (bot) -| (dr) -- ($(dr)+(0,2.7)$)
    to[out=90,in=-90] ($(start)+(0,-1.2)$)
    -- cycle;
  \end{tikzpicture}
  \quad\raisebox{2.3cm}{$=$}\quad
  \begin{tikzpicture}[scale=.8]
    \bgcylinder{0,-.5}{5}{1.6}{.3}{green}{blue}
    \begin{pgfonlayer}{foreground}
      \drawtheta{0,0}{3.2}{theta}
    \end{pgfonlayer}
    \draw[bluefill] ($(ul)!.6!(top)$) to[out=-90,in=90] node[ed] {$M$} (thetaL)
    -- ++(-.1,0) -- (ul') -- cycle;
    \draw[bluefill] ($(ur)!.6!(top)$) to[out=-90,in=90] node[ed,near start] {$N$}
    ($(dl)!.6!(bot)+(0,2)$) coordinate (Nbend)
    -- ++(0,-5) -| ($(dr')!.6!(bot')$) -- ++(0,2.2) coordinate (Mbend)
    to[out=90,in=-90] (thetaR) -| (ur') -- cycle;
    \path[dotsF] ($(ul)!.3!(top)$) to[out=-90,in=90] ($(thetaL)+(-.7,0)$) -- ++(.7,-.6)
    to[out=90,in=-90] ($(ul)!.9!(top)$) -- cycle;
    \path[dotsF] ($(ur)!.8!(top)$) to[out=-90,in=90] ($(dl)!.6!(bot)+(-.4,2)$)
    -- ++(0,-.7) to[out=-90,in=90] ($(dl)+(0,.3)$) -- (dl)
    |- (bot) -| (dr) -- ++(0,.3) to[out=90,in=-90] ($(dr)!.6!(bot)+(.4,1.3)$)
    -- ++(0,.5) to[out=90,in=-90] ($(thetaR)+(.7,0)$) -- ++(-.7,.5)
    to[out=-90,in=90] ($(Mbend)+(-.3,0)$)
    to[rrel,out=-90,in=-90,looseness=3] (-.7,0)
    to[out=90,in=-90,looseness=1] ($(ur)!.6!(top)+(.3,0)$) -- cycle;
  \end{tikzpicture}}
  \caption{A lax shadow functor}
  \label{fig:shfr}
\end{figure}
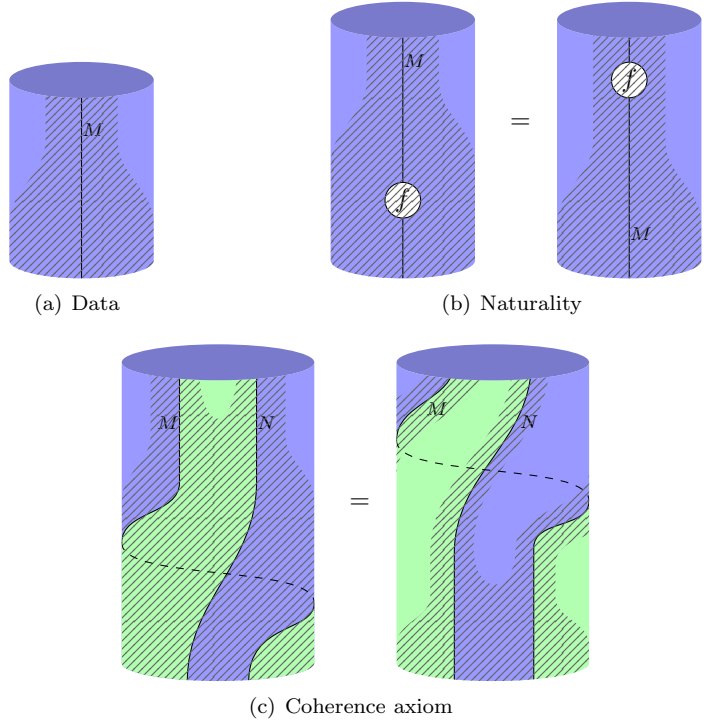

\begin{prop}\label{thm:bicat-funct-pres-tr}
  Let $F\maps \sB\rightarrow \sC$ be a lax shadow functor and $M\maps
  R\hto S$ a right dualizable 1-cell in $\sB$.
  \begin{enumerate}
  \item If $\fc\maps F(M)\odot F(\rdual{M})\rightarrow F(M\odot \rdual{M})$ and $\fii\maps U_F(R)\rightarrow F(U_R)$ are isomorphisms, then $F(M)$ is dualizable with dual $F(\rdual{M})$.\label{item:bfpt1}
  \item If, furthermore, $\fc_{M,P}\maps F(M)\odot F(P)\rightarrow F(M\odot P)$ is an isomorphism, then for any $f\maps Q\odot M\rightarrow M\odot P$, the following square commutes:\label{item:bfpt2}
    \[\xymatrix@C=1.5in{\sh{F(Q)}\ar[r]^{\tr(\fc_{M,P}^{-1}\circ F(f)\circ \fc_{Q,M})}
      \ar[d]_\phi &\sh{F(P)}\ar[d]^\phi\\
      \Ft\,\sh{Q}\ar[r]^{\Ft(\tr(f))}&\Ft\,\sh{P}.}\]
  \end{enumerate}
\end{prop}
\begin{proof}
  Statement~\ref{item:bfpt1} is proven exactly as in the symmetric monoidal case, while a graphical proof of statement~\ref{item:bfpt2} is shown in \autoref{fig:bicatfunc}.
  Since this is one of the centrally important facts about bicategorical traces, and since string diagrams for functors have not yet been formalized, we have chosen to show two intermediate steps of this proof.
  To get from the first diagram to the second, we apply naturality to slide an instance of \fc\ and of its inverse up to the top and bottom, respectively, and then apply the axioms from \autoref{fig:laxfrax} to cancel them with an instance of \fii\ and its inverse, respectively.
  The step from the second to the third diagram is precisely the axiom in Figure~\ref{fig:shfrax}, applied to the 1-cells $P\odot \rdual{M}$ and $M$.
  Finally, to get to the final diagram we cancel an instance of \fc\ with its inverse (removing the ``hole'' in the middle), and apply naturality to slide \phi\ all the way up to the top.
\end{proof}

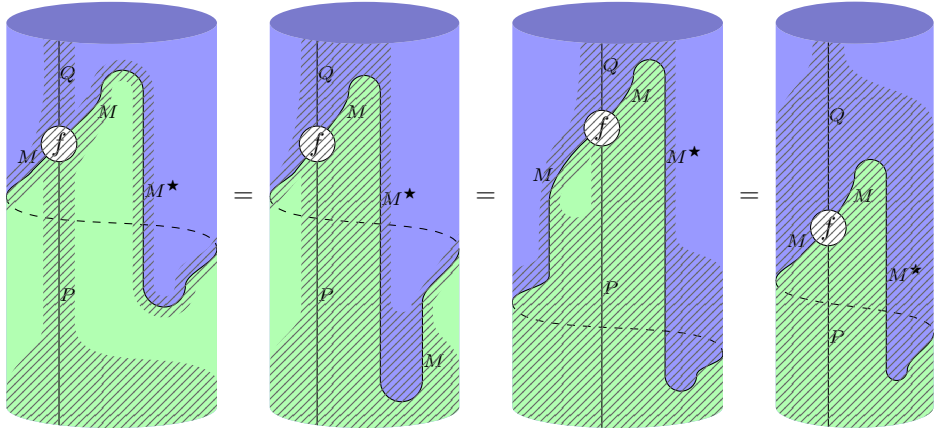
\begin{figure}[tbp]
  \centering
  \begin{tikzpicture}[scale=.7]
    \bgcylinder{0,-3.5}{7.3}{2}{0.4}{blue}{blue}
    \node[vert] (f) at (1,1.5) {$f$};
    \drawtheta{f}{-1}{theta}
    \begin{pgfonlayer}{background}
      \draw[greenfill,looseness=0.7]
      (thetaL) to[rrel,out=90,in=-120] (0.5,0.5)
      -- node [ed,near start] {$M$} (f.center)
      to[rrel,out=60,in=-90] node [ed,swap,near end] {$M$} (.8,1)
      arc (180:0:0.4)
      -- node [ed] {$\rdual{M}$} ++(0,-3.7)
      arc (-180:0:0.4) coordinate (c)
      to [out=90,in=-90] (thetaR)
      -- ++(0.1,0)
      -- (dr' |- bot') -- (dl' |- bot') -- ($(thetaL)+(-0.1,0)$) -- (thetaL);
    \end{pgfonlayer}
    \draw (f |- top) -- node[ed] {$Q$} (f);
    \draw (f) -- node[ed] {$P$} (f |- bot);
    \path[dotsF] (f |- top) -- ++(-.3,0)
    to[out=-90,in=70] ($(f)+(-.5,0)$)
    to[out=-110,in=90] ($(thetaL)+(0,.3)$)
    -- ++(0,-.6) -- ($(f)+(-.3,-.6)$)
    -- ($(f |- bot)+(-.3,2)$)
    to[out=-90,in=90] ($(dl)+(0,.3)$)
    -- (dl') |- (bot') -| (dr') -- ++(0,.3)
    to[out=90,in=-90] ($(f |- bot)+(.3,2)$)
    -- ($(f)+(.3,-.6)$) to[out=70,in=-90] ($(f)+(1,1)$)
    arc (180:0:.2) -- ++(0,-3.7)
    arc (-180:0:.6) -- ($(thetaR)+(0,-.3)$) -- ++(0,.6)
    to[out=-110,in=90,looseness=1.3] ($(c)+(-.2,0)$) arc (0:-180:.2) -- ++(0,3.7) arc (0:180:.6)
    -- ($(f)+(.3,.6)$) -- ($(f |- top)+(.3,0)$) -- cycle;
  \end{tikzpicture}
  \;\raisebox{3cm}{$=$}\;
  \begin{tikzpicture}[scale=.7]
    \bgcylinder{0.4,-3.5}{7.3}{1.8}{0.4}{blue}{blue}
    \node[vert] (f) at (1.3,1.5) {$f$};
    \drawtheta{f}{-1}{theta}
    \begin{pgfonlayer}{background}
      \draw[greenfill,looseness=0.7]
      (thetaL) to[rrel,out=90,in=-120] (0.5,0.5)
      -- (f.center)
      to[rrel,out=60,in=-90] node [ed,swap,near end] {$M$} (.6,1)
      arc (180:0:0.3)
      -- node [ed] {$\rdual{M}$} ++(0,-4) node[coordinate] (b) {}
      -- ++(0,-1.5)
      arc (-180:0:0.4) -- node[ed,swap,near start] {$M$} ++(0,1.5) coordinate (c)
      to [out=90,in=-90] (thetaR)
      -- ++(0.1,0)
      -- (dr' |- bot') -- (dl' |- bot') -- ($(thetaL)+(-0.1,0)$) -- (thetaL);
    \end{pgfonlayer}
    \draw (f |- top) -- node[ed] {$Q$} (f);
    \draw (f) -- node[ed] {$P$} (f |- bot);
    \path[dotsF] (f |- top) -| ($(b)+(.2,0)$)
    arc (-180:0:.2) -- ($(thetaR)+(0,.3)$) -- ++(0,-.6)
    -- ($(c)+(.2,0)$) -- ++(0,-.3) to[out=-90,in=90]
    ($(dr)+(0,.8)$) |- (bot) -| ($(dl)+(0,.8)$)
    to[out=90,in=-90] ($(f)+(-.3,-3.1)$) -- ++(0,2.5) -- ($(thetaL)+(0,-.3)$) -- ++(0,.6)
    to[out=40,in=-90,looseness=.7] ($(f |- top)+(-.3,0)$)
    -- cycle;
  \end{tikzpicture}
  \;\raisebox{3cm}{$=$}\;
  \begin{tikzpicture}[scale=.7]
    \bgcylinder{0,-3.5}{7.3}{2}{0.4}{blue}{blue}
    \node[vert] (f) at (1.7,1.8) {$f$};
    \drawtheta{f}{-3.3}{theta}
    \begin{pgfonlayer}{background}
      \draw[greenfill,looseness=0.7]
      (thetaL) to[rrel,out=90,in=-90] (.7,.5)
      -- ++(0,1.4) to[out=90,in=-150] node [ed,near start] {$M$} (f.center)
      to[rrel,out=60,in=-90] node [ed,swap,near end] {$M$} (.6,1)
      arc (180:0:0.3)
      -- node [ed] {$\rdual{M}$} ++(0,-3) node[coordinate] (b) {}
      -- ++(0,-2.7)
      arc (-180:0:0.3) coordinate (c)
      to [out=90,in=-90] (thetaR)
      -- ++(0.1,0)
      -- (dr' |- bot') -- (dl' |- bot') -- ($(thetaL)+(-0.1,0)$) -- (thetaL);
    \end{pgfonlayer}
    \draw (f |- top) -- node[ed] {$Q$} (f);
    \draw (f) -- node[ed] {$P$} (f |- bot);
    \path[dotsF] (f |- top) -| ($(b)+(.2,0)$)
    to[out=-90,in=90] ($(dr)+(0,2.5)$) |- (bot) -| ($(dl)+(0,2.5)$)
    to[rrel,out=90,in=-90] (.5,.4)
    to[out=90,in=-90,looseness=2] ($(f |- top)+(-.4,0)$) -- cycle;
    \path[greenfill] (f) ++(-.2,-.4) node[coordinate] (fl) {}
    -- ++(0,-1) arc (0:-180:.3)
    to[out=90,in=-90] (fl);
  \end{tikzpicture}
  \;\raisebox{3cm}{$=$}\;
  \begin{tikzpicture}[scale=.7]
    \bgcylinder{0,-1.9}{7.3}{1.5}{0.4}{blue}{blue}
    \node[vert] (f) at (1,1.5) {$f$};
    \drawtheta{f}{-1}{theta}
    \begin{pgfonlayer}{background}
      \draw[greenfill,looseness=0.7]
      (thetaL) to[rrel,out=90,in=-120] (0.5,0.5)
      -- node [ed,near start] {$M$} (f.center)
      to[rrel,out=60,in=-90] node [ed,swap,near end] {$M$} (.5,1)
      arc (180:0:0.3)
      -- node [ed] {$\rdual{M}$} ($(f)+(1.1,-2.7)$)
      arc (-180:0:0.2)
      to [out=90,in=-90] (thetaR)
      -- ++(0.1,0)
      -- (dr' |- bot') -- (dl' |- bot') -- ($(thetaL)+(-0.1,0)$) -- (thetaL);
    \end{pgfonlayer}
    \draw (f |- top) -- node[ed] {$Q$} (f);
    \draw (f) -- node[ed] {$P$} (f |- bot);
    \path[dotsF] (f |- top) -- ++(-.3,0) -- ++(0,-.5)
    to[out=-90,in=90] ($(ul)+(0,-2)$)
    -- (dl) |- (bot) -| (dr) -- ($(ur)+(0,-2)$)
    to[out=90,in=-90] ($(f |- top)+(.3,-.5)$) -- ++(0,.5) -- cycle;
  \end{tikzpicture}
  \caption{Functoriality of the bicategorical trace}
  \label{fig:bicatfunc}
\end{figure}

As in the symmetric monoidal case, in the situation of~\ref{item:bfpt1} above, we say that $F$ \textbf{preserves} the dual $\rdual{M}$ of $M$.

\begin{eg}\label{sym_shadow_functor}
  If \bC\ and \bD\ are symmetric monoidal categories, and we equip
  $\bccat{\bC}$ and $\bccat{\bD}$ with their canonical shadows, then any lax
  monoidal functor $F\maps \bC\to\bD$ gives a lax shadow functor
  $\bccat{\bC}\to\bccat{\bD}$ such that \phi\ is an isomorphism.  In this
  case \autoref{thm:bicat-funct-pres-tr} reduces to
  \autoref{thm:funct-pres-tr}.
\end{eg}

\begin{eg}
  Homology is a lax shadow functor
  \[\calCh \to \calGrMod,\]
  where $\calGrMod$ is like \calMod\ except that its 1-cells are
  \emph{graded} bimodules.
  The K\"unneth theorem holds for rings that are not necessarily
  commutative, so \autoref{thm:bicat-funct-pres-tr} implies that
  $\tr(H_*(f))=H_*(\tr(f))$ for any map $f\colon Q_*\odot C_*\rightarrow 
  C_*\odot P_*$
  of a chain complex of $S$-$R$-modules, as long as $C_p$ and $H_p(C_*)$
  are projective for each $p$ and $C_*$ is finitely generated.

\end{eg}

\begin{eg}\label{thm:reidemeister-2}
  The rational chain complex functor induces a lax shadow functor
  \[C_*\colon \Ho(\bgpsp)\rightarrow \Ho(\calCh).\]
  On 0-cells we have $C_*(G)=\bbQ G$, the group ring, while on 1-cells and 2-cells it is the usual rational chain complex functor, equipped with the induced actions by group rings.
  The shadow action $(C_*)_{\tr} \colon \Ho(\Sp)\to \Ho(\bCh{\mathbb{Z}})$ is simply the ordinary rational chain complex.

  If $f\colon X\rightarrow X$ is an endomorphism of a closed smooth manifold
  or a compact ENR, as in \autoref{eg:reidemeister-1}, then
  \[C_*(\tilde{f})\colon C_*(\tilde{X})\rightarrow C_*(\tilde{X})\]
  is $\psi$-equivariant for the induced map $\psi\colon \bbQ\pi_1(X)\rightarrow \bbQ\pi_1(X)$.
  Therefore, as in \autoref{eg:equivar-trace}, we can regard it as a 2-cell (also denoted $C_*(\tilde{f})$)
  \[ C_*(\tilde{X}) \to C_*(\tilde{X}) \odot \bcd{\psi}{\bbQ\pi_1(X)}. \]
  Its trace in $\Ho(\calCh)$ is then a map $\bbQ\rightarrow \bbQ\sh{\pi_1(X)_\psi}$.
  This is another way to define the Reidemeister trace;
  \autoref{thm:bicat-funct-pres-tr} shows that
  \[\tr(C_*(\tilde{f})) = C_*(\tr(\tilde{f})).\]
  (Recall that $\tr(\tilde{f})$ was what we called the Reidemeister
  trace in \autoref{eg:reidemeister-1}.)  This is a more refined
  version of the Lefschetz fixed point theorem; see \cite{kate:traces}.

  We might like to be able to combine this example with the previous
  one and calculate the Reidemeister trace at the level of homology,
  but unfortunately the resulting modules over the group ring are
  rarely projective, so the K\"unneth theorem generally fails.
\end{eg}

\section{2-functoriality of trace in bicategories with shadows}
\label{sec:transf}

We observed in \S\ref{sec:traces} that in the symmetric monoidal case, traces are respected not only by monoidal functors, but by monoidal transformations (\autoref{thm:trans-pres-tr}).
To conclude the main portion of the paper, we would like to prove a version of this for bicategories, but we have to be careful regarding what sort of transformation to consider.
It turns out that the appropriate type is the following.

\begin{defn}
  Let \sB\ and \sC\ be bicategories and $F,G\colon \sB \to \sC$ be lax functors.  A \textbf{conjunctional transformation} $\al\colon F\to G$ consists of the following.
  \begin{enumerate}
  \item For each 0-cell $R\in \sB$, a right dualizable 1-cell $\al_{R}\colon F R \hto G R$ in \sC, with right dual $\ral R$.
  \item For each 1-cell $M\colon R\hto S$ in \sB, a 2-cell $\al_M\colon F(M) \odot \al_S \to  \al_R \odot G(M)$ in \sC, with consequent mate $\ral{M}\colon \ral{R}\odot F(M) \to G(M)\odot \ral{S}$.
  \item Some coherence axioms are satisfied.  (See below for the axioms.)
  \end{enumerate}
\end{defn}

\begin{rmk}\label{rmk:why-conjunctional}
  To motivate this definition, we consider the main example we want to apply it to.
  Recall from \autoref{thm:reidemeister-2} that we have a lax shadow functor
  \[C_*(-;\bbQ)\colon \Ho(\bgpsp)\rightarrow \Ho(\calCh)\]
  which enables us to compute the Reidemeister trace as a bicategorical trace in $\Ho(\calCh)$.
  In \autoref{eg:intrat-lefschetz} we used \autoref{thm:trans-pres-tr} to conclude that the ordinary Lefschetz number is the same whether computed with integral or rational coefficients, so we would like a similar statement in the bicategorical situation.

  We certainly have another lax shadow functor
  \[C_*(-;\bbZ)\colon \Ho(\bgpsp)\rightarrow \Ho(\calCh)\] 
  defined using integral chain complexes.
  (Since we are working only at the level of chain complexes, as in the remark at the end of \autoref{eg:intrat-lefschetz}, we don't need to quotient by torsion to make this a strong functor.)
  Thus, when looking for a definition of transformation, we should ask what sort of transformation the inclusion $\bbZ\rightarrow \bbQ$ induces from $C_*(-;\bbZ)$ to $C_*(-;\bbQ)$.

  The first obvious thing that is induced is a collection of ring homomorphisms $\alpha_G\colon \bbZ G \to \bbQ G$, for any 0-cell $G$ in \bgpsp.
  A ring homomorphism $\psi\colon R\to S$ is not itself any sort of cell in \calCh, but it does induce a dual pair of bimodules $(\bc\psi S,\bcd\psi S)$.
  (We have already met these bimodules in \autoref{eg:equivar-trace}, and some analogous objects in \autoref{eg:reidemeister-1}, where they supplied the 1-cells $P$ and $Q$ by which bicategorical traces were ``twisted.'')
  This motivates the choice to take dual pairs as the 1-cell components of a conjunctional transformation.

  Next, for any 1-cell $M\colon G\hto H$ in \bgpsp, we have a map
  \[\alpha_M\colon C_*(M;\bbZ)\to C_*(M;\bbQ)\]
  of chain complexes.
  This is not a 2-cell in any hom-category of \calCh, but instead is an ``$\alpha_G$-$\alpha_H$-equivariant map'' (i.e.\ it satisfies $\alpha_M(x\cdot m \cdot y) = \alpha_G(g)\cdot \alpha_M(m) \cdot \alpha_H(y)$).
  Such an equivariant map can be identified with a map
  \[C_*(M;\bbZ) \;\too \;\bc{i_G}{\bbQ G} \;\odot\; C_*(M;\bbQ) \;\odot\; \bcd{i_H}{\bbQ H}.\]
  Finally, the mate of such a map is a morphism
  \[C_*(M;\bbZ) \;\odot\; \bc{i_H}{\bbQ H} \;\too\; \bc{i_G}{\bbQ G} \;\odot\; C_*(M;\bbQ)\]
  of $\bbZ G$-$\bbQ H$-bimodules, and this provides the 2-cell components of a conjunctional transformation.

  As we will see later, this seemingly \emph{ad hoc} definition also provides exactly the right structure necessary to prove 2-functoriality of traces.
  It can also be shown to arise naturally from a natural sort of transformation for a class of ``fibrant'' double categories; see~\cite{shulman:frbi,csmb}.
\end{rmk}

As usual, the coherence axioms of a conjunctional transformation are most naturally visualized in string diagram notation.
Just like in the symmetric monoidal case, we picture such a transformation as a `membrane' dividing the $F$-region from the $G$-region.
However, now the membrane itself is actually a string: a horizontally drawn string representing the dual pair $(\al_{R},\ral{R})$, as described in \S\ref{sec:bicat-traces}.
Similarly, the locations where other strings cross over the membrane represent the 2-cell components of $\al$, although we usually do not draw them as nodes.
Note that according to the convention for drawing functors established in \S\ref{sec:funct-bicat}, a colored region with no pattern can equally denote a 0-cell in the image of $F$ or in the image of $G$, although in practice there should never be any ambiguity about which is intended.

With these conventions in place, \autoref{fig:ctdata} displays the data of a conjunctional transformation, and \autoref{fig:ctax} shows the coherence axioms.

\begin{figure}[hbt]
  \centering
  \begin{tikzpicture}[scale=1.5]
    \coordinate (lctr) at (0,0);
    \clip (.1,-.9) rectangle (1.9,.9);
    \filldraw[bluefill]
    (lctr) -- +(0,1) -- +(2,1) -- +(2,0) -- node [ed,swap,bluefill,circle] {$\al_R$} (lctr);
    \node [ed,anchor=north west,color=blue,bluefill] at (0.1,.9) {$F(R)$};
    \filldraw[bluefill]
    (lctr) -- +(0,-1) -- +(2,-1) -- +(2,0) -- (lctr);
    \node [ed,anchor=south west,color=blue,bluefill] at (0.1,-.9) {$G(R)$};
  \end{tikzpicture}
  \qquad
  \begin{tikzpicture}[x=1.5cm]
    \clip (-1.4,-1.4) rectangle (1.4,1.4);
    \coordinate (ctr) at (0,0);
    \filldraw[greenfill]
    (ctr) -- node [ed,swap,greenfill] {$M$}
    +(0,1.5) -- +(1.5,1.5) -- +(1.5,0) -- node [ed,near start,above=2pt] {$\al_S$} (ctr);
    \node [ed,anchor=north east,color=green!50!black,greenfill] at (1.4,1.4) {$F(S)$};
    \filldraw[greenfill]
    (ctr) -- node [ed,greenfill] {$M$}
    +(0,-1.5) -- +(1.5,-1.5) -- +(1.5,0) -- (ctr);
    \node [ed,anchor=south east,color=green!50!black,greenfill] at (1.4,-1.4) {$G(S)$};
    \filldraw[bluefill]
    (ctr) -- +(0,1.5) -- +(-1.5,1.5) -- +(-1.5,0) -- node [ed,near start,above=2pt] {$\al_R$} (ctr);
    \node [ed,anchor=north west,color=blue,bluefill] at (-1.4,1.4) {$F(R)$};
    \filldraw[bluefill]
    (ctr) -- +(0,-1.5) -- +(-1.5,-1.5) -- +(-1.5,0) -- (ctr);
    \node [ed,anchor=south west,color=blue,bluefill] at (-1.4,-1.4) {$G(R)$};
    \path[dotsF] (ctr) -- +(-.7,0) rectangle +(.7,1.5);
    \path[dotsG] (ctr) -- +(-.7,-1.5) rectangle +(.7,0);
  \end{tikzpicture}
  \caption{The data of a conjunctional transformation}
  \label{fig:ctdata}
\end{figure}
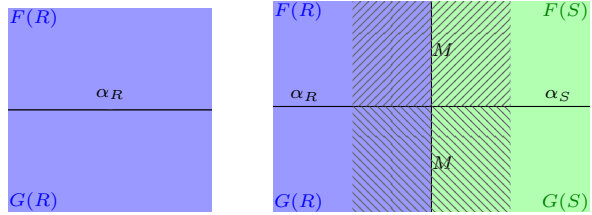

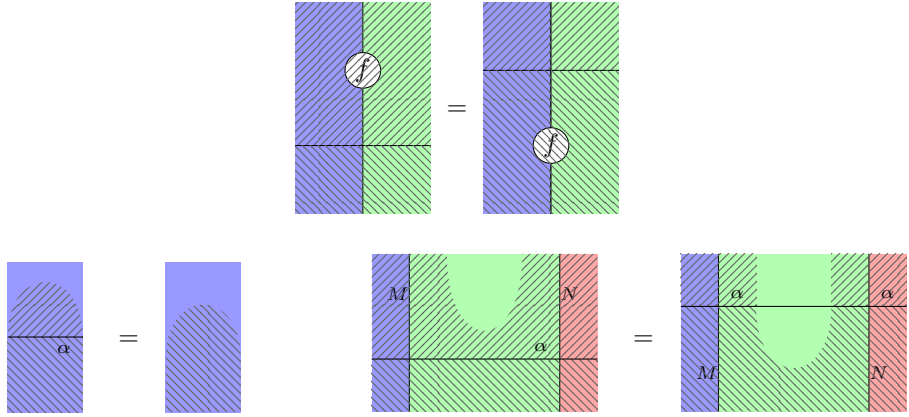
\begin{figure}[hbt]
  \centering
  \begin{tikzpicture}
    \clip (0.1,-.9) rectangle (1.9,1.9);
    \filldraw[bluefill] (0,-1) rectangle (1,2);
    \filldraw[greenfill] (1,-1) rectangle (2,2);
    \node[vert] at (1,1) {$f$};
    \filldraw[dotsF] (0,0) rectangle (2,2);
    \fill[dotsG] (0,0) rectangle (2,-1);
  \end{tikzpicture}
  \;\raisebox{1.3cm}{=}\;
  \begin{tikzpicture}
    \clip (0.1,-.9) rectangle (1.9,1.9);
    \filldraw[bluefill] (0,-1) rectangle (1,2);
    \filldraw[greenfill] (1,-1) rectangle (2,2);
    \node[vert] at (1,0) {$f$};
    \filldraw[dotsF] (0,1) rectangle (2,2);
    \fill[dotsG] (0,1) rectangle (2,-1);
  \end{tikzpicture}
  \\\vspace{.5cm}
  \begin{tikzpicture}
    \fill[bluefill] (0,0) rectangle (1,2);
    \path[dotsG] (0,0) rectangle (1,1);
    \path[dotsF] (0,1) to[out=90,in=90,looseness=2.5] (1,1) -- (0,1);
    \draw (0,1) -- node[ed,near end,below=3pt] {$\alpha$} (1,1);
  \end{tikzpicture}
  \quad\raisebox{.9cm}{$=$}\quad
  \begin{tikzpicture}
    \fill[bluefill] (0,0) rectangle (1,2);
    \path[dotsG] (0,0) -- (0,.7) to[rrel,out=90,in=90,looseness=2.5] (1,0) -- (1,0) -- cycle;
  \end{tikzpicture}
  \hspace{1.5cm}
  \begin{tikzpicture}[yscale=.7]
    \fill[bluefill] (.5,-1) rectangle (1,2);
    \fill[greenfill] (1,-1) rectangle (3,2);
    \fill[redfill] (3,-1) rectangle (3.5,2);
    \draw (1,-1) -- node[ed,near end] {$M$} (1,2);
    \draw (3,-1) -- node[ed,near end,swap] {$N$} (3,2);
    \path[dotsF] (.5,0) -- (3.5,0) -- (3.5,2) -- (2.5,2)
    to[out=-90,in=-90,looseness=5] (1.5,2) -- (0.5,2) -- cycle;
    \path[dotsG] (.5,-1) rectangle (3.5,0);
    \draw (.5,0) -- node[ed,above=3pt,near end] {$\alpha$} (3.5,0);
  \end{tikzpicture}
  \quad\raisebox{.9cm}{$=$}\quad
  \begin{tikzpicture}[yscale=.7]
    \fill[bluefill] (.5,-1) rectangle (1,2);
    \fill[greenfill] (1,-1) rectangle (3,2);
    \fill[redfill] (3,-1) rectangle (3.5,2);
    \draw (1,-1) -- node[ed,near start] {$M$} (1,2);
    \draw (3,-1) -- node[ed,near start,swap] {$N$} (3,2);
    \path[dotsG] (.5,-1) -- (3.5,-1) -- (3.5,1) -- (2.5,1)
    to[out=-90,in=-90,looseness=4] (1.5,1) -- (0.5,1) -- cycle;
    \path[dotsF] (.5,1) rectangle (1.5,2);
    \path[dotsF] (2.5,1) rectangle (3.5,2);
    \draw (.5,1) -- node[ed,above=3pt,near end] {$\alpha$} (1.5,1)
    -- node[ed,above=3pt,very near end] {$\alpha$} (3.5,1);
  \end{tikzpicture}
  \caption{The axioms of a conjunctional transformation}
  \label{fig:ctax}
\end{figure}

When interpreting the final equation in \autoref{fig:ctax}, note that according to the conventions for functors established in \S\ref{sec:funct-bicat},
\raisebox{-3mm}{\begin{tikzpicture}[scale=.5]
  \fill[bluefill] (0,-1) rectangle ++(1,2);
  \fill[greenfill] (1,-1) rectangle ++(2,2);
  \fill[redfill] (3,-1) rectangle ++(1,2);
  \path[dotsF] (0,0) rectangle ++(4,1);
  \path[dotsG] (0,-1) rectangle ++(4,1);
  \draw (0,0) -- node[ed,above=2pt] {$\alpha$} ++(4,0);
  \draw (1,1) -- node[ed,near end,swap] {$M$} ++(0,-2);
  \draw (3,1) -- node[ed,near end] {$N$} ++(0,-2);
\end{tikzpicture}}
denotes $\al_{M\odot N}$
, while
\raisebox{-3mm}{\begin{tikzpicture}[scale=.5]
  \fill[bluefill] (0,-1) rectangle ++(1,2);
  \fill[greenfill] (1,-1) rectangle ++(3,2);
  \fill[redfill] (4,-1) rectangle ++(1,2);
  \path[dotsF] (0,0) rectangle ++(2,1);
  \path[dotsG] (0,-1) rectangle ++(2,1);
  \path[dotsF] (3,0) rectangle ++(2,1);
  \path[dotsG] (3,-1) rectangle ++(2,1);
  \draw (0,0) -- node[ed,above=2pt] {$\alpha$} ++(5,0);
  \draw (1,1) -- node[ed,near end,swap] {$M$} ++(0,-2);
  \draw (4,1) -- node[ed,near end] {$N$} ++(0,-2);
\end{tikzpicture}}
denotes $\left(\al_M \odot \id_{G(N)}\right)\circ\left(\id_{F(M)}\odot \al_N\right)$.

\begin{rmk}
  The axioms of a conjunctional transformation say precisely that the 1- and 2-cell components $\al_R$ and $\al_M$ form an ``oplax natural transformation,'' or equivalently that the dual components $\ral R$ and $\ral M$ form a ``lax natural transformation.''
  We call this a \emph{conjunctional transformation} because it is a ``conjoint pair'' in the double category of lax and oplax natural transformations; see for instance~\cite{csmb}.
\end{rmk}

Before we can state and prove an analogue of \autoref{thm:trans-pres-tr}, we need one further observation.
Namely, if $Q\colon R\hto R$ is an endo-1-cell, then since $\al_R$ is right dualizable, we can take the trace of the 2-cell $\al_Q\colon F(Q) \odot \al_R \to \al_R \odot G(Q)$ to obtain a morphism $\tr(\al_Q)\colon \sh{F(Q)} \to \sh{G(Q)}$.
We naturally depict this as in \autoref{fig:trconj}.

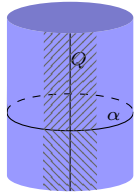
\begin{figure}[hbt]
  \centering
  \begin{tikzpicture}[scale=.7]
    \bgcylinder{0,0}{3}{1.2}{.3}{blue}{blue}
    \draw (top) -- node[ed,near start] {$Q$} (bot);
    \coordinate (mid) at ($(dl)!.4!(ul)$);
    \elltheta{mid}{0}{theta}
    \begin{scope}
      \draw[clip] ($(dl)!.4!(ul)$) to[out=-90,in=-90,looseness=.5]
      node [ed,near end,above=2pt] {$\alpha$} ($(dr)!.4!(ur)$)
      -- ++(.1,0) -- (ur') -- (ul') -- cycle;
      \path[dotsF] (top) -- ++(-.5,0) -- ++(0,-3) -- ++(1,0) -- ++(0,3) -- cycle;
    \end{scope}
    \begin{scope}
      \clip ($(dl)!.4!(ul)$) to[out=-90,in=-90,looseness=.5] ($(dr)!.4!(ur)$)
      -- ++(.1,0) -- (dr') |- (bot') -| (dl') -- cycle;
      \path[dotsG] ($(bot)+(-.5,-1)$) -- ++(0,3) -- ++(1,0) -- ++(0,-3) -- cycle;
    \end{scope}
  \end{tikzpicture}
  \caption{The trace of a conjunctional transformation}
  \label{fig:trconj}
\end{figure}

\begin{thm}\label{thm:trans-pres-tr-bicat}
  Let \sB\ and \sC\ be bicategories with shadows, $F, G\maps \sB\to\sC$ lax functors, and $\alpha\maps F\to G$ a conjunctional transformation, let $M\maps R\hto S$ be right dualizable in \sB, and suppose that $F$ and $G$ preserve its dual $\rdual{M}$.
  \begin{enumerate}
  \item Then $\al_M\colon F(M) \odot \al_S \to \al_R \odot G(M)$ is an isomorphism, whose inverse is the mate of $\al_{\rdual{M}}$ under the dual pairs $(F(M),F(\rdual{M}))$ and $(G(M),G(\rdual{M}))$.\label{item:tptb1}
  \end{enumerate}
  Assume furthermore that $\fc\colon F(M)\odot F(P)\to F(M\odot P)$ and $\fc\colon G(M)\odot G(P)\to G(M\odot P)$ are isomorphisms.
  \begin{enumerate}[resume]
  \item Then for any 2-cell $f\maps Q\odot M\to M\odot P$, the following square commutes.\label{item:tptb2}
    \[\xymatrix@C=5pc{\sh{F(Q)} \ar[r]^{\tr(\fc^{-1}\circ Ff\circ\fc)}\ar[d]_{\tr(\al_Q)} &
      \sh{F(P)} \ar[d]^{\tr(\al_P)}\\
      \sh{G(Q)}\ar[r]_{\tr(\fc^{-1}\circ Gf\circ\fc)} & \sh{G(P)}}\]
  \end{enumerate}
\end{thm}
\begin{proof}
  For~\ref{item:tptb1}, \autoref{fig:dualsinvert-bicat} displays the equality asserting that one composite of $\al_M$ with its putative inverse is the identity.
  To prove it we simply ``pull the lower loop through the $\alpha$ membrane,'' using the axioms from \autoref{fig:ctax} (and the consequent relations for the inverses of \fc\ and \fii, when these transformations are invertible) and then straighten the $F(M)$ string using a triangle identity.
  The other equality is similar.
  (This part does not require $\alpha$ to be conjunctional, only oplax.
  A corresponding result is true when $\alpha$ is lax and $M$ is left dualizable.)

  \begin{figure}[tbp]
    \centering
    \begin{tikzpicture}[scale=.8]
      \clip (-2,0) rectangle (4,5);
      \fill[bluefill] (-2,0) rectangle (4,5);
      \draw[greenfill] (-1,0) to[out=90,in=-90] node[ed] {$M$} (-.2,3) arc (180:0:.5)
      to[out=-90,in=90] node[ed,near end] {$\rdual{M}$} (1.8,1) arc (-180:0:.5)
      to[out=90,in=-90] (3,3) node[ed,right] {$M$} to[out=90,in=-90] (2,5.1)
      -- ++(0,.1) -- (5.1,5.1) -- (5.1,-.1) -- (0,-.1) -- cycle;
      \draw (0,0) -- node[ed,pos=.6] {$\alpha$} ++(3,5);
      \begin{scope}
        \clip (-2,0) -- (0,0) -- (3,5) -- (-2,5) -- cycle;
        \path[dotsF] (-1.3,0) to[out=90,in=-90] (-.5,3) arc (180:0:.8)
        to[out=-90,in=90] (2.1,1) -- (1.5,1) to[out=90,in=-90] (.6,3) arc (0:180:.3)
        to[out=-90,in=90] (-.7,0) -- cycle;
        \path[dotsF] (1.7,5) to[out=-90,in=90] (2.7,3) -- (3.3,3) to[out=90,in=-90] (2.3,5) -- cycle;
      \end{scope}
      \begin{scope}
        \clip (0,0) -- (3,5) -- (4,5) -- (4,0) -- cycle;
        \path[dotsG] (1.7,5) to[out=-90,in=90] (2.7,3)
        to[out=-90,in=90] (2.6,1) arc (0:-180:.3)
        to[out=90,in=-90] (1.1,3) -- (.5,3) to[out=-90,in=90] (1.5,1)
        arc (-180:0:.8) to[out=90,in=-90]
        (3.3,3) to[out=90,in=-90] (2.4,5) -- cycle;
      \end{scope}
    \end{tikzpicture}
    \quad\raisebox{2cm}{$=$}\quad
    \begin{tikzpicture}[scale=.8]
      \fill[bluefill] (0,0) rectangle (1.5,5);
      \fill[greenfill] (1.5,0) rectangle (4,5);
      \draw (3,0) -- node[ed,near end,swap] {$\alpha$} (3,5);
      \draw (1.5,0) -- node[ed,near end,swap] {$M$} (1.5,5);
      \path[dotsF] (1.1,0) -- ++(0,5) -- ++(.8,0) -- ++(0,-5) -- cycle;
    \end{tikzpicture}
    \caption{Duals invert a bicategorical transformation}
    \label{fig:dualsinvert-bicat}
  \end{figure}
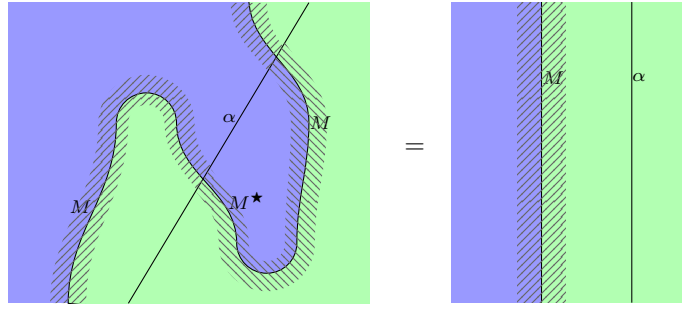

  The proof of~\ref{item:tptb2} is the first (and only) time in the paper that we do not need to invoke the actual definition of the bicategorical trace: all we need are the properties proven in \S\ref{sec:prop-bicat} and the axioms of a conjunctional transformation.
  Thus, we can now simplify the pictures by drawing the bicategorical trace with a plain loop around the back of the cylinder, rather than breaking it down into coevaluation, shadow, and evaluation.
  With this convention, the proof is shown in \autoref{fig:2func-bicat}.
  In the first step, we apply \autoref{thm:tr-vfunc} to replace $\tr(\alpha_P)\circ\tr(\fc^{-1}\circ Ff\circ\fc)$ by the trace of a composite.
  Secondly, we use part~\ref{item:tptb1} to introduce an instance of $\alpha_M$ and its inverse.
  In the third step, we use the axioms of \autoref{fig:ctax} to slide the $\alpha$ string across $\fc^{-1}\circ Ff\circ\fc$.
  We then use \autoref{thm:tr-cyclnat} to bring $\alpha_M$ around the back of the cylinder, so that we can cancel it with its inverse.
  Finally, we apply \autoref{thm:tr-vfunc} again to obtain $\tr(\fc^{-1}\circ Gf\circ\fc)\circ\tr(\alpha_Q)$, as desired.
\end{proof}

  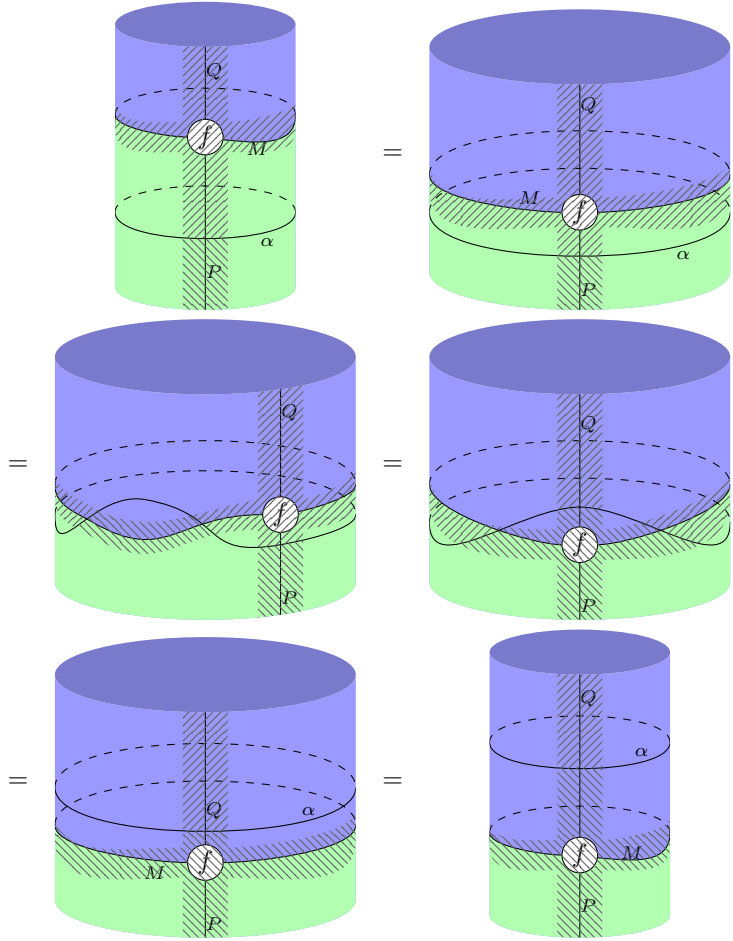
\begin{figure}[tbp]
    \centering
    \begin{tabular}{cccc}
      &
    \begin{tikzpicture}
      \bgcylinder{-1.2,0}{3.5}{1.2}{.3}{blue}{blue}
      \begin{pgfonlayer}{foreground}
        \node[vert] (f) at (0,2) {$f$};
        \elltheta{f}{.3}{theta}
        \elltheta{thetaL}{-1.3}{theta2}
      \end{pgfonlayer}
      \draw[greenfill] (thetaL) to[out=-90,in=180,looseness=.5] (f) to[out=0,in=-90] node[ed,swap] {$M$} (thetaR)
      -- ++(.1,0) -- (dr') |- (bot') -| (dl') -- (thetaL);
      \draw (f |- top) -- node[ed] {$Q$} (f) -- node[ed,near end] {$P$} (f |- bot);
      \draw (theta2L) to[out=-90,in=-90,looseness=.5] node[ed,near end,below=2pt] {$\alpha$} (theta2R);
      \begin{pgfonlayer}{foreground}
      \begin{scope}
        \clip (theta2L) to[out=-90,in=-90,looseness=.5] (theta2R) -- (ur) -- (ul) -- cycle;
        \path[dotsF] (top) ++(.3,0) -- ($(f)+(.3,.2)$) to[out=0,in=-90,looseness=.5] ($(thetaR)+(0,.2)$)
        -- ++(0,-.4) to[out=-90,in=0] ($(f)+(.3,-.2)$) -- ($(bot)+(.3,0)$) -- ++(-.6,0)
        -- ($(f)+(-.3,-.2)$) to[out=180,in=-90] ($(thetaL)+(0,-.2)$) -- ++(0,.4)
        to[out=-90,in=180] ($(f)+(-.3,.2)$) -- ($(top)+(-.3,0)$) -- cycle;
      \end{scope}
      \begin{scope}
        \clip (theta2L) to[out=-90,in=-90,looseness=.5] (theta2R) -- (dr) |- (bot) -| (dl) -- cycle;
        \path[dotsG] (bot) -- ++(.3,0) -- ++(0,2) -- ++(-.6,0) -- ++(0,-2) -- cycle;
      \end{scope}
      \end{pgfonlayer}
    \end{tikzpicture}
    & \raisebox{2cm}{$=$} &
    \begin{tikzpicture}
      \bgcylinder{-2,.2}{3}{2}{.5}{blue}{blue}
      \begin{pgfonlayer}{foreground}
        \node[vert] (f) at (0,1) {$f$};
        \elltheta{f}{.5}{theta}
        \elltheta{f}{0}{theta2}
      \end{pgfonlayer}
      \draw[greenfill] (f) to[out=180,in=-90,looseness=.5] node[ed,near start,above=2pt] {$M$} (thetaL)
      -- ++(-.1,0) -- (dl') |- (bot') -| (dr') -- ($(thetaR)+(.1,0)$)
      -- (thetaR) to[out=-90,in=0,looseness=.5] (f);
      \draw (theta2L) to[out=-90,in=-90,looseness=.5] node[ed,near end,below=2pt] {$\alpha$} (theta2R);
      \draw (f |- top) -- node[ed,pos=.4] {$Q$} (f) -- node[ed,near end] {$P$} (f |- bot);
      \begin{pgfonlayer}{foreground}
      \begin{scope}
        \clip (theta2L) to[out=-90,in=-90,looseness=.5] (theta2R) -- (ur) -- (ul) -- cycle;
        \path[dotsF] (top) ++(.3,0) -- ($(f)+(.3,.2)$) to[out=0,in=-90,looseness=.5] ($(thetaR)+(0,.2)$)
        -- ++(0,-.4) to[out=-90,in=0] ($(f)+(.3,-.2)$) -- ($(bot)+(.3,0)$) -- ++(-.6,0)
        -- ($(f)+(-.3,-.2)$) to[out=180,in=-90] ($(thetaL)+(0,-.2)$) -- ++(0,.4)
        to[out=-90,in=180] ($(f)+(-.3,.2)$) -- ($(top)+(-.3,0)$) -- cycle;
      \end{scope}
      \begin{scope}
        \clip (theta2L) to[out=-90,in=-90,looseness=.5] (theta2R) -- (dr) |- (bot) -| (dl) -- cycle;
        \path[dotsG] (bot) -- ++(.3,0) -- ++(0,2) -- ++(-.6,0) -- ++(0,-2) -- cycle;
      \end{scope}
      \end{pgfonlayer}
    \end{tikzpicture}
    \\
    \raisebox{2cm}{$=$} &
    \begin{tikzpicture}
      \bgcylinder{-2,.2}{3}{2}{.5}{blue}{blue}
      \begin{pgfonlayer}{foreground}
        \node[vert] (f) at (1,1.1) {$f$};
        \elltheta{f}{.4}{theta}
        \elltheta{f}{0}{theta2}
      \end{pgfonlayer}
      \draw[greenfill] (f)
      to[out=10,in=-90,looseness=.5] (thetaR) 
      -- ++(.1,0) -- (dr') |- (bot') -| (dl') -- ($(thetaL)+(-.1,0)$)
      -- (thetaL)
      to[rrel,out=-90,in=150,looseness=1] (.5,-.5) node[coordinate] (Xabove) {}
      to[out=-30,in=180,looseness=1.3] (f);
      \draw (f |- top) -- node[ed,pos=.4] {$Q$} (f) -- node[ed,near end] {$P$} (f |- bot);
      \def\alphaline{($(theta2L)+(-.1,0)$) -- (theta2L)
        to[rrel,out=-90,in=170,looseness=1.5] (1.2,.2) node[coordinate] (Xbelow) {}
        to[out=-10,in=190,looseness=1.5] ($(f)+(0,-.4)$)
        to[out=10,in=-90,looseness=.5] (theta2R)
        -- ++(.1,0)}
      \def\loweredpath{($(f)+(-.3,-.2)$)
        to[out=180,in=-30,looseness=1.3] ($(Xabove)+(0,-.2)$)
        to[out=150,in=-90,looseness=1] ($(thetaL)+(0,-.2)$) -- ++(0,.4)
        to[out=-90,in=150,looseness=1] ($(Xabove)+(0,.2)$)
        to[out=-30,in=190] ($(f)+(-.3,.2)$)}
      \begin{pgfonlayer}{foreground}
      \begin{scope}
        \draw[clip] \alphaline -- (ur') -- (top') -- (ul') -- cycle;
        \path[dotsF] (f |- top) ++(.3,0) -- ($(f)+(.3,.2)$) to[out=10,in=-90,looseness=.5] ($(thetaR)+(0,.2)$)
        -- ++(0,-.4) to[out=-90,in=10] ($(f)+(.3,-.2)$) -- ($(f |- bot)+(.3,0)$) -- ++(-.6,0)
        -- \loweredpath -- ($(f |- top)+(-.3,0)$) -- cycle;
      \end{scope}
      \begin{scope}
        \clip \alphaline -- (dr) |- (bot) -| (dl) -- cycle;
        \path[dotsG] (f |- bot) -- ++(.3,0) -- ++(0,2) -- ++(-.6,0) -- ++(0,-2) -- cycle;
        \path[dotsG] \loweredpath -- cycle;
      \end{scope}
      \end{pgfonlayer}
    \end{tikzpicture}
    &\raisebox{2cm}{$=$}&
    \begin{tikzpicture}
      \bgcylinder{-2,.2}{3}{2}{.5}{blue}{blue}
      \begin{pgfonlayer}{foreground}
        \node[vert] (f) at (0,.7) {$f$};
        \elltheta{f}{.8}{theta}
        \elltheta{f}{.3}{theta2}
      \end{pgfonlayer}
      \draw[greenfill] (f) to[out=180,in=-90,looseness=.5] (thetaL)
      -- ++(-.1,0) -- (dl') |- (bot') -| (dr') -- ($(thetaR)+(.1,0)$)
      -- (thetaR) to[out=-90,in=0,looseness=.5] (f);
      \draw (f |- top) -- node[ed,pos=.4] {$Q$} (f) -- node[ed,near end] {$P$} (f |- bot);
      \def\alphaline{($(thetaL)+(-.1,0)$) --
        (theta2L) to[out=-90,in=180,looseness=1] ($(f)+(0,.5)$)
        to[out=0,in=-90,looseness=1] (theta2R) -- ++(.1,0)}
      \def\shading{
        (top) ++(.3,0) -- ($(f)+(.3,.2)$) to[out=0,in=-90,looseness=.5] ($(thetaR)+(0,.2)$)
        -- ++(0,-.4) to[out=-90,in=0] ($(f)+(.3,-.2)$) -- ($(bot)+(.3,0)$) -- ++(-.6,0)
        -- ($(f)+(-.3,-.2)$) to[out=180,in=-90] ($(thetaL)+(0,-.2)$) -- ++(0,.4)
        to[out=-90,in=180] ($(f)+(-.3,.2)$) -- ($(top)+(-.3,0)$) -- cycle}
      \begin{pgfonlayer}{foreground}
      \begin{scope}
        \draw[clip] \alphaline -- (ur') -- (ul') -- cycle;
        \path[dotsF] \shading ;
      \end{scope}
      \begin{scope}
        \clip \alphaline -- (dr) |- (bot) -| (dl) -- cycle;
        \path[dotsG] \shading ;
      \end{scope}
      \end{pgfonlayer}
    \end{tikzpicture}
    \\
    \raisebox{2cm}{$=$}&
    \begin{tikzpicture}
      \bgcylinder{-2,.2}{3}{2}{.5}{blue}{blue}
      \begin{pgfonlayer}{foreground}
        \node[vert] (f) at (0,.7) {$f$};
        \elltheta{f}{.5}{theta}
        \elltheta{f}{1}{theta2}
      \end{pgfonlayer}
      \draw[greenfill] (f) to[out=180,in=-90,looseness=.5] node[ed,near start,below=2pt] {$M$} (thetaL)
      -- ++(-.1,0) -- (dl') |- (bot') -| (dr') -- ($(thetaR)+(.1,0)$)
      -- (thetaR) to[out=-90,in=0,looseness=.5] (f);
      \draw (theta2L) to[out=-90,in=-90,looseness=.5] node[ed,near end,above=2pt] {$\alpha$} (theta2R);
      \draw (f |- top) -- node[ed,pos=.8] {$Q$} (f) -- node[ed,near end] {$P$} (f |- bot);
      \begin{pgfonlayer}{foreground}
      \begin{scope}
        \clip (theta2L) to[out=-90,in=-90,looseness=.5] (theta2R) -- (ur) -- (ul) -- cycle;
        \path[dotsF] (top) -- ++(.3,0) -- ++(0,-4) -- ++(-.6,0) -- ++(0,4) -- cycle;
      \end{scope}
      \begin{scope}
        \clip (theta2L) to[out=-90,in=-90,looseness=.5] (theta2R) -- (dr) |- (bot) -| (dl) -- cycle;
        \path[dotsG] (top) ++(.3,0) -- ($(f)+(.3,.2)$) to[out=0,in=-90,looseness=.5] ($(thetaR)+(0,.2)$)
        -- ++(0,-.4) to[out=-90,in=0] ($(f)+(.3,-.2)$) -- ($(bot)+(.3,0)$) -- ++(-.6,0)
        -- ($(f)+(-.3,-.2)$) to[out=180,in=-90] ($(thetaL)+(0,-.2)$) -- ++(0,.4)
        to[out=-90,in=180] ($(f)+(-.3,.2)$) -- ($(top)+(-.3,0)$) -- cycle;
      \end{scope}
      \end{pgfonlayer}
    \end{tikzpicture}
    &\raisebox{2cm}{$=$}&
    \begin{tikzpicture}
      \bgcylinder{-1.2,0}{3.5}{1.2}{.3}{blue}{blue}
      \begin{pgfonlayer}{foreground}
        \node[vert] (f) at (0,.8) {$f$};
        \elltheta{f}{.3}{theta}
        \elltheta{thetaL}{1.2}{theta2}
      \end{pgfonlayer}
      \draw[greenfill] (thetaL) to[out=-90,in=180,looseness=.5] (f) to[out=0,in=-90] node[ed] {$M$} (thetaR)
      -- ++(.1,0) -- (dr') |- (bot') -| (dl') -- (thetaL);
      \draw (f |- top) -- node[ed,near start] {$Q$} (f) -- node[ed] {$P$} (f |- bot);
      \draw (theta2L) to[out=-90,in=-90,looseness=.5] node[ed,near end,above=2pt] {$\alpha$} (theta2R);
      \begin{pgfonlayer}{foreground}
      \begin{scope}
        \clip (theta2L) to[out=-90,in=-90,looseness=.5] (theta2R) -- (ur) -- (ul) -- cycle;
        \path[dotsF] (top) -- ++(.3,0) -- ++(0,-4) -- ++(-.6,0) -- ++(0,4) -- cycle;
      \end{scope}
      \begin{scope}
        \clip (theta2L) to[out=-90,in=-90,looseness=.5] (theta2R) -- (dr) |- (bot) -| (dl) -- cycle;
        \path[dotsG] (top) ++(.3,0) -- ($(f)+(.3,.2)$) to[out=0,in=-90,looseness=.5] ($(thetaR)+(0,.2)$)
        -- ++(0,-.4) to[out=-90,in=0] ($(f)+(.3,-.2)$) -- ($(bot)+(.3,0)$) -- ++(-.6,0)
        -- ($(f)+(-.3,-.2)$) to[out=180,in=-90] ($(thetaL)+(0,-.2)$) -- ++(0,.4)
        to[out=-90,in=180] ($(f)+(-.3,.2)$) -- ($(top)+(-.3,0)$) -- cycle;
      \end{scope}
      \end{pgfonlayer}
    \end{tikzpicture}
    \end{tabular}
    \caption{2-functoriality of the bicategorical trace}
    \label{fig:2func-bicat}
  \end{figure}

When $F$ and $G$ are additionally shadow functors, we can also compare $\Ft(\tr(f))$ and $\Gt(\tr(f))$, but we need some extra structure on \al.

\begin{defn}
  Let $F,G\colon\sB\to\sC$ be lax shadow functors.
  A \textbf{shadow conjunctional transformation} $\al\colon F\to G$ is a conjunctional transformation together with a natural transformation
  \[ \alt \colon \Ft \to \Gt\]
  such that a coherence axiom is satisfied.
\end{defn}

We draw the component
\[(\alt)_{\scriptsh{M}}\colon \Ft\big(\sh{M}\big) \to \Gt\big(\sh{M}\big)\]
as in \autoref{fig:shtrans}.
Of course, despite what the picture may suggest, $\alt$ is not actually the trace of anything.

\begin{figure}[htb]
  \centering
  \subfigure[Data]{
  \begin{tikzpicture}
    \bgcylinder{0,0}{3}{1.3}{.3}{blue}{blue}
    \elltheta{1.3,1}{.3}{theta}
    \draw (top) -- node[ed,near start] {$M$} (bot);
    \draw[dotsF] (thetaL) to[out=-90,in=-90,looseness=.5] node[ed,near end,above=2pt] {$\alpha$} (thetaR)
    -- ++(.1,0) -- (ur') -- (top') -- (ul') -- ($(thetaL)+(-.1,0)$) -- cycle;
    \path[dotsG] (thetaL) to[out=-90,in=-90,looseness=.5] (thetaR)
    -- ++(.1,0) -- (dr') |- (bot') -| (dl') -- ($(thetaL)+(-.1,0)$) -- cycle;
  \end{tikzpicture}}
  \hspace{2cm}
  \subfigure[Axiom]{\label{fig:shtransax}
  \begin{tikzpicture}[scale=.8]
    \bgcylinder{0,0}{4}{1.3}{.3}{blue}{blue}
    \elltheta{1.3,1}{.3}{theta}
    \draw[dotsG] (thetaL) to[out=-90,in=-90,looseness=.5] node[ed,near end,above=2pt] {$\alpha$} (thetaR)
    -- ++(.1,0) -- (dr') |- (bot') -| (dl') -- ($(thetaL)+(-.1,0)$) -- cycle;
    \draw (top) -- node[ed,near start] {$M$} (bot);
    \begin{scope}
      \clip (thetaL) to[out=-90,in=-90,looseness=.5] (thetaR) -- (ur) -- (ul) -- cycle;
      \path[dotsF] (top) ++ (-.5,0) to[out=-90,in=90,looseness=1.5] (0,2) -- (dl) -- (dr) -- ++(0,2)
      to[out=90,in=-90,looseness=1.5] ($(top)+(.5,0)$) -- cycle;
    \end{scope}
  \end{tikzpicture}
  \quad\raisebox{1.8cm}{$=$}\quad
  \begin{tikzpicture}[scale=.8]
    \bgcylinder{0,0}{4}{1.3}{.3}{blue}{blue}
    \elltheta{1.3,2.5}{.3}{theta}
    \draw (top) -- node[ed,near end] {$M$} (bot);
    \begin{scope}
      \draw[clip] (thetaL) to[out=-90,in=-90,looseness=.5] node[ed,near end,above=2pt] {$\alpha$} (thetaR)
      -- ++(.1,0) -- (ur') -- (ul') -- ($(thetaL)+(-.1,0)$) -- cycle;
      \path[dotsF] (top) -- ++(-.5,0) -- ++(0,-3) -- ++(1,0) -- ++(0,3) -- cycle;
    \end{scope}
    \begin{scope}
      \path[clip] (thetaL) to[out=-90,in=-90,looseness=.5] node[ed,near end,above=2pt] {$\alpha$} (thetaR)
      -- (dr') |- (bot') -| (dl') -- cycle;
      \path[dotsG] (top) -- ++(-.5,0) -- ++(0,-2) coordinate (x)
      to[out=-90,in=90,looseness=1.5] (0,.5) -- (dl) |- (bot) -| (dr) -- ++(0,.5)
      to[out=90,in=-90] ($(x)+(1,0)$) -- ++(0,2) -- cycle;
    \end{scope}
  \end{tikzpicture}}
  \caption{A shadow conjunctional transformation}
  \label{fig:shtrans}
\end{figure}
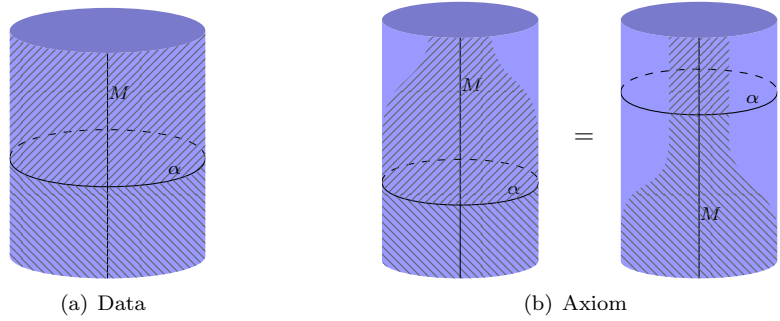

We can now observe that if $\alt$ is a shadow conjunctional transformation, $M$ is right dualizable, $f\colon Q\odot M \to M \odot P$ is a 2-cell, and \autoref{thm:bicat-funct-pres-tr} applies to $F$ and to $G$, then all six faces of the cube shown in \autoref{fig:cube} commute.
The back face is \autoref{thm:trans-pres-tr-bicat}, the front is just naturality of \alt, the top and bottom faces are \autoref{thm:bicat-funct-pres-tr} for $F$ and $G$ respectively, and the left and right faces are the axiom of \alt\ from Figure~\ref{fig:shtransax}.

\begin{figure}
  \centering
  \begin{equation*}
  \xymatrix{
  \sh{F(Q)} \ar[rr]^{\tr(Ff)}\ar[dd]_{\sh{\al_Q}} \ar[dr]^\phi &&
  \sh{F(P)} \ar'[d][dd]^(.4){\sh{\al_P}} \ar[dr]^\phi\\
  &\Ft\;\sh{Q} \ar[rr]^(.33){\Ft(\tr(f))} \ar[dd]_(.3){(\alt)_{\scriptsh{Q}}} &&
  \Ft\;\sh{P} \ar[dd]^{(\alt)_{\scriptsh{P}}}\\
  \sh{G(Q)}\ar'[r][rr]_(.4){\tr(Gf)} \ar[dr]_\phi && \sh{G(P)} \ar[dr]^\phi\\
  &\Gt\;\sh{Q}\ar[rr]_{\Gt(\tr(f))} && \Gt\;\sh{P}}\label{eq:trace-cube}
\end{equation*}
\caption{The cube of traces}
\label{fig:cube}
\end{figure}
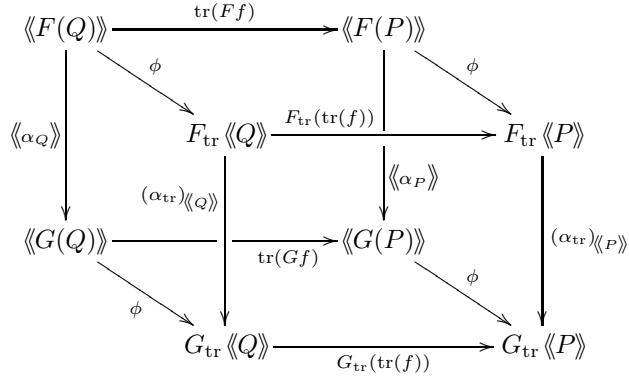

\begin{eg}
  As observed in \autoref{rmk:why-conjunctional}, we have two lax shadow functors $C_*(-;\bbZ)$ and $C_*(-;\bbQ)$ from $\Ho(\bgpsp)$ to $\Ho(\calCh)$, and the inclusion $\bbZ\rightarrow \bbQ$ then defines a conjunctional transformation from the first to the second.
  For any $P\colon R\hto R$ in $\Ho(\bgpsp)$, the shadow $\tr(\al_P)$ is just the induced map $\sh{C_*(P;\bbZ)} \to \sh{C_*(P;\bbQ)}$.
  Therefore, \autoref{thm:trans-pres-tr-bicat} implies that this map identifies the Reidemeister trace, computed as in \autoref{thm:reidemeister-2} using the rational chain complex, with the analogous version computed using the integral chain complex.

  This transformation \al\ is also a shadow conjunctional transformation; the map $\alt\colon C_*(-;\bbZ)_{\tr} \to C_*(-;\bbQ)_{\tr}$ simply includes the ordinary integral chain complex of a spectrum into its rational one.
  Thus we have four different ways to compute the Reidemeister trace, and the cube in \autoref{fig:cube} shows that they are all equivalent.
\end{eg}

\appendix
\section{String Diagrams}
\label{sec:string-diagrams}

Our goal in this appendix is to prove the validity of reasoning with cylinder diagrams for bicategories with shadows, such as we have been using throughout the paper.
The proof is a fairly straightforward adaptation of that given by Joyal and Street~\cite{js:geom-tenscalc-i} for monoidal categories.
(The version for plain bicategories, as in \S\ref{sec:bicategories}, seems not to be treated formally anywhere in the literature, but it is ``folklore'' among category theorists and the proof is an even more straightforward adaptation, so we leave that case to the reader.)
We will not attempt to formalize the string diagrams for shadow functors and transformations used in \S\S\ref{sec:funct-bicat}--\ref{sec:transf}, but we expect them to be a similarly straightforward adaptation once string diagrams for \emph{monoidal} functors and transformations have been formalized.

\subsection{Progressive cylinder graphs}

\begin{defn}
A \textbf{(topological) graph} $\Gamma = (\Gamma, \Gamma_0)$ consists of:
\begin{enumerate}
\item a compact Hausdorff space $\Gamma$, and
\item a discrete closed subset $\Gamma_0\subset \Gamma$, whose complement $\Gamma \setminus \Gamma_0$ is a 1-dimensional manifold without boundary (not in general connected).
\end{enumerate}
The points in $\Gamma_0$ are called the \textbf{vertices} of $\Gamma$, and the connected components of $\Gamma \setminus \Gamma_0$ are called \textbf{edges}.
The \textbf{degree} of a vertex $x$ is the number of connected components of $V\setminus \{x\}$, where $V$ is a sufficiently small connected neighbourhood of $x$.
A \textbf{graph with boundary} is a graph as above together with
\begin{enumerate}[resume]
\item a distinguished subset $\partial\Gamma\subseteq \Gamma_0$ of vertices of $\Gamma$ of degree one.
\end{enumerate}
The elements of $\partial \Gamma$ are called \textbf{outer vertices} and those of $\Gamma_0 \setminus\partial \Gamma$ are called \textbf{inner vertices}.
\end{defn}

In what follows, we will consider $S^1$ as an oriented manifold, but not containing any distinguished points.
(In all of our pictures, we draw the orientation as clockwise.)
However, as noted in \autoref{rmk:noemb}, we do need basepoints on the top and bottom cylinders in order to assign well-defined values.

\begin{defn}
  Let $a < b$ be real numbers.
  A \textbf{progressive cylinder graph} (between the levels $a$ and $b$) is a graph $\Gamma$ with boundary embedded in $S^1\times [a, b]$, together with two points $\delta_a,\delta_b\in S^1$ called the \textbf{domain basepoint} and \textbf{codomain basepoint},  such that
\begin{enumerate}
\item $\partial \Gamma =\Gamma \cap (S^1\times \{a, b\})$,
\item $(\delta_a,a)\notin\Gamma\cap (S^1\times\{a\})$ and $(\delta_b,b)\notin\Gamma\cap (S^1\times\{b\})$, and
\item the second coordinate projection $\proj{2}\colon S^1\times
 [a, b] \rightarrow  [a, b]$ is injective on each edge.\label{item:cylgr3}
\end{enumerate}
\end{defn}

Condition~\ref{item:cylgr3} implies that each edge of $\Gamma$ is homeomorphic to an open interval and acquires a linear ordering from $\proj2$.
Using this ordering we can define the \textbf{source} of an edge as the unique vertex in the closure of its first half, and its \textbf{target} as the vertex in the closure of its second half.
This notion of source and target should not be confused with the domain and codomain of any sort of morphism or cell in a category or bicategory, which is its ``Poincar\'e dual.''

Specifically, the \textbf{domain} $\mathrm{dom}(\Gamma)$ of $\Gamma$ consists of the edges which have
outer vertices as sources, while its \textbf{codomain} $\mathrm{cod}(\Gamma)$ consists of
the edges which have outer vertices as targets.
Note that a given edge could be in both the domain and the codomain, and that our convention of drawing domains at the top and codomains at the bottom means that the real-valued second projection $\proj2$ of our diagrams increases \emph{down} the page.

The set of sources of edges in $\mathrm{dom}(\Gamma)$ is equal to $\Gamma \cap (S^1\times \{a\})$, which does not contain $(\delta_a,a)$.
Since $(S^1\times \{a\}) \setminus \{(\delta_a,a)\}$ is an oriented 1-manifold homeomorphic to an interval, it is linearly ordered, so $\mathrm{dom}(\Gamma)$ also acquires a linear ordering.
Similarly, $\mathrm{cod}(\Gamma)$ acquires a linear ordering from $(S^1\times \{b\})\setminus \{(\delta_b,b)\}$.

For an inner vertex $x$, its \textbf{input} $\mathrm{in}(x)$ is the set of edges of $\Gamma$ that have target $x$.
Similarly, its \textbf{output} $\mathrm{out}(x)$ is the set of edges that have source $x$.
These two sets are linearly ordered as follows.
Choose any point $z \in S^1 \setminus \{\proj2(x)\}$ (think of it as ``on the other side'' of the circle) and let $u\in [a,\proj{2}(x)]$ be such that
\[
\Big(\mathrm{in}(x)\cap \big(S^1\times [u,\proj{2}(x)]\big)\Big) \cap \Big(\{z\}\times[a,b]\Big) \,=\, \emptyset.
\]
Then each edge $e \in\mathrm{in}(x)$ intersects
$\left(S^1\setminus \{z\}\right)\times \{u\}$ in one point, and different edges intersect it in different points.
The linear order on $S^1\setminus \{z\}$ now induces one on $\mathrm{in}(x)$, which is independent of the choices of $z$ and $u$.
The order on $\mathrm{out}(x)$ is defined similarly, by intersecting with $S^1\times \{u\}$ for $u$ larger than but close to $\proj{2}(x)$.

\subsection{Decomposition of progressive cylinder graphs}

\begin{defn}
  Let $\Gamma$ be a progressive cylinder graph such that $\delta_a=\delta_b$.
  A \textbf{tensor decomposition} of $\Gamma$ is an $n$-tuple $(\Gamma^1, \Gamma^2, \ldots, \Gamma^n)$ such that
\begin{itemize} 
\item $\Gamma$ is the disjoint union of subgraphs $\Gamma^1,\Gamma^2,\ldots, \Gamma^n$, 
\item there are pairwise disjoint, connected, open subsets $U^1,U^2,\ldots, U^n$ in $S^1$ such that
\[\Gamma^i\subset U^i\times[a,b],\]
\item we have
  \[\delta_a=\delta_b<U^1<U^2<\dots <U^n\]
  in the cyclic order on $S^1$ induced by its orientation.
\end{itemize}
\end{defn}
Given such a tensor decomposition, we write 
\[\Gamma =\sh{\Gamma^1\odot \ldots \odot \Gamma^n}.\]
We will also write $\delta$ for the common value of $\delta_a$ and $\delta_b$.

A connected graph with exactly one inner vertex is called \textbf{prime}.
A graph with no inner vertices is called \textbf{invertible}.
A graph $\Gamma$ is called \textbf{elementary} when it has a tensor decomposition
$\Gamma= \sh{\Gamma^1 \odot \ldots\odot \Gamma^n}$ with each $\Gamma^i$
($1 \leq i\leq n$) either prime or invertible, such as shown in \autoref{fig:elemdiag}.
Note that such a tensor decomposition is not unique.
Moreover, a graph can be prime or invertible without being elementary: this can happen if $\delta_a\neq\delta_b$, or if $\delta_a=\delta_b$ but $\big(\{\delta_a\} \times [a,b] \big) \cap \Gamma \neq\emptyset$, as shown in \autoref{fig:noneleminv}.

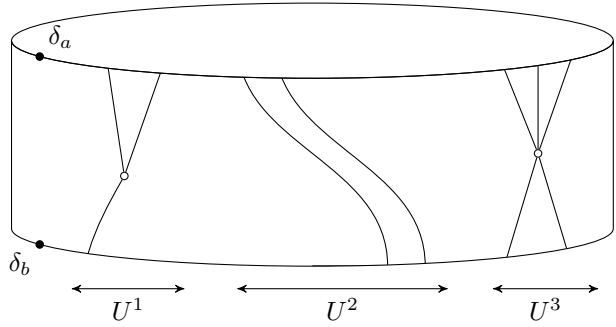
\begin{figure}
  \centering
  \begin{tikzpicture}[label distance=-2pt]
    \begin{scope}
      \bgcylinder{-4,0}{2.5}{4}{.5}{}{}
      \node[vert] (f) at (3,1) {};
      \draw (f) -- ++(-.8,2);
      \draw (f) -- ++(0,2);
      \draw (f) -- ++(.7,2);
      \draw (f) -- ++(-.6,-2);
      \draw (f) -- ++(.6,-2);
      \draw (bot) -- ++(1.5,0) to[out=90,in=-90] ($(top)+(-.5,0)$);
      \draw (bot) -- ++(1,0) to[out=90,in=-90] ($(top)+(-1,0)$);
      \node[vert] (g) at (-2.5,.7) {};
      \draw (g) -- ++(-.3,2);
      \draw (g) -- ++(.7,2);
      \draw (g) to[rrel,out=-120,in=90] (-.6,-2);
    \end{scope}
    \path (bot) ++(0,-.3) coordinate (Ubot);
    \draw[<->] (g |- Ubot) ++(-.7,0) -- node[below] {$U^1$} ++(1.5,0);
    \draw[<->] (Ubot) ++(-1,0) -- node[below] {$U^2$} ++(2.8,0);
    \draw[<->] (f |- Ubot) ++(-.6,0) -- node[below] {$U^3$} ++(1.4,0);
    \path (ul) arc (-180:-155:4 and .5) node[vert,fill=black,label=above right:$\delta_a$] {};
    \path (dl) arc (-180:-155:4 and .5) node[vert,fill=black,label=below left:$\delta_b$] {};
  \end{tikzpicture}
  \caption{Part of a tensor-decomposed elementary diagram}
  \label{fig:elemdiag}
\end{figure}

\begin{figure}
  \centering
  \begin{tikzpicture}[label distance=-2pt]
    \begin{scope}
      \bgcylinder{0,0}{2}{1.6}{.4}{}{}
      \draw (top) ++(-.3,0) -- ++(.2,-3);
      \draw (top) ++(-1.1,0) -- ++(-.1,-3);
      \draw (top) ++(1,0) -- ++(-.2,-3);
    \end{scope}
    \path (ul) arc (-180:-120:1.6 and .4) node[vert,fill=black,label=above right:$\delta_a$] {};
    \path (dl) arc (-180:-70:1.6 and .4) node[vert,fill=black,label=below:$\delta_b$] {};
  \end{tikzpicture}
  \hspace{2cm}
  \begin{tikzpicture}
    \begin{scope}
      \bgcylinder{0,0}{2.5}{1.6}{.4}{}{}
      \drawtheta{0,0}{1.5}{theta}
      \draw (top) ++(1,-.3) to[out=-90,in=90,looseness=.4] (thetaL);
      \draw (bot) ++(-1,.1) to[out=90,in=-90,looseness=.4] (thetaR);
    \end{scope}
    \path (ul) arc (-180:-90:1.6 and .4) node[vert,fill=black,label=above:$\delta_a$] {};
    \path (dl) arc (-180:-90:1.6 and .4) node[vert,fill=black,label=below:$\delta_b$] {};
  \end{tikzpicture}
  \caption{Some non-elementary invertible diagrams}
  \label{fig:noneleminv}
\end{figure}
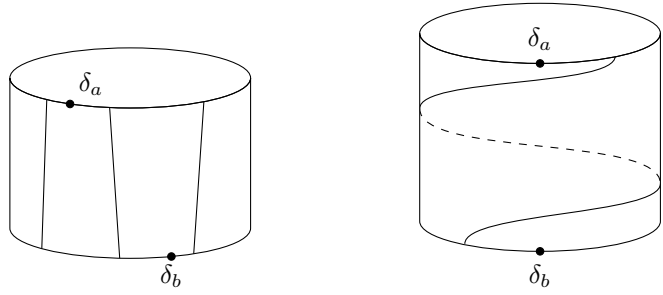

We now describe how to decompose an arbitrary graph into elementary and invertible ones.
For any progressive cylinder graph $\Gamma$, a number $u \in [a, b]$ is called a \textbf{regular level} for $\Gamma$ when the circle $S^1\times \{u\}$ contains no inner vertices.
If $c < d$ are regular levels of $\Gamma$ and $\delta_c$ and $\delta_d$ are points of $S^1$ such that
\begin{align*}
  (\delta_c,c) &\notin \Gamma\cap (S^1\times \{c\}) \qquad\text{and}\\
  (\delta_d,d) &\notin \Gamma\cap (S^1\times \{d\})
\end{align*}
we write $\Gamma[c,d]$ (leaving $\delta_c$ and $\delta_d$ implicit in the notation) for the graph $\Gamma \cap (S^1\times [c,d])$ with domain and codomain basepoints $\delta_c$ and $\delta_d$.
Its set of inner vertices is $(\Gamma_0\setminus\partial \Gamma)\cap (S^1 \times [c,d])$, and its set of outer vertices is $\Gamma \cap (S^1 \times \{c,d\})$.
It is a progressive cylinder graph between the levels $c$ and $d$, and is called a \textbf{layer} of $\Gamma$.

\begin{lem}\label{regularlevels}
For each progressive cylinder graph $\Gamma$ there are real numbers 
\[a=u_0<u_1<u_2<\ldots < u_n=b\]
and points $\delta_1,\dots,\delta_{n-1}\in S^1$ such that each $u_i$ is a regular level, and 
each layer $\Gamma[u_i,u_{i+1}]$ is either elementary or invertible.
\end{lem}

\begin{proof}
We first choose real numbers $u_1 ,\ldots ,u_{n-1}$ very close to, and on both sides of, each non-regular level.
In other words, the first non-regular level is sandwiched very close between $u_1$ and $u_2$, the next is between $u_3$ and $u_4$, and so on---so in particular $n=2k+1$ if there are $k$ non-regular levels.
If ``close'' is close enough, then we get decompositions 
\[
\Gamma[u_{2i+1},u_{2i+2}] =\Gamma^1_i\amalg \Gamma^2_i\amalg \ldots
\amalg \Gamma^{k_i}_i
\]
where each $\Gamma^j_i$, $1\leq j\leq k_i$, is prime or invertible.
For each $i$ we can choose a point $\delta_{2i+1} = \delta_{2i+2}$ to make this into a tensor decomposition, so that each layer $\Gamma[u_{2i+1},u_{2i+2}]$ is elementary.
Moreover, each layer $\Gamma[u_{2i},u_{2i+1}]$ is invertible by definition.
An example of the resulting decomposition is shown in \autoref{fig:regularlevels}.
\end{proof}

\begin{figure}
  \centering
  \begin{tikzpicture}[label distance=-2pt]
    \begin{scope}[semithick]
      \bgcylinder{-2.5,0}{4}{2.5}{.4}{}{}
      \node[vert] (f) at (0,2) {};
      \node[vert] (g) at (1,.4) {};
      \node[vert] (h) at (-1,.4) {};
      \node[vert] (k) at (.5,3) {};
      \draw (f) -- (k);
      \draw (k) -- (top);
      \draw (k) -- ($(top)+(1,0)$);
      \drawtheta{f}{.1}{theta}
      \draw (f) -- (g);
      \draw (f) -- (h);
      \draw (h) to[out=90,in=-90] ($(top)+(-.3,0)$);
      \draw (g) -- ($(bot)+(.3,0)$);
      \draw (thetaR) to[out=-90,in=90,looseness=.5] (g);
      \draw (thetaL) to[out=90,in=-90,looseness=1] ($(top)+(-1.5,0)$);
    \end{scope}
    \path (ur) node[label={right:$u_0=a$}] {};
    \path (dr) node[label={right:$u_7=b$}] {};
    \path (ul) arc (-180:-60:2.5 and .4) node[vert,fill=black,label=above:$\delta_a$] {};
    \path (dl) arc (-180:-140:2.5 and .4) node[vert,fill=black,label=below:$\delta_b$] {};
    \begin{scope}[color=red!90!black,semithick]
      \draw (dl |- k) ++(0,.55) arc (-180:0:2.5 and .4) node[label=right:$u_1$] {};
      \draw (dl |- k) ++(0,.25) arc (-180:0:2.5 and .4) node[label=right:$u_2$] {};
      \draw (dl |- f) ++(0,.55) arc (-180:0:2.5 and .4) node[label=right:$u_3$] {};
      \draw (dl |- f) ++(0,.25) arc (-180:0:2.5 and .4) node[label=right:$u_4$] {};
      \draw (dl |- g) ++(0,.55) arc (-180:0:2.5 and .4) node[label=right:$u_5$] {};
      \draw (dl |- g) ++(0,.25) arc (-180:0:2.5 and .4) node[label=right:$u_6$] {};
    \end{scope}
    \begin{scope}[every node/.style={color=red!90!black,vert,fill=red!90!black}]
      \path (dl |- k) ++(0,.55) arc (-180:-120:2.5 and .4) node[label=above:$\delta_1$] {};
      \path (dl |- k) ++(0,.25) arc (-180:-120:2.5 and .4) node[label=below:$\delta_2$] {};
      \path (dl |- f) ++(0,.55) arc (-180:-40:2.5 and .4) node[label=above:$\delta_3$] {};
      \path (dl |- f) ++(0,.25) arc (-180:-40:2.5 and .4) node[label=below:$\delta_4$] {};
      \path (dl |- g) ++(0,.55) arc (-180:-90:2.5 and .4) node[label=above:$\delta_5$] {};
      \path (dl |- g) ++(0,.25) arc (-180:-90:2.5 and .4) node[label=below:$\delta_6$] {};
    \end{scope}
  \end{tikzpicture}
  \caption{A decomposition into elementary and invertible layers}
  \label{fig:regularlevels}
\end{figure}
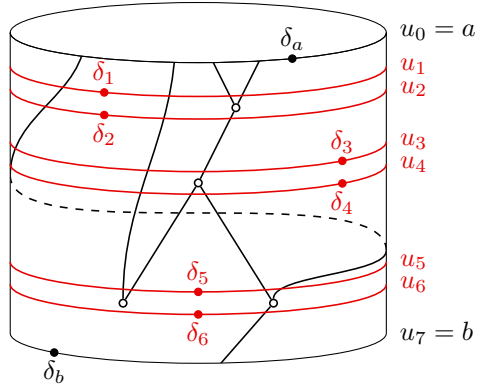

\subsection{Valuations of progressive cylinder graphs}

Recall that for a progressive cylinder graph $\Gamma$, we denote  its set of vertices by $\Gamma_0$.
We will denote its set of edges by $\Gamma_1$, and the set of connected components of $(S^1\times[a,b])\setminus \Gamma$ by $\Gamma_2$.
For each edge $e\in\Gamma_1$, we can identify which of these components lie to its right and left.
More formally, we define two functions, $\sR$ and $\sL$, from $\Gamma_1$ to $\Gamma_2$ as follows.
For each $e\in\Gamma_1$, choose some $x\in e$.
Then there is a point $y\in S^1\times \{\proj{2}(x)\}$ so that, for all edges $e'$ such that $e'\cap (S^1\times \proj{2}(x))$ is nonempty, we have
\[\proj{1}(x)<\proj{1}(y)<\proj{1}(e'\cap (S^1\times \proj{2}(x)))\]
in the cyclic order induced from the orientation on $S^1$.
Then we define $\sR(e)$ to be the component of $(S^1\times[a,b])\setminus \Gamma$ that contains $y$.
This function is well defined (and in particular independent of the choices of $x$ and $y$).
The function $\sL$ is defined similarly to identify the component to the left.

\begin{defn}
  A \textbf{valuation} $v\colon\Gamma \rightarrow \sB$ of a progressive cylinder graph $\Gamma$ in a bicategory $\sB$ consists of
\begin{itemize}

\item A function $v_2\colon \Gamma_2\rightarrow \mathrm{obj}(\sB)$.

\item For each edge $e\in\Gamma_1$, a 1-cell
\[v_1(e)\colon v_2(\sL(e))\hto v_2(\sR(e))\]

\item For each inner vertex $x\in \Gamma_0 \setminus \partial\Gamma$, a 2-cell
\[v_0(x)\colon v_1(\sigma_1)\odot \cdots \odot v_1(\sigma_m)
\rightarrow v_1(\tau_1)\odot \cdots \odot v_1(\tau_k)\]
where $\sigma_1 < \cdots <\sigma_m$ and $\tau_1<\cdots < \tau_k$ are the ordered lists of elements of $\mathrm{in}(x)$ and $\mathrm{out}(x)$, respectively.

\end{itemize}
The pair $(\Gamma, v)$ is called a \textbf{(progressive cylinder) diagram in} $\sB$, and is denoted merely by $\Gamma$ when the context is clear.
\end{defn}

We observe that the definitions of the order on $\mathrm{in}(x)$ and $\mathrm{out}(x)$ and of $\sL$ and $\sR$ imply that the composites $\odot$ in the source and target of $v_0(x)$ are defined.

If $c < d$ are regular levels for a diagram $\Gamma= (\Gamma, v)$, a valuation $v$ ``restricts'' in an obvious way to a valuation on the layer $\Gamma[c, d]$, which we also denote by $v$.
Similarly, if $\Gamma=\sh{\Gamma^1\odot \Gamma^2}$, then $v$ restricts to valuations on $\Gamma^1$ and $\Gamma^2$ again denoted by $v$.

It should hopefully be clear that the pictures we have drawn in \S\S\ref{sec:shadows}, \ref{sec:bicat-traces}, and \ref{sec:prop-bicat} can be made precise as progressive cylinder diagrams equipped with valuations in some bicategory.
We have denoted the function $v_2$ by assigning colors to connected components and the functions $v_1$ and $v_0$ by adding labels.

Our next goal is to define the \emph{value} of such a labeled diagram.
We begin with elementary diagrams, then invertible diagrams, and finally put these together to handle arbitrary diagrams.

\subsection{Value of an elementary diagram}

Suppose that $\Gamma$ is elementary, with a given tensor decomposition $\Gamma= \sh{\Gamma^1 \odot \ldots\odot \Gamma^n}$ in which each $\Gamma^i$ is either prime or invertible, and that \sB\ is equipped with a shadow functor with target category \bT.
We will also henceforth assume, for simplicity, that \sB\ is a strict 2-category.
There is no loss of generality in this, since every bicategory is equivalent to a strict 2-category with the same set of objects (0-cells), and a shadow functor can easily be transported across such an equivalence.
Then the \textbf{value} of $\Gamma$ is
\[v(\Gamma)=\bigsh{w(\Gamma^1)\odot \ldots \odot w(\Gamma^n)},\]
where
\[w(\Gamma^i)=
\begin{cases}
  v_0(x) &
  \text{if $\Gamma^i$ is prime with inner vertex $x$,}\\
  \id_{v_1(e_1)\odot \ldots \odot v_1(e_n)} &
  \text{if $\Gamma^i$ is invertible containing edges } e_1 < e_2 < \ldots < e_n.
\end{cases}
\]
In the second case, the order referred to is the linear order on $S^1\setminus \{\delta\}$, with $\delta$ the common domain and basepoint of $\Gamma$.
Strictly speaking we should consider the order of the points $e_j\cap (S^1 \times \{u\})$ for some $u\in [a,b]$, but continuity and progressivity imply that the resulting order is independent of the choice of $u$.
The assumed strictness of \sB\ ensures that no choices are involved in defining an $n$-ary composite of 1-cells in \sB.

Note that each $w(\Gamma^i)$ is a 2-cell in \sB, whereas $v(\Gamma)$ is a morphism in \bT.
The domain of $v(\Gamma)$ is $\sh{v_1(\sigma_1)\odot \dots \odot v_1(\sigma_n)}$ and its codomain is $\sh{v_1(\tau_1)\odot \dots \odot v_1(\tau_m)}$, where $\sigma_1 < \cdots <\sigma_n$ and $\tau_1<\cdots < \tau_m$ are the ordered lists of elements of $\mathrm{dom}(\Gamma)$ and $\mathrm{cod}(\Gamma)$, respectively.

\begin{lem}\label{lem:tensdecomp1}
Any two tensor decompositions of an elementary diagram $\Gamma$ have the same value.
\end{lem}
\begin{proof}
Since prime diagrams are tensor indecomposable and cannot occur inside any other prime or invertible diagram, any prime diagram occurring in one tensor decomposition of a given $\Gamma$ must occur in all others.
Thus, the only way tensor decompositions can differ comes from the fact that the tensor of two invertible diagrams is again invertible (and conversely, if $\Omega^1\odot \Omega^2$ is invertible then so must $\Omega^1$ and $\Omega^2$ be).
Since in this case the list of edges of $\Omega^1\odot \Omega^2$ is the ordered sum of those in $\Omega^1$ and $\Omega^2$, and therefore $w(\Omega^1\odot\Omega^2) = w(\Omega^1)\odot w(\Omega^2)$, such differences make no change in the value.
\end{proof}

\begin{prop}\label{refinevaluation} If $u$ is a regular level for an elementary diagram $\Gamma$
between levels $a$ and $b$ then $\Gamma[a, u]$, $\Gamma[u, b]$ are elementary (with $\delta_u$ defined to be the common domain and codomain basepoint $\delta$ of $\Gamma$), and
\[v(\Gamma) = v(\Gamma[u, b])\circ v(\Gamma[a, u]).\]
\end{prop}

\begin{proof}
First note that if $\Omega$ is an invertible progressive cylinder diagram between $a$ and $b$, and $u$ is a regular level for $\Omega$, then both $\Omega[a,u]$ and $\Omega[u,b]$ are invertible.
If instead $\Omega$ is prime, then one of $\Omega[a,u]$ and $\Omega[u,b]$ is prime and the other is invertible.
In either case, we have $w(\Omega)=w(\Omega[u,b])\circ w(\Omega[a,u])$.

Now suppose $\Gamma$ has tensor decomposition $\sh{\Gamma^1\odot
\Gamma^2\odot \ldots \Gamma ^n}$. Then $\Gamma[a,u]$
has tensor decomposition 
\[\bigsh{\Gamma^1[a,u]\odot \Gamma^2[a,u]\odot \cdots \odot \Gamma ^n[a,u]}\]
and $\Gamma[u,b]$ has tensor decomposition
\[\bigsh{\Gamma^1[u,b]\odot \Gamma^2[u,b]\odot \cdots\odot \Gamma ^n[u,b]}.\]

Using the definition of the valuation, functoriality of
the shadow, functoriality of $\odot$, and the
observation above, we have
\begin{align*}
  v(\Gamma[u,b])&\circ v(\Gamma[a,u])\\
  &=\; \Bigsh{w(\Gamma^1[u,b])\odot\cdots\odot w(\Gamma ^n[u,b])} \;\circ\;
  \Bigsh{w(\Gamma^1[a,u])\odot\cdots\odot w(\Gamma ^n[a,u])}\\
  &=\; \Bigsh{\big(w(\Gamma^1[u,b])\odot\cdots\odot w(\Gamma ^n[u,b])\big)\circ
    \big(w(\Gamma^1[a,u])\odot\cdots\odot w(\Gamma ^n[a,u])\big)}\\
  &=\; \Bigsh{\big(w(\Gamma^1[u,b])\circ w(\Gamma^1[a,u])\big)
    \odot\cdots\odot \big(w(\Gamma ^n[u,b])\circ w(\Gamma ^n[a,u])\big)}\\
  &=\; \bigsh{w( \Gamma^1)\odot \cdots \odot w(\Gamma ^n)}\\
  &=\; v(\Gamma).\qedhere
\end{align*}
\end{proof}

\subsection{Value of an invertible diagram}

We also need to define directly the value of an invertible diagram (not necessarily elementary).
First we introduce some notation.
Suppose $M_1,\dots,M_{n}$ are ``cyclically composable'' 1-cells, in the sense that $M_k\odot M_{(k+1)\operatorname{mod}n}$ is defined for all $1\le k\le n$.
Then for any $0\le k\le n$, we write
\begin{multline*}
  \theta^n_k \coloneqq \theta_{M_1\odot\dots\odot M_{k \operatorname{mod} n} \;,\; M_{(k+1) \operatorname{mod} n}\odot\dots\odot M_{n} }
  \;\colon\\
  \Bigsh{M_1 \odot\dots\odot M_{n}} \;\too[\iso]\;
  \Bigsh{M_{(k+1) \operatorname{mod} n}
  \odot\dots\odot M_{n} \odot M_1 \odot\cdots\odot
  M_{k \operatorname{mod} n}}.
\end{multline*}
Note that $\theta^n_0 = \theta^n_n = \id$ by the unit axioms for $\theta$.
We claim that also
\begin{equation}
  \theta^n_{k}\circ \theta^n_{m} = \theta^n_{(k+m) \operatorname{mod} n}\label{eq:theta-add}
\end{equation}
for any $k,m \in \{0,\dots,n\}$.
If $k+m \le n$, this follows immediately from the hexagon axiom for $\theta$, while if $k+m > n$, we can write
\[
\theta^n_k \circ \theta^n_m = \theta^n_k \circ \theta^n_{n-k} \circ \theta^n_{m+k-n}
= \theta^n_n \circ \theta^n_{m+k-n} = \theta^n_{(m+k)\operatorname{mod} n}.
\]
Therefore, from now on we implicitly interpret all subscripts to $\theta^n$ as taken $\operatorname{mod} n$.
Naturality of $\theta$ also immediately implies that for any 2-cells $f_i\colon M_i\to N_i$, we have
\begin{equation}\label{eq:theta-nat}
\theta^n_k \,\circ\, \Bigsh{f_1\odot\cdots\odot f_{n}} \;=\;
\Bigsh{f_{(k+1)\operatorname{mod} n}\odot\cdots\odot f_n \odot f_0 \odot\cdots\odot f_{k}} \,\circ\, \theta^n_k
\end{equation}
Finally, given $n$ lists $M_{k,1},\dots,M_{k,m_k}$ for $1\le k\le n$ such that the concatenated list $M_{1,1},\dots,M_{n,m_n}$ is cyclically composable, for any $0\le k< n$ we have
\begin{equation}\label{eq:theta-combine}
\theta^n_k = \theta^{m_1+\dots+m_n}_{m_1+\dots+m_{k}}
\;\colon
\bigsh{M_{1,1} \odot\cdots \odot M_{n,m_n}} \too
  \bigsh{M_{k+1,1}\odot\cdots\odot M_{k,m_k}}.
\end{equation}
This follows easily from repeated application of the hexagon axiom.

Now, if $\Gamma$ is invertible from $a$ to $b$, then we have bijections
\[\mathrm{dom}(\Gamma) \cong \pi_0(\Gamma)\cong \mathrm{cod}(\Gamma)\]
between the domain, connected components, and codomain, and the composite preserves cyclic order.
Therefore, the ordered lists $\mathrm{dom}(\Gamma)$ and $\mathrm{cod}(\Gamma)$ differ by a cyclic shift of $k$ for some $0\le k<n$, so we can define $v(\Gamma) = \theta^n_k$.

\begin{prop}\label{refineval-inv}
If $u$ is a regular level for an invertible diagram $\Gamma$ between levels $a$ and $b$, then for any valid choice of $\delta_u$, the diagrams $\Gamma[a, u]$ and $\Gamma[u, b]$ are invertible and
\[v(\Gamma) = v(\Gamma[u, b])\circ v(\Gamma[a, u]).\]  
\end{prop}
\begin{proof}
  If $\Gamma[a,u]$ and $\Gamma[u,b]$ induce cyclic shifts of $k$ and $m$, respectively, then $\Gamma$ must induce one of $k+m$.
  Thus, the proposition follows from~\eqref{eq:theta-add}.
\end{proof}

\subsection{Value of a progressive cylinder diagram}

We are now ready to define the value of an arbitrary progressive cylinder diagram.
Let $\Gamma$ be such a diagram between levels $a$ and $b$; its value is
\[v(\Gamma)=v(\Gamma[u_{n-1},u_n])\circ \cdots \circ v(\Gamma[u_0,u_1]),\]
where $a = u_0 < u_1 < \cdots < u_n = b$ are regular levels for $\Gamma$ such that each
layer $\Gamma[u_{i-1},u_i]$ is elementary or invertible, and we have chosen valid basepoints $\delta_i$ (for $0<i<n$) disjoint from $\Gamma\cap S^1\times\{u_i\}$.
Such a decomposition of $\Gamma$ exists by \autoref{regularlevels}.

\begin{prop}
  The value of a progressive cylinder diagram, as defined above, is well-defined, i.e.\ it is independent of the choice of levels $u_i$ and basepoints $\delta_i$.
\end{prop}
\begin{proof}
  By Propositions \ref{refinevaluation} and \ref{refineval-inv}, the value is invariant under adding more regular levels (with arbitrary valid basepoints) to the list $(u_i)$.
  Moreover, if a diagram is both elementary and invertible, then the two ways of defining its value agree (both are an identity 2-cell).
  Thus, given any two choices of levels and basepoints, we can add new levels to each, without changing their values, until they have the same regular levels.
  However, some interior levels might still be equipped with two different choices of basepoint.

  Now we add further new levels $u_i^-$ and $u_i^+$ immediately below and above each $u_i$, for $0<i<n$, with the same basepoints as $u_i$.
  If we choose these new levels close enough to $u_i$, then the layers $\Gamma[u_i^-,u_i]$ and $\Gamma[u_i,u_i^+]$ are both elementary and invertible.
  By two applications of \autoref{refineval-inv}, we can then arbitrarily change the basepoint of the level $u_i$ (subject to disjointness from $\Gamma$) without changing the assigned value.
  Thus, by considering the layers $\Gamma[u_i,u_{i+1}]$ one at a time, we may assume that $n=3$, that $\Gamma[u_0,u_1]$ and $\Gamma[u_2,u_3]$ are invertible, and that for one of the choices of basepoints $v(\Gamma[u_0,u_1])$ and $v(\Gamma[u_2,u_3])$ are identities.

  If $\Gamma[u_1,u_2]$ is also invertible, we are done by \autoref{refineval-inv}, so assume it is elementary and not invertible.
  Write $\delta$ for the basepoint assigned to the level $u_1$ which makes $v(\Gamma[u_0,u_1])$ and $v(\Gamma[u_2,u_3])$ identities, and $\delta'$ for the other basepoint assigned to $u_1$.
  (Since $\Gamma[u_1,u_2]$ is elementary, these are also the basepoints assigned to the level $u_2$.)
  We will write $v$ and $v'$ for the values determined by these two choices of basepoints.
  Let $\sh{\Gamma^1\odot\cdots\odot \Gamma^m}$ be a tensor decomposition of $\Gamma[u_1,u_2]$ which is valid for both $\delta$ and $\delta'$.
  (Such exists, since any two tensor decompositions have a common refinement.)
  Then by assumption
  \[v(\Gamma) =\; \bigsh{w(\Gamma^1)\odot\cdots\odot w(\Gamma^m)},\]
  while
  \[v'(\Gamma) = v'\big(\Gamma[u_2,u_3]\big) \;\circ\;
  \Bigsh{w(\Gamma^{k+1})\odot\cdots\odot w(\Gamma^n)\odot w(\Gamma^0)\odot\cdots\odot w(\Gamma^k)}
  \;\circ\; v'\big(\Gamma[u_0,u_1]\big).
  \]
  for some $k\in \{0,\dots,n-1\}$.

  Now, if there are $p$ edges meeting level $u_0$ and (equivalently) level $u_1$, then we have $v'(\Gamma[u_0,u_1]) = \theta^p_q$ for some $q\in \{0,\dots,p-1\}$.
  But if there are $m_i$ edges in the source of $\Gamma^i$ for each $i$, then we have $p=m_1+\cdots+m_n$, and moreover $q = m_1+\dots+m_k$ by placement of the basepoints.
  Therefore,~\eqref{eq:theta-combine} implies that $v'(\Gamma[u_0,u_1]) = \theta^p_q = \theta^n_k$.
  By a similar argument, $v'(\Gamma[u_2,u_3]) = \theta^n_{n-k}$, so by~\eqref{eq:theta-nat} we have
  \begin{align*}
    v'(\Gamma)
    &= \theta^n_{n-k} \;\circ\;
    \Bigsh{w(\Gamma^{k+1})\odot\cdots\odot w(\Gamma^n)\odot w(\Gamma^0)\odot\cdots\odot w(\Gamma^k)}
    \;\circ\; \theta^n_k\\
    &= \bigsh{w(\Gamma^1)\odot\cdots\odot w(\Gamma^m)} \;\circ \theta^n_{n-k} \circ \theta^n_k\\
    &= \bigsh{w(\Gamma^1)\odot\cdots\odot w(\Gamma^m)}\\
    &= v(\Gamma).\qedhere
  \end{align*}
\end{proof}

\subsection{Deformations of progressive cylinder diagrams}

Finally, we prove that the value of a diagram is invariant under deformations.

\begin{defn}
  A \textbf{deformation of progressive cylinder graphs} consists of an embedding $\Gamma\hookrightarrow S^1\times [a,b]$ together with continuous maps
  \begin{gather*}
    H\colon S^1\times [a,b] \times [0,1] \too S^1\times [a,b]\\
    \delta_a\colon [0,1] \too S^1\\
    \delta_b\colon [0,1] \too S^1
  \end{gather*}

  such that
  \begin{itemize}
  \item For each $t\in[0,1]$, the map $H_t = H(-,-,t) \colon S^1\times [a,b] \to S^1\times [a,b]$ is a homeomorphism, which restricts to a self-homeomorphism of the boundary components $S^1\times\{a\}$ and $S^1\times\{b\}$;
  \item $H_0$ is the identity of $S^1\times[a,b]$; and
  \item for each $t\in [0,1]$, the composite
    \[ \Gamma \hookrightarrow S^1\times[a,b] \xto{H_t} S^1\times[a,b] \]
    is a progressive cylinder graph with basepoints $\delta_a(t)$ and $\delta_b(t)$.
  \end{itemize}
\end{defn}

For $t\in [0,1]$ we denote the above composite by $\Gamma\res{t}$.
Since $H_t$ is a homeomorphism, it induces a bijection between connected components of the complements of $\Gamma\res{0}$ and $\Gamma\res{t}$; in other words we have 
$(\Gamma\res{0})_2 \cong (\Gamma\res{t})_2$.
Of course, we also have $(\Gamma\res{0})_1 \cong (\Gamma\res{t})_1$ and 
$(\Gamma\res{0})_0 \cong (\Gamma\res{t})_0$, since the underlying topological graph $\Gamma$ is the same for all $t$.
Therefore, a valuation $v$ on $\Gamma\res{0}$ can be transported canonically along $H$ to give a valuation on $\Gamma\res{t}$ for all $t$, and in particular on $\Gamma\res{1}$ (the ``target'' of the deformation).  We call this valuation on $\Gamma\res{t}$ the valuation 
\textbf{induced} from 
the valuation on $\Gamma\res{0}$ by $H$ and denote it by $v_H$.

\begin{thm}If $(H,\delta_a,\delta_b)$ 
 is a deformation of progressive cylinder graphs and $v$ is a valuation on $\Gamma\res{0}$,
then
\[v(\Gamma\res{0})=v_H(\Gamma\res{1}).\]
\end{thm}

\begin{proof}
For each
$0\leq t_0\leq 1$ we will find a connected relatively-open subset $V_{t_0}$ of $[0,1]$,
 containing $t_0$, such that for all $t\in V_{t_0}$ we have
\[v_{H}(\Gamma\res{t}) = v_{H}(\Gamma\res{t_0}).\]
Since $[0, 1]$ is connected, this will show that $v_{H}(\Gamma\res{t})$ is constant in $t$.

First we use \autoref{regularlevels} to choose regular levels
\[a=u_0<u_1<\ldots <u_n=b\]
for $\Gamma\res{t_0}$, and basepoints $\delta_i\in S^1$ for $0< i< n$, such that the layers $\Gamma\res{t_0}[u_{0},u_{1}]$ and $\Gamma\res{t_0}[u_{n-1},u_{n}]$ are invertible, and the remaining layers $\Gamma\res{t_0}[u_{i},u_{i+1}]$ are either elementary or invertible.
Moreover, for each $i$ such that $\Gamma\res{t_0}[u_i,u_{i+1}]$ is elementary, we choose a tensor decomposition
\[\bigsh{\Gamma^1_i \odot \Gamma^2_i \odot \cdots \odot \Gamma^{k_i}_i}\]
of it such that each $\Gamma^j_i$ is prime or invertible, and we also choose pairwise disjoint, connected, open subsets $U^j_i \subset S^1$ such that $\Gamma^j_i\subset U^j_i\times [u_i,u_{i+1}]$ for all $i$ (which exist by the definition of a tensor decomposition).
Moreover, whenever $\Gamma^j_i$ is prime, we write $x^j_i$ for its unique inner vertex.

Now, we observe that the following conditions are all \emph{open} conditions as functions of $t$, and that there are finitely many of them.
\begin{enumerate}
\item Each $u_i$ is a regular level for $\Gamma\res{t}$.
\item $(\delta_i,u_i)\notin\Gamma\res{t}$ whenever $0<i<n$.
\item $(\delta_a(t),a)\notin \Gamma\res{t}$ and $(\delta_b(t),b)\notin \Gamma\res{t}$.
\item For each $i$ such that $\Gamma\res{t_0}[u_i,u_{i+1}]$ is elementary, we have
  \[\Gamma\res{t}[u_i,u_{i+1}]  \subset \bigcup_j U^j_i.\]
\end{enumerate}
Therefore, we can choose an open subset $V_{t_0}$ of $[0,1]$, containing $t_0$, such that all of these conditions are satisfied for all $t\in V_{t_0}$.
This implies the following.
\begin{itemize}
\item No inner vertex can cross one of the levels $u_i$ during $V_{t_0}$.
  Therefore, any invertible layer of $\Gamma\res{t_0}$ remains invertible in $\Gamma\res{t}$ for $t\in V_{t_0}$.
\item No edge can cross over a basepoint $(\delta_i,u_i)$ (including $(\delta_a(t),a)$ and $(\delta_b(t),b)$) during $V_{t_0}$.
  Therefore, the linear orders of the domain and codomain of each invertible layer remain the same throughout $V_{t_0}$.
  This means that the cyclic shift it induces also remains constant, and hence so does its value.
\item Inside a layer which is elementary in $\Gamma\res{t_0}$, no edge or inner vertex can cross from one $U^j_i$ to another during $V_{t_0}$, since the $U^j_i$ are pairwise disjoint for $1\le j \le k_i$.
  Therefore, if $\Gamma\res{t_0}[u_i,u_{i+1}] \cap U^j_i$ is prime or invertible, it remains so in $\Gamma\res{t}$ for $t\in V_{t_0}$, and its value also remains the same.
\item If $\Gamma\res{t_0}[u_i,u_{i+1}]$ is elementary, then $0<i<(n-1)$, and hence the (equal) basepoints $\delta_i$ and $\delta_{i+1}$ do not change with $t$.  Thus the basepoints of $\Gamma\res{t}[u_i,u_{i+1}]$ are also equal.
\item Therefore, if $\Gamma\res{t_0}[u_i,u_{i+1}]$ is elementary, then so is $\Gamma\res{t}[u_i,u_{i+1}]$ for $t\in V_{t_0}$, and its value is also constant.
\end{itemize}
Finally, since the value of $\Gamma\res{t}$ is the composite of those of its layers, this value is constant over $V_{t_0}$, as desired.
\end{proof}

\begin{rmk}
  Just as the main theorem about string diagrams for monoidal categories is exploited in~\cite{js:geom-tenscalc-i} to give a presentation of the free monoidal category on a tensor scheme, our theorem about string diagrams for bicategories with shadows can be used to give a presentation of the free bicategory-with-shadow on a computad (see~\cite{street:catstruct,street:catval-2lim} for the notion of computad).
  However, we will not pursue this direction here.
\end{rmk}




\bibliographystyle{plain}
\bibliography{traces_2}

\end{document}